\documentclass[11pt]{article}

\usepackage[a4paper,margin=1in]{geometry}
\usepackage{amsmath}
\usepackage{amssymb}
\usepackage{booktabs}
\usepackage{array}
\usepackage{multirow}
\usepackage{bm}
\usepackage{mathtools}
\usepackage{stmaryrd}
\usepackage{enumitem}
\usepackage{xcolor}
\usepackage{graphicx}
\usepackage{amsthm}
\usepackage[colorlinks=true,linkcolor=blue,citecolor=blue,urlcolor=blue]{hyperref}
\usepackage[nameinlink,capitalize,noabbrev]{cleveref}

\theoremstyle{plain}
\newtheorem{theorem}{Theorem}
\newtheorem{lemma}{Lemma}
\newtheorem{corollary}{Corollary}
\newtheorem{proposition}{Proposition}
\theoremstyle{definition}
\newtheorem{definition}{Definition}
\newtheorem{assumption}{Assumption}
\theoremstyle{remark}
\newtheorem{remark}{Remark}

\newcommand{\email}[1]{\href{mailto:#1}{#1}}
\newenvironment{keywords}{\par\medskip\noindent\textbf{Keywords.}}{\par}
\newenvironment{AMS}{\par\medskip\noindent\textbf{MSC codes.}}{\par}

\crefname{theorem}{Theorem}{Theorems}
\Crefname{theorem}{Theorem}{Theorems}
\crefname{lemma}{Lemma}{Lemmas}
\Crefname{lemma}{Lemma}{Lemmas}
\crefname{corollary}{Corollary}{Corollaries}
\Crefname{corollary}{Corollary}{Corollaries}
\crefname{proposition}{Proposition}{Propositions}
\Crefname{proposition}{Proposition}{Propositions}
\crefname{definition}{Definition}{Definitions}
\Crefname{definition}{Definition}{Definitions}
\crefname{assumption}{Assumption}{Assumptions}
\Crefname{assumption}{Assumption}{Assumptions}
\crefname{remark}{Remark}{Remarks}
\Crefname{remark}{Remark}{Remarks}
\crefname{example}{Example}{Examples}
\Crefname{example}{Example}{Examples}

\newcommand{\R}{\mathbb{R}}

\newcommand{\dx}{\,dx}
\newcommand{\ds}{\,ds}
\newcommand{\ECR}{\mathrm{ECR}}

\newcommand{\CG}{\mathrm{CG}}

\DeclareMathOperator{\diam}{diam}

\newcommand{\jump}[1]{\llbracket #1 \rrbracket}

\title{Explicit Two-Sided Eigenvalue Bounds for Schr\"{o}dinger
       Operators with Singular Potentials via Finite Element Method}

\author{%
  Xuefeng Liu%
  \thanks{Tokyo Woman's Christian University (\email{xfliu@lab.twcu.ac.jp}).}%
}

\begin{document}
\maketitle

\begin{abstract}
We present, to the best of our knowledge, the first numerical
algorithm for \emph{explicit, computable} two-sided eigenvalue
bounds for Schr\"odinger operators $H = -\Delta + V$ on $\mathbb{R}^N$
($N=2,3$) in the presence of both an \emph{unbounded potential}
and an \emph{unbounded domain}. ``Explicit'' here means that all
constants and ingredients are derived in closed form from the mesh,
the potential, and a small set of explicit inequalities
(Payne--Weinberger, Hardy, and explicit bounded-domain Sobolev
embeddings); the conversion to fully \emph{verified}
(IEEE-754-safe, interval-arithmetic) enclosures is a separate
verification step and is left for future work.
In particular, singular attractive potentials of Coulomb type,
$V(x)=-Z/|x|$, which model the hydrogen atom and the $\mathrm{H}_2^+$
molecular ion, are covered by the theory.
The method combines domain truncation to a bounded domain $D(R)\supset\{|x|\le R\}$
with an extension of Liu's Composite Enriched Crouzeix--Raviart
(CECR) finite element method to sign-indefinite potentials. Upper
bounds come from the standard conforming Galerkin method; lower
bounds come from the CECR construction, whose gap to the exact
eigenvalue closes as the mesh is refined.
Numerical experiments on the 2D single- and two-centred Coulomb
potentials and on the 3D hydrogen atom and $\mathrm{H}_2^+$
molecular ion illustrate the algorithm and confirm the predicted
convergence.
\end{abstract}

\begin{keywords}
explicit eigenvalue bounds, Schr\"odinger operator, Coulomb singularity,
CECR finite element method, Crouzeix--Raviart element,
fixed shift, form-norm CECR constant, two-sided enclosure, singular potential
\end{keywords}

\begin{AMS}
65N25, 65N30, 35P15, 35J10, 81Q05
\end{AMS}

\section{Introduction}
\label{sec:intro}

Computing explicit two-sided bounds on eigenvalues
of the Schr\"odinger operator
\begin{equation}
\label{eq:schrodinger}
  H = -\Delta + V, \qquad \text{on } \R^N,\; N=2,3,
\end{equation}
arises in quantum chemistry, mathematical physics, and the numerical
analysis of partial differential equations.  Upper bounds on the $k$-th
eigenvalue $\lambda_k$ are readily obtained by the Rayleigh--Ritz
(Galerkin) method: any conforming trial space $V_h \subset H^1_0(\Omega)$
yields $\lambda_k \le \lambda_{k,h}^{\CG}$.  Obtaining a
\emph{explicit lower bound} requires non-conforming or specialised
methods, since the variational principle provides only upper bounds.
This paper develops such lower bounds for the physically important
class of Coulomb-type potentials.

\subsection*{Background and Related Work}

\paragraph{Lower bound methods}
Classical methods for eigenvalue lower bounds include
the Lehmann--Maehly--Goerisch variational method
\cite{Goerisch1986}, the Kato--Temple inequality,
Behnke's complementary variational approach \cite{Behnke2000},
and Plum's computer-assisted enclosure method \cite{Plum1991}.
In the regime of obtaining explicit eigenvalue bounds using finite
element methods, Liu \cite{LiuOis2011,LiuOis2013,Liu2015}
introduced the projection-based eigenvalue bound as follows:
\begin{equation}
\label{eq:cecr-lb-intro}
  \frac{\lambda_{k,h}}{1 + \lambda_{k,h} C_h^2} \;\le\; \lambda_k,
\end{equation}
where $\lambda_{k,h}$ is a finite element approximation to  eigenvalue and
$C_h = O(h^\alpha)$ is an explicit error constant.
Related work on explicit eigenvalue bounds was initiated
by the early work of Kikuchi and Liu \cite{KikuchiLiu2007,LiuKikuchi}
and Kobayashi \cite{Kobayashi2011,Kobayashi2015} on explicit
interpolation error constants.  See also later work including
Carstensen et al.\
\cite{CarstensenGedicke2014,CarstensenGallistl2014},
Hu, Huang and Ma \cite{HuHuangMa2016},
Canc\`es et al.\ \cite{CancesDusson2017},
and You, Xie and Liu \cite{YouXieLiu2019} (Steklov eigenvalues).
To the author’s knowledge, no existing finite element framework provides explicit and convergent lower bounds for eigenvalues associated with such unbounded singular potentials. This constitutes the central difficulty addressed in the present work, which is overcome by employing the Composite Enriched Crouzeix–Raviart (CECR) finite element method recently proposed in the author’s book \cite[\S 4.1.3]{liu2024guaranteed}.
.

\paragraph{The coercivity gap}
The existing CECR application assumes $V \ge 0$:
this ensures coercivity of the bilinear form
$a(u,u) = \|\nabla u\|^2 + (Vu,u) \ge \|\nabla u\|^2$,
a step used in the proof of \eqref{eq:cecr-lb-intro}.
Physically important potentials—Coulomb ($V=-Z/r$),
two-centered Coulomb ($V=-Z_1/|x-a_1|-Z_2/|x-a_2|$)
as in the $\mathrm{H}_2^+$ molecular ion—are purely attractive ($V<0$).
For these, $a(u,u)$ may be negative, and \eqref{eq:cecr-lb-intro}
breaks down.  A natural first attempt replaces $V$ by its piecewise
constant element average $c_h$ and shifts by
$\gamma_h=\|c_h^-\|_{L^\infty}$, but $\gamma_h$ diverges as
the mesh is refined for Coulomb singularities and the bound fails to converge
on uniform or radially graded meshes.
This paper develops a \emph{pair-space} framework and introduces
an \emph{elementwise reaction constant}
$\Gamma_h:=\max_K c_h^-|_K\cdot h_K^{\,2}/\pi^2$
that replaces the divergent $\gamma_h$ by a vanishing quantity
$\Gamma_h\to 0$ under refinement. The resulting CECR lower bound
(\Cref{thm:cecr-lb-convergent}) is fully computable; the
CECR gap closes at rate $O(h^2)$ in the mesh-family parameter on
graded meshes in both 2D and 3D.
With the patch radius chosen according to the direct estimates in
\Cref{app:graded-eps}, the final certified bound $L_k\le\lambda_k$
inherits the same $O(h^2)$ rate.
This provides the first explicit convergent eigenvalue lower bound
for Coulomb-type potentials via non-conforming finite elements.

\paragraph{Domain truncation}
For operators on $\R^N$, a standard strategy truncates to a bounded domain
$D(R)\supset\{|x|\le R\}$ with Dirichlet or Neumann boundary conditions.
Dirichlet truncation gives upper bounds \cite{Reed-Simon};
Neumann truncation gives lower bounds directly under an explicit
confinement condition on the exterior potential minimum
$\sigma_{\mathrm{ext}}(D(R)) := \inf_{\R^N\setminus D(R)} V > \lambda_k$,
without requiring the CECR machinery.
The truncation error decays exponentially in $R$ for potentials
with exponentially decaying eigenfunctions \cite{Agmon1982}.
We adopt both boundary conditions and compare.

\paragraph{Reference values for the
\texorpdfstring{$\mathrm{H}_2^+$}{H2+} molecular ion}
The electronic eigenvalue problem for $\mathrm{H}_2^+$ (clamped
nuclei, excluding proton--proton repulsion) has served as a
benchmark for quantum chemistry since the 1950s.
High-accuracy two-centre wavefunctions and eigenvalues were
obtained by Bates, Ledsham and Stewart \cite{Bates1953} via a
variational expansion in prolate spheroidal coordinates; Peek
\cite{Peek1965} tabulated the eigenparameters for the ground
($1s\sigma_g$) and first excited ($2p\sigma_u$) orbitals over
the dense grid of internuclear distances
$R = 1.0\,(0.5)\,30.0\,a_0$ (so $R=2\,a_0$ and its multiples
are included), and Madsen and Peek \cite{MadsenPeek1971}
extended such tabulations to the lowest twenty electronic
states.  More recently, Olivares-Pil\'on and Turbiner
\cite{OlivaresTurbiner2016} reported Lagrange-mesh computations
for the low-lying states that agree with alternative
high-precision calculations to about ten significant digits.
These classical and modern references furnish the reference
values against which we benchmark the explicit two-sided bounds
in \Cref{sec:numerics}. They do not, however, provide an
\emph{explicit} lower bound together with an a-priori convergence
rate, which is the gap addressed here.

\subsection*{Contributions of this Paper}

This paper establishes the following contributions.

\begin{enumerate}[label=(\roman*)]

\item \textbf{Pair-space framework.}
We lift the CECR analysis to a two-component pair space
$\widehat V$ on which the extended bilinear forms $\widehat a$,
$\widehat b$ encode the kinetic and reaction contributions
separately (\Cref{subsec:cecr-pair-space}).  An abstract lower-bound
theorem (\Cref{thm:cecr-abstract}) is proved in this framework with a
self-contained max-min argument.
\item \textbf{Fixed-shift two-sided bound.}
We introduce a single, mesh-independent shift $\sigma>0$ that
renders the shifted CECR form $a_h^\sigma$ positive definite on
$V_h^{\ECR}$, prove a perturbation
lemma relating $\lambda_k$ of $-\Delta+V$ to $\mu_k$ of $-\Delta+c_h$
without $L^\infty$ control of $c_h^-$ (\Cref{lem:perturb-sigma}),
and assemble a two-sided bound \Cref{thm:main}.

\item \textbf{Two-centered Coulomb potential.}
We specialize the theory to $V = -Z_1/|x-a_1| - Z_2/|x-a_2|$
(model for $\mathrm{H}_2^+$ and related systems) and provide explicit
formulas for the admissible shift $\sigma$ and the perturbation
parameter $\varepsilon_h$ in both 2D and 3D
(\Cref{sec:twocenter}).

\item \textbf{Dirichlet and Neumann truncation.}
We analyze both Dirichlet and Neumann boundary conditions on the
truncated domain $D(R)$.  Dirichlet gives an upper bound; Neumann
gives a direct (CECR-free) lower bound under the explicit
confinement condition $\sigma_{\mathrm{ext}}(D(R)) > \lambda_k$
(\Cref{lem:neumann-lb}), which is verifiable from a computed
upper bound on $\lambda_k$.

\end{enumerate}

\subsection*{Organization}

\Cref{sec:setting} introduces the operator, truncation domain, and
boundary conditions.
\Cref{sec:cecr} develops the CECR method for piecewise constant
potentials and proves the shift corollary.
\Cref{sec:extension} extends the theory to unbounded potentials.
\Cref{sec:twocenter} treats the two-center Coulomb potential.
\Cref{sec:numerics} presents numerical experiments.
\Cref{sec:conclusion} concludes.

\subsection*{Notation}

$D(R)$ denotes a bounded Lipschitz domain in $\R^N$ satisfying
$\{|x|\le R\}\subseteq D(R)$; we call $R$ the \emph{truncation radius}.
In the numerical experiments $D(R)$ is taken to be a box
(rectangle for $N=2$, rectangular cuboid for $N=3$) or a ball,
depending on the test problem.
$H^1_0(\Omega)$ is the Sobolev space with zero trace on $\partial\Omega$.
$L^p(\Omega)$ is equipped with the norm $\|\cdot\|_{L^p}$.
For a triangulation $\mathcal{T}^h$, $h_K = \diam(K)$ and
$h_{\max}:=\max_K h_K$.  The symbol $h$ in $\mathcal{T}^h$ and in the
numerical tables denotes the mesh-family parameter used to label a
refinement level; on graded meshes it need not equal $h_{\max}$.
$\lambda_1^D(\Omega)$ is the first Dirichlet eigenvalue of $-\Delta$
on $\Omega$.
$C \lesssim D$ means $C \le c\, D$ for a constant $c$ depending only on the domain $\Omega$ and the dimension $N$.

\section{Problem Setting and Domain Truncation}
\label{sec:setting}

\subsection{The Schrödinger Eigenvalue Problem}
\label{subsec:schrodinger}

Let $N \in \{2,3\}$ and let $V:\R^N\to\R$ be a measurable function
with $V = V^+ - V^-$ where $V^\pm := \max(\pm V,0) \ge 0$.
We consider the eigenvalue problem
\begin{equation}
\label{eq:evp-full}
  (-\Delta + V)\,u = \lambda\, u \quad \text{in } \R^N,
  \qquad u \in H^1(\R^N),
\end{equation}
in the sense of quadratic forms: find $\lambda \in \R$ and
$u \in H^1(\R^N) \setminus \{0\}$ such that
\begin{equation}
\label{eq:weak-full}
  \int_{\R^N} \nabla u \cdot \nabla v + V\,uv \dx
  = \lambda \int_{\R^N} uv \dx
  \quad \forall v \in H^1(\R^N).
\end{equation}

\begin{assumption}[Standing assumptions on $V$]
\label{ass:V}
Throughout the paper we assume:
\begin{enumerate}[label=(\roman*)]
  \item $V^+ \in L^1_{\rm loc}(\R^N)$;
  \item $V^- \in L^{p_0}_{\rm loc}(\R^N)$ for some $p_0 > N/2$,
        with $V^-(x)\to 0$ as $|x|\to\infty$ (no growth at infinity);
  \item The operator $H=-\Delta+V$ is self-adjoint and bounded below on
        $\R^N$ in the form-sum sense; this is guaranteed by the
        KLMN/Kato--Rellich framework whenever $V^-$ is locally
        $L^{p_0}$ with $p_0>N/2$ \emph{and} is form-bounded with
        relative bound $<1$, see
        \Cref{ass:form-bounded} and \cite[Theorem~X.17]{Reed-Simon}.
\end{enumerate}
\end{assumption}

\begin{remark}[Why local rather than global integrability]
\label{rmk:Lploc-vs-Lp}
The Coulomb tail $V^-(x)\sim Z/|x|$ is \emph{not} in $L^{p_0}(\R^N)$
for any single $p_0>N/2$: integrability at the origin requires
$p_0<N$, while integrability at infinity requires $p_0>N$, two
incompatible conditions. We therefore impose only local
$L^{p_0}$-integrability on $\R^N$, supplemented by decay at infinity;
this suffices for self-adjointness and discrete spectrum below
$\inf\sigma_{\rm ess}(H)$ for Coulomb-type $V$ via the form-bound
of \Cref{ass:form-bounded} (Hardy in $N=3$,
Gagliardo--Nirenberg in $N=2$). \emph{Global} $L^{p_0}$-integrability
is invoked only after truncation, where it reduces to
$V\in L^{p_0}(D(R))$ on the bounded domain $D(R)$.
\end{remark}

The Coulomb-type potentials treated in this paper satisfy
\Cref{ass:V}: the single-centre potential $V^-=Z/|x-a|$ belongs to
$L^{p_0}(D(R))$ for every $p_0<N$ (since
$\int_{D(R)}|x-a|^{-p_0}\dx<\infty$ when $p_0<N$), so in particular
for any $p_0\in(N/2,\,N)$; the two-centre potential
$V^-=Z_1/|x-a_1|+Z_2/|x-a_2|$ inherits the same local integrability.
Concretely, we take $p_0=4/3$ when $N=2$ and $p_0=3/2+\delta$
(any small $\delta>0$) when $N=3$. After truncation to $D(R)$,
$V\in L^{p_0}(D(R))$ holds and is what enters the perturbation
analysis of \Cref{sec:extension}.

Under \Cref{ass:V}, the operator $H$ has a discrete spectrum
$\lambda_1 \le \lambda_2 \le \cdots$ below
$\inf\sigma_{\rm ess}(H)$; see, e.g., \cite{persson1960bounds}.
We study the problem of computing explicit two-sided bounds on
$\lambda_k$ for $k = 1, 2, \ldots$.

\subsection{Truncation to a Bounded Domain}
\label{subsec:truncation}

Since $\R^N$ is unbounded, we truncate to a bounded domain
$D(R)\supset\{|x|\le R\}$ and impose boundary conditions on
$\partial D(R)$. We choose it large enough that the computed discrete eigenvalues
yield a valid \emph{lower bound for the ground state}
$\lambda_1$, i.e.\ so that (i) the confinement condition
$\sigma_{\mathrm{ext}}(D(R))>\lambda_1$ of \Cref{lem:neumann-lb}
holds (see \Cref{rmk:confinement-verify}), and (ii) the
truncation error $\delta_1(R)$ is negligible relative to the
discretisation error. Higher eigenvalues $\lambda_k$, $k\ge 2$,
are reported but may not be covered by the Neumann-truncation
lower-bound machinery at the chosen $R$.
The domains used in \S\ref{sec:numerics} are listed in
\Cref{tab:confinement-check}; they contain the relevant nuclei
and are large enough for the ground-state confinement check.

\subsubsection*{Dirichlet Truncation}

Let $\lambda_k^{R,D}$ denote the $k$-th eigenvalue of
$-\Delta + V$ on $D(R)$ with Dirichlet boundary condition
$u = 0$ on $\partial D(R)$.  By the min-max principle and domain
monotonicity (restricting to $H^1_0(D(R)) \subset H^1(\R^N)$):
\begin{equation}
\label{eq:dirichlet-upper}
  \lambda_k \;\le\; \lambda_k^{R,D}.
\end{equation}
Thus $\lambda_k^{R,D}$ is a computable \emph{upper bound} for $\lambda_k$.

\subsubsection*{Neumann Truncation}

Let $\lambda_k^{R,N}$ denote the $k$-th eigenvalue of $-\Delta+V$
on $D(R)$ with Neumann boundary condition
$\partial u/\partial\bm{n} = 0$ on $\partial D(R)$.
Define the \emph{exterior potential minimum}
\begin{equation}
\label{eq:sigma-R}
  \sigma_{\mathrm{ext}}(D(R)) \;:=\; \inf_{x\in\R^N\setminus D(R)}\, V(x).
\end{equation}
For the Coulomb-type potentials considered in this paper, all
singularities $a_i$ lie inside $D(R)$, so $V$ is bounded below
on $\R^N\setminus D(R)$ and $\sigma_{\mathrm{ext}}(D(R)) > -\infty$.
Explicitly: $\sigma_{\mathrm{ext}}(D(R)) = -Z/R$ for the single
Coulomb potential on a ball of radius $R$, while for a general
box domain the value in \eqref{eq:sigma-R} is obtained by minimizing
the explicit Coulomb expression over the boundary faces.

The Neumann lower bound below is a multidimensional Coulomb
specialisation of the classical Dirichlet--Neumann bracketing
principle for self-adjoint operators (see~\cite[Sec.~XIII.15]{Reed-Simon}),
in which the criterion takes the explicit
\emph{exterior-potential minimum} form \eqref{eq:confinement}. An
analogous Neumann-truncation lower bound under explicit checkable
conditions has been used in the one-dimensional Sturm--Liouville
literature for turning-point-type problems. The following lemma gives
the corresponding lower eigenvalue bound for the present setting.

\begin{lemma}[Neumann truncation lower bound]
\label{lem:neumann-lb}
Assume that $H = -\Delta + V$ on $\R^N$ has at least $k$ discrete
eigenvalues $\lambda_1 \le \cdots \le \lambda_k$ below
$\inf\sigma_{\rm ess}(H)$, with orthonormal eigenfunctions
$\psi_1,\ldots,\psi_k$.  If the confinement condition
\begin{equation}
\label{eq:confinement}
  \sigma_{\mathrm{ext}}(D(R)) \;>\; \lambda_k
\end{equation}
holds, then
\begin{equation}
\label{eq:neumann-lower}
  \lambda_k^{R,N} \;\le\; \lambda_k.
\end{equation}
\end{lemma}

\begin{proof}
Let $U_k := \mathrm{span}\{\psi_1,\ldots,\psi_k\}$ and define the
restriction map $T\colon U_k \to H^1(D(R))$, $Tu = u|_{D(R)}$.

\emph{Injectivity.}  If $u \in U_k$ satisfies $u|_{D(R)} = 0$,
then $u$ is supported on $\R^N \setminus D(R)$, and
\[
  \lambda_k\|u\|_{L^2(\R^N)}^2
  \;\ge\; a(u,u)
  \;=\; \int_{\R^N\setminus D(R)}\!\bigl(|\nabla u|^2+V|u|^2\bigr)\dx
  \;\ge\; \sigma_{\mathrm{ext}}(D(R))\|u\|_{L^2(\R^N)}^2,
\]
which forces $u = 0$ because $\sigma_{\mathrm{ext}}(D(R)) > \lambda_k$.
Hence $W := T(U_k)$ is a $k$-dimensional subspace of $H^1(D(R))$,
admissible in the Neumann min-max formula.

\emph{Rayleigh quotient bound.}  For $u \in U_k$ and
$v := u|_{D(R)} \in W$, splitting
$\R^N = D(R) \sqcup (\R^N\setminus D(R))$ gives
\begin{align*}
  a_{D(R)}(v,v)
  &= a(u,u) - \int_{\R^N\setminus D(R)}\!
     \bigl(|\nabla u|^2+V|u|^2\bigr)\dx \\
  &\le a(u,u)
     - \sigma_{\mathrm{ext}}(D(R))\bigl(\|u\|_{L^2(\R^N)}^2
     - \|v\|_{L^2(D(R))}^2\bigr).
\end{align*}
Writing $u = \sum_{j=1}^k c_j\psi_j$ and using orthonormality,
$$
 a(u,u) = \sum_{j=1}^k \lambda_j|c_j|^2
\le \lambda_k\|u\|_{L^2(\R^N)}^2\;.
$$
Since $\lambda_k - \sigma_{\mathrm{ext}}(D(R)) < 0$ and
$\|u\|_{L^2(\R^N)}^2 \ge \|v\|_{L^2(D(R))}^2$,
\[
  a_{D(R)}(v,v)
  \le \bigl(\lambda_k - \sigma_{\mathrm{ext}}(D(R))\bigr)
          \|u\|_{L^2(\R^N)}^2
        + \sigma_{\mathrm{ext}}(D(R))\|v\|_{L^2(D(R))}^2
  \le \lambda_k\|v\|_{L^2(D(R))}^2.
\]
Every Rayleigh quotient on $W$ is at most $\lambda_k$; the
min-max principle yields $\lambda_k^{R,N} \le \lambda_k$.
\end{proof}

\begin{remark}[Verifying the confinement condition for
\texorpdfstring{$\lambda_1$}{lambda\_1}]
\label{rmk:confinement-verify}
The ground-state eigenvalue $\lambda_1$ is not known a priori,
but \eqref{eq:confinement} can be verified a posteriori from
any computed upper bound $\bar\lambda_1\ge\lambda_1$---for
example, the conforming Galerkin eigenvalue
$\lambda_{1,h}^{\CG}$, or an accurate reference value from the
literature---since $\sigma_{\mathrm{ext}}(D(R))>\bar\lambda_1$
implies $\sigma_{\mathrm{ext}}(D(R))>\lambda_1$. The closed-form
threshold for ball truncations is
\begin{equation}
\label{eq:R-min}
  R_{\min,1} \;=\;
  \begin{cases}
    Z/|\lambda_1| & \text{(single-centre } V=-Z/|x|\text{)}, \\
    A+Z/|\lambda_1| &
      \text{(two-centre, } A:=\max_i|a_i|,\;Z:=Z_1+Z_2\text{)}.
  \end{cases}
\end{equation}
\Cref{tab:confinement-check} reports, for each test problem,
the chosen domain, the exterior-potential minimum
$\sigma_{\mathrm{ext}}(D(R))$, and the upper-bound surrogate
$\bar\lambda_1$ (source indicated in the footnote). In every
row $\sigma_{\mathrm{ext}}(D(R))>\bar\lambda_1$, so
\eqref{eq:confinement} is verified for $\lambda_1$ on each
problem.

\begin{table}[htbp]
\centering\small
\renewcommand{\arraystretch}{1.2}
\caption{Verification of the confinement condition
$\sigma_{\mathrm{ext}}(D(R))>\lambda_1$ for the four test problems
of \S\ref{sec:numerics}. The exterior minimum is computed for the
actual domain used in each experiment, so the Neumann-truncation
lower bound of \Cref{lem:neumann-lb} applies to the ground state.}
\label{tab:confinement-check}
\begin{tabular}{lccc}
\toprule
Model & Domain & $\sigma_{\mathrm{ext}}(D(R))$
      & $\bar\lambda_1{}^{\dagger}$ \\
\midrule
2D hydrogen ($V=-1/|x|$) & $[-5,5]^2$ & $-0.2000$ & $-1$      \\
2D $\mathrm{H}_2^+$ ($a_{1,2}=(\mp 2,0)$) & $[-7,7]\!\times\![-5,5]$ & $-0.3713$ & $-1.3170$ \\
3D hydrogen ($V=-1/|x|$) & $B(0,6)$ & $-0.2000$ & $-0.250$ \\
3D $\mathrm{H}_2^+$ ($a_{1,2}=(\mp 2,0,0)$) & $[-8,8]\!\times\![-6,6]^2$ & $-0.3713$ & $-0.5303$ \\
\bottomrule
\end{tabular}
\\[0.4em]
{\footnotesize $^{\dagger}$Sources for $\bar\lambda_1$:
2D hydrogen, exact $\lambda_1=-1$; 2D $\mathrm{H}_2^+$,
$\lambda_{1,h}^{P_1}=-1.3170$ at $h=0.2$ (\Cref{tab:twocenter-2d-cecr});
3D hydrogen, exact $\lambda_1=-0.250$; 3D $\mathrm{H}_2^+$,
$\lambda_{1,h}^{P_1}=-0.5303$ at $h=1.0$ (\Cref{tab:twocenter-3d-cecr}).}
\end{table}
\end{remark}

\begin{remark}[Why we do not use boundary flux signs]
\label{rmk:boundary-sign}
For a single eigenfunction $\psi_j$, Green's identity on $D(R)$
gives
\[
  a_{D(R)}(\psi_j,\psi_j)
  =
  \lambda_j\|\psi_j\|_{L^2(D(R))}^2
  +
  \int_{\partial D(R)}\psi_j\,\partial_{\bm n}\psi_j\,ds .
\]
Thus a nonpositive boundary flux would imply that the Neumann
Rayleigh quotient of $\psi_j|_{D(R)}$ is at most $\lambda_j$.
This observation can be useful in special situations, for example
for explicitly radial ground states.
However, it is not a robust proof of \eqref{eq:neumann-lower}.
Agmon estimates \cite{Agmon1982} give exponential decay, but they
do not by themselves imply a pointwise sign or monotonicity
condition for $\psi_j$ on large spheres; excited states may have
nodal sets reaching infinity.  Moreover, for the $k$-th min--max
argument one needs a Rayleigh quotient bound on every linear
combination in $\mathrm{span}\{\psi_1,\ldots,\psi_k\}$,
equivalently a sign condition on the corresponding boundary flux
quadratic form, not merely on the individual eigenfunctions.
\Cref{lem:neumann-lb} avoids these unverifiable boundary sign
assumptions and replaces them with the explicit, computable
criterion \eqref{eq:confinement}.
\end{remark}

\begin{remark}
The two truncated problems satisfy
$\lambda_k^{R,N} \le \lambda_k \le \lambda_k^{R,D}$
(under the confinement condition \eqref{eq:confinement}),
providing a bracket whose width $\lambda_k^{R,D} - \lambda_k^{R,N}$
decreases as $R \to \infty$.
\end{remark}

\subsection{Truncation Error}
\label{subsec:truncation-error}

The gap $\delta_k(R) := \lambda_k^{R,D} - \lambda_k \ge 0$
measures the error introduced by replacing $\R^N$ with $D(R)$.
A detailed derivation of explicit upper bounds for $\delta_k(R)$
is deferred to future work; here we collect the qualitative
picture sufficient to guide the choice of $R$ in practice.

\paragraph{Exponential decay of eigenfunctions}
By Agmon estimates \cite{Agmon1982}, eigenfunctions of $H = -\Delta+V$
decay exponentially outside any compact set, provided $V$ grows or
stays bounded.  More precisely:
\begin{itemize}
  \item \emph{Confining potential} ($V(x)\to+\infty$ as $|x|\to\infty$):
    the $k$-th eigenfunction satisfies
    $u_k(x) = O(e^{-\alpha_k|x|})$ for some $\alpha_k > 0$ depending
    on $\lambda_k$ and the growth rate of $V$.
  \item \emph{Attractive Coulomb} ($V = -Z/|x|$, $\lambda_k < 0$):
    eigenfunctions decay with rate $\alpha_k = \sqrt{|\lambda_k|}$
    (known exactly for the hydrogen atom \cite{Reed-Simon}).
  \item \emph{Two-centered Coulomb} ($V = -Z_1/|x-a_1|-Z_2/|x-a_2|$):
    the same Agmon theory applies with $\alpha_k = \sqrt{|\lambda_k|}$,
    since $V(x)\to 0$ as $|x|\to\infty$ and $\lambda_k < 0$.
\end{itemize}

\paragraph{Consequence for \texorpdfstring{$\delta_k(R)$}{delta\_k(R)}}
The exponential decay of $u_k$ implies that $\delta_k(R)$ is
exponentially small in $R$: schematically,
$\delta_k(R) \sim e^{-2\alpha_k R}$ as $R\to\infty$.
In practice, $R$ is chosen large enough that $\delta_k(R)$ is
negligible compared to the CECR approximation gap $U_k - L_k$.

\begin{remark}
An explicit, computable upper bound for $\delta_k(R)$—taking into
account the constants in the Agmon estimate—is an important
complement to the CECR framework developed here, and is left
as a topic for a companion paper.
\end{remark}

\subsection{Two-Sided Truncation Sandwich}
\label{subsec:reduction}

The Dirichlet and Neumann truncated eigenvalues together provide a
two-sided sandwich for the full-space eigenvalue $\lambda_k$.

\paragraph{Upper bound via Dirichlet truncation.}
By \eqref{eq:dirichlet-upper}, $\lambda_k\le\lambda_k^{R,D}$.
A conforming $P_1$ Galerkin discretisation of the Dirichlet problem
\begin{equation}
\label{eq:evp-dirichlet}
  (-\Delta + V)\,u = \lambda\, u \;\text{ on } D(R),
  \quad u\in H^1_0(D(R)),
\end{equation}
with bilinear forms
$a(u,v):=(\nabla u,\nabla v)_{D(R)}+(Vu,v)_{D(R)}$
and $b(u,v):=(u,v)_{D(R)}$,
yields Ritz values $\lambda_{k,h}^{\CG}\ge\lambda_k^{R,D}\ge\lambda_k$:
\begin{equation}
\label{eq:two-sided-ub}
  \lambda_k \;\le\; \lambda_k^{R,D} \;\le\; \lambda_{k,h}^{\CG}.
\end{equation}

\paragraph{Lower bound via Neumann truncation and CECR.}
By \eqref{eq:neumann-lower}, $\lambda_k^{R,N}\le\lambda_k$.
The CECR framework of \S\S\ref{sec:cecr}--\ref{sec:extension}
delivers a computable lower bound $L_k\le\lambda_k^{R,N}$
(via the piecewise-constant approximation $c_h\approx V$ and
the perturbation lemma of \Cref{sec:extension});
combining these gives $L_k\le\lambda_k$.

\paragraph{Schematic two-sided bound.}
\begin{equation}
\label{eq:two-sided-schematic}
  \underbrace{L_k}_{\text{CECR lower bound}}
  \;\le\; \lambda_k^{R,N}
  \;\le\; \lambda_k
  \;\le\; \lambda_k^{R,D}
  \;\le\;
  \underbrace{\lambda_{k,h}^{\CG}}_{\text{Galerkin upper bound}}.
\end{equation}
The truncation gap $\lambda_k^{R,D}-\lambda_k^{R,N}$ is
exponentially small in $R$ for large $R$
(see \S\ref{subsec:truncation-error}) and is treated as
negligible in the numerical experiments of \S\ref{sec:numerics}.

\begin{remark}[Boundary conditions in the numerics]
\label{rmk:bc-convention}
The CECR eigenproblem (§\S\ref{sec:cecr}--\ref{sec:extension})
uses the \emph{Neumann} variational space $H^1(\Omega)$
(natural boundary condition for the nonconforming ECR element;
no face-DOFs on $\partial\Omega$ are eliminated), giving lower
bounds for $\lambda_k^{R,N}$ and hence for $\lambda_k$.
The conforming $P_1$ eigenproblem uses
$H^1_0(\Omega)$ (\emph{Dirichlet} BC), giving upper bounds for
$\lambda_k^{R,D}$ and hence, after noting $\lambda_k^{R,D}\ge\lambda_k$,
effective upper bounds for $\lambda_k$ up to the exponentially small
truncation gap.
All CECR tables in \S\ref{sec:numerics} report Neumann-BC computations;
all conforming Galerkin tables report Dirichlet-BC computations.
\end{remark}

\section{CECR FEM for Piecewise Constant Potentials}
\label{sec:cecr}

Throughout this section $\Omega \subset \R^N$ is a bounded polyhedral
domain with a shape-regular triangulation $\mathcal{T}^h$
(triangles for $N=2$, tetrahedra for $N=3$).
We consider the model eigenvalue problem
\begin{equation}
\label{eq:evp-pc}
  a_h(u,v) = \mu\,b(u,v) \quad \forall v \in H^1(\Omega),
\end{equation}
where $c_h \in \mathbb{P}^0(\mathcal{T}^h)$ is a piecewise constant
function and
\begin{equation}
\label{eq:forms-pc}
  a_h(u,v) = (\nabla u,\nabla v) + (c_h\,u,v),
  \qquad b(u,v) = (u,v).
\end{equation}
The eigenvalues of \eqref{eq:evp-pc} are denoted
$\mu_1 \le \mu_2 \le \cdots$

\subsection{The CECR Space}
\label{subsec:cecr-space}

\paragraph{ECR space}
The Enriched Crouzeix--Raviart (ECR) finite element space on
$\mathcal{T}^h$ is
\begin{equation}
\label{eq:ecr-space}
  V_h^{\ECR} := \left\{
    v_h \in L^2(\Omega) \;\middle|\;
    \begin{array}{l}
      v_h|_K \in \mathbb{P}^1(K) \oplus \mathrm{span}\{\phi_K\}\;\forall K,\\[3pt]
      \displaystyle\int_F \jump{v_h}\,\ds = 0\;\forall \text{ interior face } F
    \end{array}
  \right\},
\end{equation}
where $\phi_K$ is the cell bubble on $K$ and $\jump{\cdot}$ is the
jump across an interior face.
The degrees of freedom are the face averages and the cell averages.

\paragraph{CECR bilinear form}
For piecewise constant $c_h$, the CECR discretization of
\eqref{eq:evp-pc} reads: find $u_h \in V_h^{\ECR}$ such that
\begin{equation}
\label{eq:cecr-bilinear}
  (\nabla u_h, \nabla v_h)_{\mathcal{T}^h}
  + (c_h\,\Pi_{0,h} u_h,\, \Pi_{0,h} v_h)
  = \mu_{k,h}\,(u_h,v_h)
  \quad \forall v_h \in V_h^{\ECR},
\end{equation}
where $\Pi_{0,h}$ is the $L^2$-projection onto $\mathbb{P}^0(\mathcal{T}^h)$
(element average), and $(\nabla\cdot,\nabla\cdot)_{\mathcal{T}^h}$
is the broken gradient inner product.

\paragraph{ECR interpolant}
The ECR interpolation operator $I_h^{\ECR}\colon H^1(\Omega)\to V_h^{\ECR}$
is defined by matching face and element moments:
\begin{equation}
\label{eq:ecr-interp}
  \int_F I_h^{\ECR}v\,\ds = \int_F v\,\ds \;\forall\text{ face }F,
  \qquad
  \int_K I_h^{\ECR}v\,\dx = \int_K v\,\dx \;\forall K.
\end{equation}
In particular, $v - I_h^{\ECR}v$ has zero mean on every element $K$.

\paragraph{CECR interpolation constant}
For each element $K$, the (dimensionless) ECR interpolation constant
$C^{\ECR}(K)$ is the smallest constant such that
\begin{equation}
\label{eq:cecr-constant-K}
  \| v - I_h^{\ECR} v \|_{L^2(K)}
  \;\le\; C^{\ECR}(K)\,h_K\,\|\nabla(v - I_h^{\ECR} v)\|_{L^2(K)}
  \quad \forall v \in H^1(K).
\end{equation}
Since $v - I_h^{\ECR}v$ has zero element mean, \eqref{eq:cecr-constant-K}
is a Poincar\'e-type estimate for the ECR interpolation residual.
The global CECR constant is
\begin{equation}
\label{eq:cecr-constant-global}
  C_h \;:=\; \max_{K \in \mathcal{T}^h} C^{\ECR}(K)\,h_K.
\end{equation}

\begin{proposition}[CECR constant via Payne--Weinberger]
\label{prop:cecr-constant-pw}
For any simplex $K\subset\mathbb{R}^N$ ($N=2,3$),
\begin{equation}
\label{eq:cecr-constant-pw}
  C^{\ECR}(K) \;\le\; \frac{1}{\pi}.
\end{equation}
For $N=2$ the sharper bound $C^{\ECR}(K)\le 0.1490$ holds
\cite{Xie2-Liu-2018}.
In particular,
$C_h=\max_K C^{\ECR}(K)h_K \le h_{\max}/\pi$ for every
shape-regular mesh.
\end{proposition}

\begin{proof}
Fix $v\in H^1(K)$ and set $w := v - I_h^{\ECR}v$.
By the moment-matching definition \eqref{eq:ecr-interp},
$\int_K w\,\dx = 0$, so $w$ has zero mean on $K$.
The Payne--Weinberger inequality \cite{PayneWeinberger1960,Bebendorf2003}
for a convex domain of diameter $h_K$ gives
\[
  \|w\|_{L^2(K)}
  \;\le\; \frac{h_K}{\pi}\,\|\nabla w\|_{L^2(K)},
\]
i.e.\ $\|v-I_h^{\ECR}v\|_{L^2(K)}
\le \frac{1}{\pi}\,h_K\,\|\nabla(v-I_h^{\ECR}v)\|_{L^2(K)}$.
Taking the supremum over $v$ yields $C^{\ECR}(K)\le 1/\pi$.
\end{proof}

\begin{remark}[Sharper constants]
\label{rmk:cecr-sharper}
The Payne--Weinberger bound $C^{\ECR}(K)\le 1/\pi\approx 0.3183$
holds for all zero-mean $H^1(K)$ functions and is therefore
shape-universal but not sharp for the ECR residual.
For $N=2$, the tighter bound $C^{\ECR}(K)\le 0.1490$,
proved rigorously in \cite{Xie2-Liu-2018} for the ECR element,
improves the explicit gap by a factor $0.1490\pi\approx 0.47$.
For $N=3$ we use the universal bound $C^{\ECR}(K)\le 1/\pi$.
Applying the Crouzeix--Raviart interpolation estimate also gives a
rigorous but rough bound \cite{Liu2015}. A corresponding rigorous and
sharp bound for the 3D ECR element is not available in the literature
and is left for future work.
In the numerical experiments (\Cref{sec:numerics}) we use
$C_h^{\mathrm{PW}}=\max_K h_K/\pi$ for all dimensions.
\end{remark}

\subsection{Pair-Space Lift and the Abstract CECR Explicit-Bound Theorem}
\label{subsec:cecr-pair-space}

\paragraph{Why a pair space?}
The CECR bilinear form $a_h$ in \eqref{eq:cecr-bilinear}
applies the potential $c_h$ to the cell average
$\Pi_{0,h}u_h$, not to $u_h$ itself. This decoupling of the
gradient component ($\nabla u_h$) from the reaction component
($\Pi_{0,h}u_h$) is naturally captured by embedding each function
$u_h$ into a pair $(u_h,\Pi_{0,h}u_h)$ in a product space: the
first slot carries the gradient energy, the second the reaction
energy. In this pair space, the CECR projection
$\widehat\Pi_h$ becomes $\widehat a$-orthogonal, which is the
engine behind all subsequent lower-bound theorems.

\smallskip
The sign-definite bound \eqref{eq:cecr-abstract} of
\Cref{thm:cecr-abstract} and the convergent elementwise bound
\eqref{eq:cecr-lb-convergent} of \Cref{thm:cecr-lb-convergent} below are
both instances of a single abstract theorem. We record the common
machinery here, following the pair-space construction of
\cite[\S4.1.3]{liu2024guaranteed} and the max-min argument of
\cite[\S3.1]{liu2024guaranteed}; the presentation is self-contained.

\paragraph{Pair space}
On the Cartesian product $L^2(\Omega)\times L^2(\Omega)$ define
\begin{equation}
\label{eq:pair-forms}
  \widehat a\bigl((u_1,u_2),(v_1,v_2)\bigr)
    := (\nabla u_1,\nabla v_1)_{\mathcal{T}^h}+(c_h u_2,v_2),
  \quad
  \widehat b\bigl((u_1,u_2),(v_1,v_2)\bigr)
    := (u_1,v_1).
\end{equation}
Let
\begin{equation}
\label{eq:pair-spaces}
  \widehat V   := \{(u,u):u\in H^1(\Omega)\},\qquad
  \widehat V_h := \{(u_h,\Pi_{0,h}u_h):u_h\in V_h^{\ECR}\},
\end{equation}
and $\widehat V(h):=\widehat V+\widehat V_h$.

\begin{lemma}[Diagonal isometry]
\label{lem:diag-isometry}
For every $u,v\in H^1(\Omega)$ and $\widehat u:=(u,u)$,
$\widehat v:=(v,v)$,
\begin{equation}
\label{eq:diag-iso}
  \widehat a(\widehat u,\widehat v) \;=\; a_h(u,v),
  \qquad
  \widehat b(\widehat u,\widehat v) \;=\; (u,v).
\end{equation}
The linear map $u\mapsto(u,u)$ is therefore an isometric bijection
$H^1(\Omega)\to\widehat V$ in the $(\widehat a,\widehat b)$-inner products
(when $c_h\ge 0$, $\widehat a$ defines a norm; for sign-changing $c_h$
it is a quadratic form used only in the isometry sense),
and the pair-space eigenvalue problem
\begin{equation}
\label{eq:pair-evp}
  \widehat a(\widehat u,\widehat v)
  \;=\; \widehat\mu\,\widehat b(\widehat u,\widehat v)
  \qquad \forall\,\widehat v\in\widehat V
\end{equation}
has spectrum $\{\widehat\mu_k\}=\{\mu_k\}$, where $\{\mu_k\}$ are
the eigenvalues of the original problem \eqref{eq:evp-pc}.
\end{lemma}

\begin{proof}
For $u,v\in H^1(\Omega)$, $(\nabla u,\nabla v)_{\mathcal{T}^h}
=(\nabla u,\nabla v)$ because $u,v$ carry no jumps across interior
faces, so $\widehat a(\widehat u,\widehat v)
=(\nabla u,\nabla v)+(c_h u,v)=a_h(u,v)$. The $\widehat b$-identity
is immediate. Bijectivity gives the eigenvalue correspondence.
\end{proof}

\paragraph{Pair-space interpolant}
Define
\begin{equation}
\label{eq:pair-interp}
  \widehat\Pi_h:\widehat V(h)\to\widehat V_h,
  \qquad
  \widehat\Pi_h(u_1,u_2) := \bigl(I_h^{\ECR}u_1,\,\Pi_{0,h}u_2\bigr),
\end{equation}
where $I_h^{\ECR}:H^1(\Omega)\to V_h^{\ECR}$ is the canonical ECR
interpolant defined in \eqref{eq:ecr-interp} below (in the proof of
\Cref{thm:cecr-abstract}). Because $\Pi_{0,h}I_h^{\ECR}u
=\Pi_{0,h}u$, the image of a diagonal pair $(u,u)\in\widehat V$ is
$(I_h^{\ECR}u,\Pi_{0,h}I_h^{\ECR}u)\in\widehat V_h$, so
$\widehat\Pi_h$ does map $\widehat V\to\widehat V_h$.

\begin{lemma}[$\widehat a$-orthogonality]
\label{lem:pair-orthogonality}
For every $\widehat u\in\widehat V(h)$ and every
$\widehat v_h=(v_h,\Pi_{0,h}v_h)\in\widehat V_h$,
\begin{equation}
\label{eq:pair-orth}
  \widehat a(\widehat u-\widehat\Pi_h\widehat u,\widehat v_h) \;=\; 0.
\end{equation}
Consequently, whenever $\widehat a$ is positive definite on
$\widehat V_h$, the map $\widehat\Pi_h$ is the $\widehat a$-orthogonal
projection $\widehat V(h)\to\widehat V_h$, and the Pythagorean identity
\begin{equation}
\label{eq:pair-pythagoras}
  \|\widehat u\|_{\widehat a}^{\,2}
  = \|\widehat\Pi_h\widehat u\|_{\widehat a}^{\,2}
  + \|\widehat u-\widehat\Pi_h\widehat u\|_{\widehat a}^{\,2}
\end{equation}
holds for every $\widehat u\in\widehat V(h)$.
\end{lemma}

\begin{proof}
Expanding $\widehat a$ and writing $\widehat u=(u_1,u_2)$,
$\widehat u-\widehat\Pi_h\widehat u
=(u_1-I_h^{\ECR}u_1,\,u_2-\Pi_{0,h}u_2)$, so
\[
  \widehat a(\widehat u-\widehat\Pi_h\widehat u,\widehat v_h)
  = (\nabla(u_1-I_h^{\ECR}u_1),\nabla v_h)_{\mathcal{T}^h}
  + (c_h(u_2-\Pi_{0,h}u_2),\Pi_{0,h}v_h).
\]
The first term vanishes by the ECR broken-gradient orthogonality,
derived as follows.
On each element $K$, $\text{div}\; (\nabla v_h)|_K$ is constant (since
$v_h\in V_h^{\ECR}\subset\mathbb{P}^2(\mathcal{T}^h)$ locally), so
integration by parts gives
\[
  (\nabla(u_1-I_h^{\ECR}u_1),\nabla v_h)_K
  = \int_{\partial K}(u_1-I_h^{\ECR}u_1)
    \underbrace{(\nabla v_h|_K\cdot\bm{n}_K)}_{\text{const.\ on each face}}\,\mathrm{d}S.
\]
By the definition of $I_h^{\ECR}$ in \eqref{eq:ecr-interp}, the
face average of $u_1-I_h^{\ECR}u_1$ vanishes on every face
$F\subset\partial K$, and since $\nabla v_h|_K\cdot\bm{n}_K$ is
constant on $F$, the integral over $F$ is zero.
Summing over all faces of $K$ gives
$(\nabla(u_1-I_h^{\ECR}u_1),\nabla v_h)_K=0$
for every $K$ and every $v_h\in V_h^{\ECR}$.
For the
second, $c_h\Pi_{0,h}v_h\in\mathbb{P}^0(\mathcal{T}^h)$ is
element-constant, and $\int_K(u_2-\Pi_{0,h}u_2)\,\dx=0$ on every $K$,
so the $L^2$-pairing vanishes cell by cell. The Pythagorean identity
follows from $\widehat a$-orthogonality in the standard way.
\end{proof}

\begin{remark}
The assumption $c_h\in\mathbb{P}^0(\mathcal{T}^h)$ is essential: a
piecewise-affine $c_h$ would break the cell-wise cancellation in the
second term above, invalidating the Pythagorean
identity~\eqref{eq:pair-pythagoras} and the CECR lower-bound argument.
\end{remark}

\paragraph{Admissibility and positivity on pair space}
The CECR admissibility hypothesis ``$a_h$ is positive definite on
$V_h^{\ECR}$'' transfers to the pair space as follows. For
$(u_h,\Pi_{0,h}u_h)\in\widehat V_h$,
$\widehat a((u_h,\Pi_{0,h}u_h),(u_h,\Pi_{0,h}u_h))
=\|\nabla u_h\|_{\mathcal{T}^h}^{\,2}+(c_h\Pi_{0,h}u_h,\Pi_{0,h}u_h)
=a_h(u_h,u_h)$.
Hence $\widehat a$ is positive definite on $\widehat V_h$ if and only
if $a_h$ is positive definite on $V_h^{\ECR}$. Similarly, by
\Cref{lem:diag-isometry}, $\widehat a$ is positive definite on
$\widehat V$ if and only if $a_h$ is positive definite on
$H^1(\Omega)$. The sign of $c_h$ never enters.

\begin{lemma}[Pair-space CECR residual estimate]
\label{lem:cecr-pair-residual}
Let $c_h\ge 0$ be piecewise constant. Then for every
$\widehat u=(u,u)\in\widehat V$,
\begin{equation}
\label{eq:pair-error-estimate}
  \|\widehat u-\widehat\Pi_h\widehat u\|_{\widehat b}
  \;\le\; C_h\,\|\widehat u-\widehat\Pi_h\widehat u\|_{\widehat a},
\end{equation}
where $C_h:=\max_K C^{\ECR}(K)\,h_K$ is the global CECR
interpolation constant from \eqref{eq:cecr-constant-K}.
\end{lemma}

\begin{proof}
From \eqref{eq:pair-interp}, $\widehat u-\widehat\Pi_h\widehat u
=(u-I_h^{\ECR}u,\,u-\Pi_{0,h}u)$. Reading off the norms
from \eqref{eq:pair-forms}:
\begin{align*}
  \|\widehat u-\widehat\Pi_h\widehat u\|_{\widehat b}^{\,2}
  &= \|u-I_h^{\ECR}u\|^2, \\
  \|\widehat u-\widehat\Pi_h\widehat u\|_{\widehat a}^{\,2}
  &= \|\nabla_h(u-I_h^{\ECR}u)\|^2
\!   +\! \bigl(c_h(u-\Pi_{0,h}u),\,u-\Pi_{0,h}u\bigr)
  \!\ge\! \|\nabla_h(u-I_h^{\ECR}u)\|^2,
\end{align*}
where the inequality uses $c_h\ge 0$. Applying the local CECR
interpolation estimate \eqref{eq:cecr-constant-K} on each $K$
and summing,
\begin{equation}
\label{eq:ecr-local-L2-to-grad}
  \|u-I_h^{\ECR}u\|^2
  \;=\;\sum_K\|u-I_h^{\ECR}u\|_{L^2(K)}^{\,2}
  \;\le\; C_h^{\,2}\,\|\nabla_h(u-I_h^{\ECR}u)\|^2
  \;\le\; C_h^{\,2}\,\|\widehat u-\widehat\Pi_h\widehat u\|_{\widehat a}^{\,2},
\end{equation}
which is \eqref{eq:pair-error-estimate}.
\end{proof}

\begin{theorem}[CECR explicit lower bound, Liu \cite{Liu2015}]
\label{thm:cecr-abstract}
Let $c_h > 0$ be piecewise constant. Then $\mu_{k,h} >  0$ for
all $k$, and
\begin{equation}
\label{eq:cecr-abstract}
  \mu_k \;\ge\;
  \frac{\mu_{k,h}}{1 + C_h^{\,2}\,\mu_{k,h}},
  \qquad k=1,2,\ldots.
\end{equation}
\end{theorem}

\begin{proof}
Since $c_h> 0$, the form
$a_h(v_h,v_h)=\|\nabla v_h\|_{\mathcal{T}^h}^{\,2}
+(c_h\Pi_{0,h}v_h,\Pi_{0,h}v_h)\ge 0$, with equality implies $v_h=0$;
hence $a_h$ is positive definite on both $H^1(\Omega)$ and
$V_h^{\ECR}$, so $\mu_{k,h} > 0$.
The remaining argument uses only two ingredients:
(i) $\widehat\Pi_h$ is an $\widehat a$-orthogonal projection
from $\widehat V$ onto $\widehat V_h$ (so that the Pythagorean
identity \eqref{eq:pair-pythagoras} holds), and (ii) the
residual estimate of \Cref{lem:cecr-pair-residual}. Neither
the specific structure of $\widehat\Pi_h$ (defined in
\eqref{eq:pair-interp}) nor that of $V_h^{\ECR}$ is used
anywhere else in the derivation. Consequently, once
\Cref{lem:cecr-pair-residual} is in hand, the argument is
the same, step for step, as the general projection-based
eigenvalue-bound framework of \cite[Theorem~2.1]{Liu2015} and
\cite[Theorem~4.1]{liu2024guaranteed}. For self-containedness,
we reproduce the proof in \Cref{app:cecr-abstract-proof}.
\end{proof}

\subsection{Sign-Changing Piecewise Constant Potential: Classical Shift Remark}
\label{subsec:cecr-shift}

When $c_h$ takes negative values, $a_h$ need not be positive definite
and \Cref{thm:cecr-abstract} does not directly apply.  A classical
remedy is to shift by the $L^\infty$ norm of the negative part:
\begin{equation}
\label{eq:gamma-h}
  \gamma_h := \|c_h^-\|_{L^\infty(\Omega)}
  = \max_{K : c_h|_K < 0} |c_h|_K|.
\end{equation}
Because the CECR mass matrix is unchanged by a constant shift of
$c_h$, applying \Cref{thm:cecr-abstract} to
$-\Delta+(c_h+\gamma_h)$ and un-shifting gives
\begin{equation}
\label{eq:cecr-optimal-shift}
  \mu_k \;\ge\;
  \frac{\mu_{k,h}+\gamma_h}{1+(\mu_{k,h}+\gamma_h)\,C_h^2}-\gamma_h
  \;=:\; L_k^{\mu,\gamma_h}.
\end{equation}

\begin{remark}[Limitation of the classical $\gamma_h$-shift]
\label{rmk:gamma-h-limitation}
The estimate \eqref{eq:cecr-optimal-shift} is sharp when $\gamma_h$
is mesh-bounded, e.g.\ for $c_h^-\in L^\infty(\Omega)$ uniformly in
$h$.  For Coulomb-type singularities, however, the element average
yields $\gamma_h\sim h_{K(a)}^{-1}$ on \emph{every} shape-regular mesh
(the singular element contributes
$|K|^{-1}\int_K |x-a|^{-1}\,\dx\sim h_K^{-1}$), so the additive
term $-\gamma_h$ in \eqref{eq:cecr-optimal-shift} diverges as
the local diameter $h_{K(a)}$ at the singularity tends to zero.
Convergence of $L_k^{\mu,\gamma_h}$ then requires the
combined product $\gamma_h\varepsilon_h$ (controlled by the
perturbation to the continuous potential, see \S\ref{sec:extension})
to vanish, a condition that typically fails on uniform meshes and
forces aggressive mesh grading.  This motivates the \emph{elementwise} CECR bound developed in
\S\ref{subsec:cecr-convergent} below, which avoids $\gamma_h$ entirely;
we keep \eqref{eq:cecr-optimal-shift} only for comparison purposes.
\end{remark}

\subsection{Convergent CECR Lower Bound via Elementwise Estimates}
\label{subsec:cecr-convergent}

We now establish a \emph{computable} CECR lower bound for
sign-changing $c_h$ that bypasses the form-norm constant
$\widetilde C_h$ entirely. The key insight is that the negative
reaction error can be bounded elementwise by exploiting the
zero-cell-mean property of $w_2=u-\Pi_{0,h}u$ via the
Payne--Weinberger inequality on each element. Combined with the
standard form-boundedness of $c_h^-$ for coercivity, this yields
a bound with an explicit, mesh-dependent correction factor $A_h\to 0$
under refinement, so that the gap closes.

\paragraph{Shifted pair-space forms}
For a fixed shift $\sigma>0$, define the \emph{shifted CECR bilinear
form} on $V_h^{\ECR}$ by
\begin{equation}
\label{eq:shifted-cecr}
  a_h^\sigma(u_h,v_h)
  :=
  (\nabla u_h,\nabla v_h)_{\mathcal{T}^h}
  + \bigl((c_h{+}\sigma)\,\Pi_{0,h}u_h,\,\Pi_{0,h}v_h\bigr),
  \qquad u_h,v_h\in V_h^{\ECR}.
\end{equation}
The corresponding shifted pair-space form is
\begin{equation}
\label{eq:shifted-pair-forms}
  \widehat a^\sigma\bigl((u_1,u_2),(v_1,v_2)\bigr)
  := (\nabla u_1,\nabla v_1)_{\mathcal{T}^h}
     + ((c_h{+}\sigma)\,u_2,v_2).
\end{equation}
All machinery from \S\ref{subsec:cecr-pair-space}---diagonal isometry,
$\widehat a^\sigma$-orthogonality of $\widehat\Pi_h$, Pythagorean
identity---carries over with $c_h$ replaced by $c_h+\sigma$.

\begin{assumption}[Form-bounded negative part]
\label{ass:form-bounded}\label{ass:kato-ch}
The negative part of the potential is form-bounded both at the
continuous level and at the discrete (cell-average) level:
\begin{enumerate}[label=(\roman*)]
\item There exist constants $\alpha\in[0,1)$ and $\beta\ge 0$ such that
\begin{equation}
\label{eq:form-bound}
  \int_\Omega V^-(x)\,|v(x)|^2\,\dx
  \;\le\; \alpha\,\|\nabla v\|_{L^2(\Omega)}^{\,2}
        + \beta\,\|v\|_{L^2(\Omega)}^{\,2}
  \qquad \forall\,v\in H^1(\Omega).
\end{equation}
\item There exist $\epsilon\in(0,1)$ and $C_\epsilon>0$, independent of
$h$, such that
\begin{equation}
\label{eq:kato-ch}
  (c_h^-\,u,u)
  \;\le\; \epsilon\,\|\nabla u\|^2 + C_\epsilon\,\|u\|^2
  \qquad \forall\,u\in H^1(\Omega).
\end{equation}
\end{enumerate}
Item~(i) underwrites self-adjointness of $H=-\Delta+V$ via KLMN and the
form comparison of \S\ref{subsec:perturbation}; item~(ii) is the
hypothesis used by the CECR coercivity argument below. We retain both
as separate hypotheses because the inequality
$(c_h^-\,u,u)\le(V^-\,u,u)$ holds Jensen-by-Jensen only on
piecewise-constant $u$, so transferring (i) to (ii) for general
$u\in H^1(\Omega)$ is not direct.
\end{assumption}

\begin{remark}[Validity of \Cref{ass:form-bounded} for Coulomb-type potentials]
\label{rmk:kato-ch-valid}
For the Coulomb negative part $V^-=Z/|x-a|$ on $\Omega=D(R)$ (single
centre; finite sums of centres extend by adding the right-hand sides):

\smallskip
\noindent\emph{Continuous bound (i).}
The required conclusion on the bounded Lipschitz domain
$\Omega=D(R)$ is the infinitesimal form bound
\begin{equation}
\label{eq:coulomb-cont-form-bound}
  \int_\Omega \frac{Z|u|^2}{|x-a|}\,dx
  \le
  \epsilon\|\nabla u\|_{L^2(\Omega)}^2
  + C_\epsilon\|u\|_{L^2(\Omega)}^2
  \qquad \forall\,\epsilon>0,\quad \forall u\in H^1(\Omega),
\end{equation}
with $C_\epsilon$ independent of the mesh.  This is the bounded-domain
version of \eqref{eq:form-bound}; it is important that the $L^2$
term is present.  The classical Hardy inequality in $N=3$,
\[
  \int_{\mathbb R^3}\frac{|w|^2}{|x-a|^2}\,dx
  \le 4\|\nabla w\|_{L^2(\mathbb R^3)}^2,
\]
is a full-space (or zero-trace) statement and should not be read as a
gradient-only estimate for general Neumann-domain functions, since
constants would then be counterexamples.

One way to obtain \eqref{eq:coulomb-cont-form-bound} in $N=3$ is to
extend $u\in H^1(\Omega)$ to $Eu\in H^1(\mathbb R^3)$, apply the
full-space Hardy inequality to $Eu$, use
$|x-a|^{-1}\le\delta |x-a|^{-2}+(4\delta)^{-1}$, and absorb the
extension constants into $C_\epsilon$. In $N=2$, one may instead
choose $p\in(1,2)$, use $|x-a|^{-1}\in L^p(\Omega)$, H\"older's
inequality, and the bounded-domain Sobolev/Gagliardo--Nirenberg
embedding $H^1(\Omega)\hookrightarrow L^{2p/(p-1)}(\Omega)$, followed
by Young's inequality. This again yields \eqref{eq:coulomb-cont-form-bound}.

For later use it is convenient to record a direct local proof that gives
the same type of estimate in both $N=2$ and $N=3$.  After translating the
centre to the origin, a weighted Hardy estimate on a ball gives, for
every $u\in H^1_0(B(\rho))$ and $N=2,3$,
\begin{equation}
\label{eq:hardy-ball}
  \int_{B(\rho)}\frac{u^2}{|x|}\,dx
  \;\le\; C_N\rho\,\|\nabla u\|_{L^2(B(\rho))}^2,
\end{equation}
where $C_N$ is dimension dependent.  In $N=3$ this follows
immediately from the classical Hardy inequality for
$|x|^{-2}$. In $N=2$ the classical $|x|^{-2}$ Hardy inequality
is not available, but \eqref{eq:hardy-ball} still holds for the
weaker Coulomb weight $|x|^{-1}$: writing
$u(r,\theta)=-\int_r^\rho\partial_su(s,\theta)\,ds$ and applying
Cauchy--Schwarz with the radial weight gives a logarithmic factor
whose integral in $r$ is finite and of order $\rho$.

Localising an arbitrary $u\in H^1(D(R))$ with a radial cutoff
$\chi_\rho$ satisfying $\chi_\rho\equiv 1$ on $B(\rho/2)$,
$\chi_\rho\equiv 0$ outside $B(\rho)$, and $|\nabla\chi_\rho|\le 2/\rho$,
splitting the integral on $B(\rho/2)$ versus $D(R)\setminus B(\rho/2)$,
applying \eqref{eq:hardy-ball} to $v=\chi_\rho u$ on the inner piece, and
bounding $1/|x|\le 2/\rho$ on the outer piece, yields
\begin{equation}
\label{eq:coulomb-rho}
  \int_{D(R)}\frac{Z\,u^2}{|x|}\,dx
  \;\le\; C Z\rho\,\|\nabla u\|^2 + \frac{C Z}{\rho}\,\|u\|^2
  \qquad\forall\,\rho>0,\;u\in H^1(D(R)).
\end{equation}
Here and below $C$ denotes a dimension- and cutoff-dependent constant
independent of $h$. Choosing $\rho\simeq\epsilon/Z$ in
\eqref{eq:coulomb-rho} recovers \eqref{eq:coulomb-cont-form-bound} with,
for example, $C_\epsilon=CZ^2/\epsilon$.  For a multi-centre potential,
the corresponding constants are summed over the centres.

\smallskip
\noindent\emph{Discrete cell-average bound (ii).}
The preceding Hardy--cutoff argument proves the continuous form
bound for $V^-$. It does not, by itself, prove
\eqref{eq:kato-ch} for arbitrary cell averages $c_h^-$ and arbitrary
meshes. For the graded, shape-regular meshes used here, however,
one obtains the discrete estimate by comparing the cell averages
locally with the Coulomb weight. On elements in the singular patch,
shape regularity and the assumption that each nucleus is a mesh
vertex give $c_h^-|_K=O(h_K^{-1})$ and
$K\subset B(a,Ch_K)$; hence
\[
  c_h^-|_K\int_K u^2
  \;\le\;
  C\int_{B(a,Ch_K)}\frac{u^2}{|x-a|}\,dx .
\]
On elements away from the nuclei, the distance-grading condition
$h_K\lesssim{\rm dist}(K,\{a_i\})$ implies
$c_h^-|_K\le C\inf_{x\in K^\ast}V^-(x)$ on a uniformly bounded
patch $K^\ast$. Summing over the mesh and using the finite overlap
of these patches reduces \((c_h^-u,u)\) to the continuous Coulomb
form, up to a mesh-independent multiplicative constant. Combining
this reduction with \eqref{eq:coulomb-cont-form-bound} gives \eqref{eq:kato-ch}
with an $h$-independent, although generally nonsharp, constant.
In the computations we instead use the sharper Rayleigh-quotient
certification described in \Cref{rmk:ceps-sharp} and
\Cref{app:ceps-eta}.

It is worthy pointing out that the Hardy--cutoff route only certifies the \emph{existence} of an
admissible pair $(\epsilon,C_\epsilon)$; the resulting
$C_\epsilon=O(Z^2/\epsilon)$ is generally far from sharp.  A computable
and substantially tighter value is obtained from a Rayleigh-quotient
principle; see \Cref{rmk:ceps-sharp} and \Cref{app:ceps-eta}.
\end{remark}

\begin{remark}[Rigorous and sharp estimation of $C_\epsilon$]
\label{rmk:ceps-sharp}
Let $\eta$ be the minimal value of the following Rayleigh quotient for a fixed
$\epsilon$:
\begin{equation}
\label{eq:opt-c-eps}
\eta = \inf _{0\ne u\in H^1(\Omega)}
\frac{\epsilon\|\nabla u\|^2 - \int_{\Omega }c_h^-u^2 \dx }{\|u\|^2}.
\end{equation}
The trial function $u=1$ shows that $\eta<0$.  By the definition of
$\eta$, one has
$$
\int_\Omega c_h^-u^2 \dx \le \epsilon \|\nabla u\|^2 - \eta \|u\|^2.
$$
Thus the optimal constant is $C_\epsilon=-\eta$. The minimization problem
\eqref{eq:opt-c-eps} has essentially the same structure as the Schr\"odinger
eigenvalue problem and can therefore be solved rigorously, by the sign-changing
CECR lower-bound theorem, to provide a sharper value of $C_\epsilon$ than the
one obtained from general inequalities.  The resulting computable lower bound
for $\eta$ is given in \Cref{app:ceps-eta}.
\end{remark}

\begin{lemma}[Coercivity of the shifted pair-space form]
\label{lem:coercivity-sigma}
Under \Cref{ass:kato-ch}, for every $\sigma>C_\epsilon$ and every
$u\in H^1(\Omega)$,
\begin{equation}
\label{eq:coercivity}
  \|(u,u)\|_{\widehat a^\sigma}^{\,2}
  \;\ge\; (1-\epsilon)\,\|\nabla u\|^2
        + (\sigma-C_\epsilon)\,\|u\|^2.
\end{equation}
In particular,
\begin{equation}
\label{eq:nabla-control}
  \|\nabla u\|^2
  \;\le\; \frac{1}{1-\epsilon}\,\|(u,u)\|_{\widehat a^\sigma}^{\,2}.
\end{equation}
\end{lemma}

\begin{proof}
Expand and separate the shift:
\begin{align*}
  \|(u,u)\|_{\widehat a^\sigma}^{\,2}
  &= \|\nabla u\|^2 + (c_h\,u,u) + \sigma\|u\|^2 \\
  &\ge \|\nabla u\|^2 - (c_h^-\,u,u) + \sigma\|u\|^2 \\
  &\ge (1-\epsilon)\|\nabla u\|^2 + (\sigma-C_\epsilon)\|u\|^2.
\end{align*}
\end{proof}

\begin{definition}[Elementwise reaction constant]
\label{def:Gamma-h}
Define the \emph{elementwise reaction constant}
\begin{equation}
\label{eq:Gamma-h}
  \Gamma_h
  \;:=\; \max_{K\in\mathcal{T}^h}
         \frac{c_h^-\big|_K\;\cdot\; h_K^{\,2}}{\pi^2},
\end{equation}
where $c_h^-|_K=\max(-c_h|_K,0)$ is the negative part of $c_h$ on $K$.
\end{definition}

\begin{lemma}[Elementwise reaction bound for zero-cell-mean functions]
\label{lem:reaction-w2}
For every $u\in H^1(\Omega)$ and $w_2:=u-\Pi_{0,h}u$,
\begin{equation}
\label{eq:reaction-w2}
  (c_h^-\,w_2,w_2)
  \;\le\; \Gamma_h\,\|\nabla u\|^2.
\end{equation}
\end{lemma}

\begin{proof}
Since $c_h^-$ is constant on each $K$ and
$\int_K w_2\dx=0$, the Payne--Weinberger inequality gives
$\|w_2\|_{L^2(K)}^{\,2}\le(h_K/\pi)^2\|\nabla w_2\|_{L^2(K)}^{\,2}
=(h_K/\pi)^2\|\nabla u\|_{L^2(K)}^{\,2}$
($\nabla w_2=\nabla u$ \ on $K$). Summing:
\[
  (c_h^-\,w_2,w_2)
  = \sum_K c_h^-\big|_K\|w_2\|_K^{\,2}
  \le \sum_K \frac{c_h^-|_K\,h_K^{\,2}}{\pi^2}\,\|\nabla u\|_K^{\,2}
  \le \Gamma_h\,\|\nabla u\|^2.
\]
\end{proof}

\begin{remark}[Behaviour of $\Gamma_h$ for Coulomb potentials]
\label{rmk:Gamma-h-rate}
For $V=-1/|x|$ on a shape-regular mesh with mesh
function $h_K$, the singular element $K_0\ni 0$ contributes
$c_h^-|_{K_0}\approx C_0/h_{K_0}$ and
$\gamma_{K_0}=c_h^-|_{K_0}\,h_{K_0}^{\,2}/\pi^2
\approx C_0\,h_{K_0}/\pi^2$,
which vanishes with the local mesh size.  On a uniform mesh
($h_K=h$ for all $K$), elements at distance $r\ge h$ from the
origin contribute
$\gamma_K\le h^2/(r\pi^2)$, with maximum $\gamma\sim h/\pi^2$ attained
near the singularity, so that $\Gamma_h=O(h)$.  On a
graded mesh with $h_K\propto h\,|x_K|^\theta$, where $h$ is the
mesh-family parameter, the far-field dominates:
$\gamma_K\sim h_{\max}^2|x_K|^{2\theta-1}/\pi^2$
at distance $|x_K|$ from the singularity, and
$\Gamma_h=O(h_{\max}^{\,2}/R)$ for $\theta>1/2$, where $R$
is the domain radius. In either case, $\Gamma_h\to 0$ as the mesh
is refined,
in sharp contrast to $\gamma_h=\|c_h^-\|_{L^\infty}\to\infty$.
\end{remark}

\begin{theorem}[Convergent CECR lower bound for sign-changing $c_h$]
\label{thm:cecr-lb-convergent}
Let $c_h\in\mathbb{P}^0(\mathcal{T}^h)$ be piecewise constant
(sign unrestricted).
Suppose \Cref{ass:kato-ch} holds with parameters $\epsilon\in(0,1)$
and $C_\epsilon>0$, and let $\sigma>C_\epsilon$ be such that
$a_h^\sigma$ (see definition in \eqref{eq:shifted-cecr}) is positive definite on both $H^1(\Omega)$ and
$V_h^{\ECR}$. Define
\begin{equation}
\label{eq:A-constant}
  A_h \;:=\; \frac{\Gamma_h}{1-\epsilon}.
\end{equation}
Let $\mu_{k,h}^\sigma>0$ be the $k$-th eigenvalue of the shifted
CECR eigenproblem $a_h^\sigma(u_h,v_h)=\mu_{k,h}^\sigma(u_h,v_h)$
on $V_h^{\ECR}$ (i.e.\ the $k$-th CECR Ritz value for potential
$c_h+\sigma$; see \S\ref{subsec:fixed-shift}), and let
$C_h$ be the CECR interpolation constant from
\eqref{eq:cecr-constant-global}. Then
\begin{equation}
\label{eq:cecr-lb-convergent}
  \mu_k
  \;\ge\;
  \frac{\mu_{k,h}^\sigma}
       {(1+A_h)\bigl(1+C_h^{\,2}\,\mu_{k,h}^\sigma\bigr)}
  \;-\;\sigma
  \;=:\; L_k^{\mu,\sigma}.
\end{equation}
If a computable upper bound $\widetilde C_h\ge C_h$ is available, then
the same lower bound remains valid with $C_h$ replaced by
$\widetilde C_h$.  In particular, one may take the universal
Payne--Weinberger surrogate
\[
  \widetilde C_h=C_h^{\mathrm{PW}}:=\max_K h_K/\pi
\]
from \Cref{prop:cecr-constant-pw}.  Moreover,
$\mu_{k,h}^\sigma\le\mu_{k,h}+\sigma$, with equality under
refinement (the difference is $O(\sigma h^2\mu_{k,h})$ in the
mesh-family parameter); see also Remark \ref{rmk:shift-comparison}.
\end{theorem}

\begin{proof}
The argument parallels that of \Cref{thm:cecr-abstract}, but
with a modified pair-space error estimate that keeps track of the
sign-indefinite reaction term.

\emph{Step~1: error estimate.}
For $(u,u)\in\widehat V$, write $w_1=u-I_h^{\ECR}u$,
$w_2=u-\Pi_{0,h}u$, and note $\int_K w_1\dx=0$ on every $K$
(by the element-moment matching in \eqref{eq:ecr-interp}).
Expanding $\widehat a^\sigma$ as a quadratic form on the error pair
(note: $\widehat a^\sigma$ is not necessarily positive on off-diagonal
pairs $(w_1,w_2)$ with $w_1\ne w_2$; we use it here purely as a
bilinear form, not a norm):
\begin{equation}
\label{eq:ew-grad-split}
  \|\nabla w_1\|_{\mathcal{T}^h}^{\,2}
  = \widehat a^\sigma((w_1,w_2),(w_1,w_2))
    - (c_h^\sigma\,w_2,w_2).
\end{equation}
Since $\sigma>0$, the negative part satisfies
$(c_h+\sigma)^-\le c_h^-$ pointwise, so
\begin{equation}
	\label{eq:est-using-gamma-h}
  -(c_h^\sigma\,w_2,w_2)
  \;\le\; (c_h^-\,w_2,w_2)
  \;\le\; \Gamma_h\,\|\nabla u\|^2
\end{equation}
by \Cref{lem:reaction-w2}, and
$\|\nabla u\|^2\le(1{-}\epsilon)^{-1}\widehat a^\sigma((u,u),(u,u))$
by \Cref{lem:coercivity-sigma}.  Denoting
$\alpha:=\widehat a^\sigma(\widehat\Pi_h(u,u),\widehat\Pi_h(u,u))$ and
$\beta:=\widehat a^\sigma((w_1,w_2),(w_1,w_2))$
(with $\alpha+\beta=\widehat a^\sigma((u,u),(u,u))$ by the
Pythagorean identity \eqref{eq:pair-pythagoras},
which holds since $(u,u)=\widehat\Pi_h(u,u)+(w_1,w_2)$ is
$\widehat a^\sigma$-orthogonal by \Cref{lem:pair-orthogonality}), we obtain
\begin{equation}
\label{eq:ew-grad-bound}
  \|\nabla w_1\|^2
  \;\le\; \beta + A_h(\alpha+\beta)
  \;=\; (1{+}A_h)\,\beta + A_h\,\alpha.
\end{equation}
Applying the CECR interpolation estimate
\eqref{eq:cecr-constant-global} to
$w_1=u-I_h^{\ECR}u$,
\begin{equation}
\label{eq:ew-w1-bound}
  \|w_1\|^2
  \;\le\; C_h^{\,2}\,\|\nabla w_1\|^2
  \;\le\; C_h^{\,2}\bigl[(1{+}A_h)\,\beta + A_h\,\alpha\bigr].
\end{equation}

\emph{Step~2: max-min argument.}
Let $E_k=\mathrm{span}\{\widehat u_1,\ldots,\widehat u_k\}$
($\widehat a^\sigma$-orthonormal eigenvectors on $\widehat V$).
Write $\widetilde R(\widehat\phi):=
\widehat b(\widehat\phi,\widehat\phi)/\widehat a^\sigma(\widehat\phi,\widehat\phi)$
for the reciprocal Rayleigh quotient, so that
$\min_{E_k}\widetilde R = 1/\mu_k^\sigma$
(cf.\ the max-min identity \eqref{eq:max-min-pair} adapted to $a^\sigma$).

\emph{Case~1.} If $\widehat\Pi_h\widehat\phi=0$ for some
$\widehat\phi\in E_k\setminus\{0\}$, then $\alpha=0$ and
$\|(u,u)\|_{\widehat a^\sigma}^{\,2}=\beta$. From
\eqref{eq:ew-w1-bound},
$\|w_1\|^2\le C_h^{\,2}(1{+}A_h)\beta$, so
$\widetilde R(\widehat\phi)\le C_h^{\,2}(1{+}A_h)$.
Since $\min_{E_k}\widetilde R=1/\mu_k^\sigma$, we get
$\mu_k^\sigma\ge 1/(C_h^{\,2}(1{+}A_h))
\ge\mu_{k,h}^\sigma/((1{+}A_h)(1+C_h^{\,2}\mu_{k,h}^\sigma))$
(the last step is elementary), which is
\eqref{eq:cecr-lb-convergent} after un-shifting.

\emph{Case~2.} Otherwise $\widehat\Pi_h$ is injective on $E_k$,
and $\widehat\Pi_h(E_k)\subset\widehat V_h$ is $k$-dimensional.
Using $\widehat\Pi_h(E_k)$ as a trial subspace in the max-min for
$1/\mu_{k,h}^\sigma$, choose
$\widehat\phi\in E_k$ so that
$\widehat v_h:=\widehat\Pi_h\widehat\phi$ satisfies
\[
  \widetilde R(\widehat v_h)
  =
  \min_{\widehat z_h\in\widehat\Pi_h(E_k)\setminus\{0\}}
  \widetilde R(\widehat z_h)
  \le \frac{1}{\mu_{k,h}^\sigma}.
\]
By homogeneity we scale $\widehat\phi$ so that
$\|\widehat\phi\|_{\widehat a^\sigma}=1$.  Then
$\alpha+\beta=1$ and triangle inequality gives
\[
  \|\widehat\phi\|_{\widehat b}
  \;\le\; \|\widehat v_h\|_{\widehat b}
        + \|\widehat\phi-\widehat v_h\|_{\widehat b}.
\]
First term: by the choice of $\widehat v_h$,
$\|\widehat v_h\|_{\widehat b}\le(\mu_{k,h}^\sigma)^{-1/2}\sqrt{\alpha}$.
Second term: $\|\widehat\phi-\widehat v_h\|_{\widehat b}^{\,2}
=\|w_1\|^2\le C_h^{\,2}[(1{+}A_h)\beta+A_h\alpha]
=C_h^{\,2}[A_h+\beta]$, where the simplification uses
$(1{+}A_h)\beta+A_h\alpha = \beta+A_h(\alpha{+}\beta)=\beta+A_h$
together with $\alpha+\beta=1$ (the normalisation
$\|\widehat\phi\|_{\widehat a^\sigma}=1$).
By the Cauchy--Schwarz inequality in $\mathbb{R}^2$,
\begin{align*}
  \|\widehat\phi\|_{\widehat b}
  &\le (\mu_{k,h}^\sigma)^{-1/2}\sqrt{\alpha}
       + C_h\sqrt{A_h+\beta} \\
  &\le \sqrt{(\mu_{k,h}^\sigma)^{-1}+C_h^{\,2}}
       \;\sqrt{\alpha+A_h+\beta}
  \;=\; \sqrt{(\mu_{k,h}^\sigma)^{-1}+C_h^{\,2}}
       \;\sqrt{1+A_h}.
\end{align*}
Since $\widetilde R(\widehat\phi)
=\|\widehat\phi\|_{\widehat b}^{\,2}
\le(1{+}A_h)((\mu_{k,h}^\sigma)^{-1}+C_h^{\,2})$
and $\min_{E_k}\widetilde R=1/\mu_k^\sigma\le\widetilde R(\widehat\phi)$,
inverting gives
\eqref{eq:cecr-lb-convergent}.
\end{proof}

\begin{remark}[Convergence and rate]
\label{rmk:convergent-rate}
As the mesh-family parameter $h\to 0$: $C_h=O(h)$, $\Gamma_h\to 0$
(\Cref{rmk:Gamma-h-rate}), and $A_h\to 0$. Hence the denominator
$(1{+}A_h)(1+C_h^{2}\mu_{k,h}^\sigma)\to 1$ and
$L_k^{\mu,\sigma}\to\mu_k$: \emph{the CECR gap closes}.
If one replaces $C_h$ by the computable surrogate
$C_h^{\mathrm{PW}}\ge C_h$, the same conclusion still holds.

\emph{Rate of the CECR gap} $(\mu_{k,h}^\sigma{-}\sigma)-L_k^{\mu,\sigma}$.
\begin{itemize}
\item Uniform meshes: $\Gamma_h=O(h)$, $A_h=O(h)$; dominant
  term $C_h^2=O(h^2)$; CECR gap closes at $O(h^2)$.
\item Graded meshes ($h_K\propto h\,|x_K|^\theta$, $\theta>1/2$,
  with $h_{\max}=O(h)$):
  $\Gamma_h=O(h_{\max}^{\,2}/R)=O(h^2)$, so $A_h=O(h^2)$
  and the CECR gap closes at $O(h^2)$ in the mesh-family parameter
  in both $N=2$ and $N=3$.
\end{itemize}

\emph{Rate of the certified $\lambda_k$-bound} $L_k\le\lambda_k$
(via \Cref{lem:perturb-sigma}) additionally depends on $\varepsilon_h$.
The direct patch/off-patch estimates in \Cref{app:graded-eps} give
$\varepsilon_h=O(h^2/\kappa_\sigma)$ in three dimensions when
$\omega_0\subset B_{\beta h^2}$, and also in two dimensions when the
singular patch is chosen sufficiently small; for example
$\omega_0\subset B_{\beta h^{\gamma_{\rm p}}}$ with
$\gamma_{\rm p}>2$ and a compatible Sobolev exponent $p>1$ close to
one.  Thus the certified $\lambda_k$-bound is expected to tighten at
$O(h^2)$ under the graded-patch design described in
\Cref{rmk:2d-rate}.

In contrast, the $\gamma_h$-shift bound
\eqref{eq:cecr-optimal-shift} has
$\gamma_h=\|c_h^-\|_{L^\infty}=\Theta(h^{-1})\to\infty$.
\Cref{thm:cecr-lb-convergent} avoids this issue by
shifting the negative reaction error from the denominator
to a small additive correction $A_h$ in the numerator.
\end{remark}

\begin{remark}[2D perturbation rate and patch radius]
\label{rmk:2d-rate}
In dimension $N=2$, the direct estimate
\eqref{eq:direct-est-c-grd-2d} may be used with any exponent
$p\in(1,2)$ and the corresponding embedding
$H^1(\Omega)\hookrightarrow L^{2p/(p-1)}(\Omega)$.
If $\omega_0\subset B_{\rho_h}$ with
$\rho_h=\beta h^{\gamma_{\rm p}}$, then the patch contribution is
$O(h^{\gamma_{\rm p}(2/p-1)})$; see
\eqref{eq:graded-eps-2d}.  Hence $\varepsilon_h=O(h^2)$ is obtained by
choosing $\gamma_{\rm p}(2/p-1)\ge2$.  In particular, a patch radius
$\rho_h=\beta h^{2+\delta}$ is theoretically sufficient for any
$\delta>0$, provided $p>1$ is chosen close enough to one.  The numerical
certification does not need this asymptotic simplification: it uses the
directly computed patch norm in \eqref{eq:direct-est-c-grd-2d}.
\end{remark}

\begin{remark}[Comparison summary]
\label{rmk:comparison-table}
\Cref{tab:comparison-methods} summarises the three approaches.
\begin{table}[htbp]
\centering
\caption{Comparison of CECR lower-bound methods for sign-changing $c_h$.}
\label{tab:comparison-methods}
\small
\begin{tabular}{lccc}
\toprule
Method & Effective $C_h^{\,2}$ & Computable? & Gap under refinement \\
\midrule
$\gamma_h$-shift (\S\ref{subsec:cecr-shift})
  & $C_h^{\,2}$, shift $\gamma_h\to\infty$ & yes & diverges \\
\Cref{thm:cecr-lb-convergent}
  & $(1{+}A_h)C_h^{\,2}$, $A_h\to 0$
  & yes & $\approx O(h^2)$ \\
\bottomrule
\end{tabular}
\end{table}
\end{remark}

\begin{remark}[Recovery of the positive-$c_h$ case]
\label{rmk:convergent-positive}
When $c_h\ge 0$, one has $c_h^-=0$, hence $\Gamma_h=0$ and $A_h=0$.
\Cref{thm:cecr-lb-convergent} then reduces to
$\mu_k^\sigma\ge\mu_{k,h}^\sigma/(1+C_h^{2}\mu_{k,h}^\sigma)$,
which upon un-shifting recovers Liu's positive-$c_h$ bound
\eqref{eq:cecr-abstract}.  If one further substitutes the computable
surrogate $C_h^{\mathrm{PW}}\ge C_h$, one obtains the fully explicit
version used in the numerical tables.
\end{remark}

\section{Extension to General Unbounded Potentials}
\label{sec:extension}

We now drop the assumption that $V$ is piecewise constant and allow
general $V = V^+ - V^- \in L^{p_0}(\Omega)$ with $p_0 > N/2$.
The strategy is:
\begin{enumerate}[label=(\roman*)]
  \item replace $V$ by its piecewise constant average $c_h$
    (\S\ref{subsec:pc-approx});
  \item introduce a \emph{fixed} additive shift $\sigma$ that renders
    the CECR form $a_h^\sigma$ positive definite on $V_h^{\ECR}$
    (\S\ref{subsec:fixed-shift});
  \item bound the perturbation $\lambda_k - \mu_k$ induced by
    $V\leftrightarrow c_h$ (\S\ref{subsec:perturbation});
  \item assemble a fully computable two-sided bound
    (\S\ref{subsec:main-theorem}), using the convergent CECR lower
    bound \Cref{thm:cecr-lb-convergent} for the shifted problem.
\end{enumerate}

Throughout this section the shift $\sigma>0$ is a single fixed constant
that is independent of $h$ once the mesh is fine enough to resolve the
ground-state energy of $-\Delta+V$ (made precise in
\S\ref{subsec:fixed-shift}). In particular, $\sigma$ does \emph{not}
depend on $\gamma_h=\|c_h^-\|_{L^\infty}$. The convergent CECR
lower bound of \Cref{thm:cecr-lb-convergent} provides a computable
bound on $\mu_k(-\Delta+c_h)$ via the elementwise constant $\Gamma_h$;
the present section combines it with a perturbation lemma
to obtain a fully computable two-sided bound on $\lambda_k(-\Delta+V)$.

\subsection{Piecewise Constant Approximation}
\label{subsec:pc-approx}

\begin{definition}[Element average]
\label{def:element-average}
For $V \in L^1_{\rm loc}(\Omega)$ and $K \in \mathcal{T}^h$, define
\begin{equation}
\label{eq:element-average}
  c_h\big|_K := \frac{1}{|K|} \int_K V(x)\,\dx,
  \qquad K \in \mathcal{T}^h.
\end{equation}
\end{definition}
Write $c_h = c_h^+ - c_h^-$ with $c_h^\pm\ge 0$ piecewise constant.
No $L^\infty$ control of $c_h^-$ is required for the present analysis;
the only relevant object is the spectrum of the discrete form
$a_h$ introduced in \S\ref{sec:cecr}.

\subsection{Fixed Shift and Positive-Definiteness of \texorpdfstring{$a_h^\sigma$}{a\_h\textasciicircum sigma}}
\label{subsec:fixed-shift}

Recall from \eqref{eq:shifted-cecr} that $a_h^\sigma$ is the CECR
form with potential $c_h+\sigma$ acting in the reaction slot.  This is
precisely the restriction of the
pair-space shifted form $\widehat a^\sigma$
\eqref{eq:shifted-pair-forms} to the ECR diagonal subspace
$\widehat V_h$, so the pair-space machinery of
\S\ref{subsec:cecr-pair-space} applies directly with $c_h$
replaced by $c_h+\sigma$.  The $k$-th eigenvalue
$\mu_{k,h}^\sigma$ of the problem
$a_h^\sigma(u_h,\cdot)=\mu_{k,h}^\sigma(\cdot,\cdot)$ on
$V_h^{\ECR}$ is computable as the $k$-th CECR Ritz value for the
potential $c_h+\sigma$.

\begin{remark}[Relation to the full $L^2$-mass shift]
\label{rmk:shift-comparison}
Adding $\sigma$ to the potential in the \emph{reaction slot}
(as in \eqref{eq:shifted-cecr}) differs from adding
$\sigma(u_h,v_h)$ to the full $L^2$ mass.  For the full-mass
shift the CECR eigenvalue shifts exactly by $\sigma$; for the
reaction-slot shift one has
\[
  \mu_{k,h}^\sigma \;\le\; \mu_{k,h}+\sigma,
\]
with $\mu_{k,h}+\sigma - \mu_{k,h}^\sigma
=\sigma\,\bigl(\|u_{k,h}\|^2-\|\Pi_{0,h}u_{k,h}\|^2\bigr)/
\|u_{k,h}\|^2 = O(\sigma h^2\mu_{k,h})$ in the mesh-family parameter,
which vanishes under refinement.
Using $\mu_{k,h}^\sigma$ (rather than $\mu_{k,h}+\sigma$) in the
lower-bound formula is therefore conservative but rigorous; in
the limit the two coincide.
\end{remark}

The continuous eigenvalues of $-\Delta+(c_h+\sigma)$ satisfy
\begin{equation}
\label{eq:sigma-shift-continuous}
  \mu_k(-\Delta+c_h+\sigma) = \mu_k(-\Delta+c_h) + \sigma
  = \mu_k + \sigma.
\end{equation}

\subsection{CECR Lower Bound on \texorpdfstring{$\mu_k$}{mu\_k} and Perturbation to \texorpdfstring{$\lambda_k$}{lambda\_k}}
\label{subsec:perturbation}

Let $\{\mu_k\}$ denote the eigenvalues of $-\Delta+c_h$ and
$\{\lambda_k\}$ those of $-\Delta+V$, both on $\Omega$ with Neumann BC.
Applying \Cref{thm:cecr-lb-convergent} to the shifted problem
$-\Delta+c_h+\sigma$ (which has CECR Ritz values
$\mu_{k,h}^\sigma>0$; see \Cref{rmk:shift-comparison}) yields the
computable bound
\begin{equation}
\label{eq:convergent-LB-sigma}
  \mu_k
  \;\ge\;
  \frac{\mu_{k,h}^\sigma}
       {(1+A_h)\bigl(1+\mu_{k,h}^\sigma(C_h^{\mathrm{PW}})^{\,2}\bigr)}
  \;-\;\sigma
  \;=:\; L_k^{\mu,\sigma},
\end{equation}
where $A_h=\Gamma_h/(1{-}\epsilon)$ is defined in
\eqref{eq:A-constant}. Every quantity is computable from the mesh
data. The gap $\mu_k-L_k^{\mu,\sigma}\to 0$ under refinement
(\Cref{rmk:convergent-rate}).

Define the coercivity constant of the shifted averaged form,
\begin{equation}
\label{eq:kappa-sigma}
  \kappa_\sigma
  \;:=\; \min\bigl(1-\epsilon,\;\sigma-C_\epsilon\bigr) \;>\; 0,
\end{equation}
where $(\epsilon,C_\epsilon)$ are the form-bound parameters of
\Cref{ass:form-bounded}\,(ii)---the discrete bound on $c_h^-$, valid
uniformly in $h$ as discussed in \Cref{rmk:kato-ch-valid}---and
$\sigma>C_\epsilon$. The argument here uses item~(ii) rather than the
continuous bound (i): as recorded in \Cref{ass:form-bounded}, the
inequality $(c_h^-u,u)\le(V^-u,u)$ holds Jensen-by-Jensen only for
piecewise-constant $u$, so the continuous form bound does not transfer
to $c_h^-$ for general $u\in H^1$, which is exactly why (ii) is stated
separately.

The \emph{perturbation ratio} $\varepsilon_h$ is any computable
quantity satisfying
\begin{equation}
\label{eq:eps-abstract}
  |((V-c_h)u,u)| \;\le\; \varepsilon_h\,\kappa_\sigma\,\|u\|_{H^1(\Omega)}^2
  \qquad\forall\,u\in H^1(\Omega).
\end{equation}
For the Coulomb potentials considered here, this perturbation form
converges to zero under the graded mesh assumptions used in
\Cref{app:graded-eps}: more precisely,
\[
  \sup_{\|u\|_{H^1(\Omega)}=1}|((V-c_h)u,u)| \;\longrightarrow\;0
  \qquad (h\to0).
\]
Thus the replacement of $V$ by its element average $c_h$ is not only
computable but also asymptotically harmless in the form topology
needed for the min--max argument. Two explicit choices for
$\varepsilon_h$, tailored respectively to numerical computation and
to convergence analysis on graded meshes, are given in
\S\ref{subsec:sobolev-constants}.

\begin{lemma}[Shifted perturbation bound]
\label{lem:perturb-sigma}
Under \Cref{ass:V,ass:kato-ch}, $\varepsilon_h\le 1$ (as in
\eqref{eq:eps-abstract}), and $\sigma\ge C_\epsilon+c_0$ for some
$c_0>0$,
\begin{equation}
\label{eq:perturb-sigma}
  (1-\varepsilon_h)(\mu_k+\sigma)
  \;\le\; \lambda_k+\sigma
  \;\le\; (1+\varepsilon_h)(\mu_k+\sigma).
\end{equation}
No $L^\infty$ control of $c_h^-$ (no $\gamma_h$) appears on either side.
\end{lemma}

\begin{proof}
Introduce the shifted continuous form
\[
  a^\sigma(u,v):=a(u,v)+\sigma(u,v)
\]
and define, in one line, the shifted form obtained by replacing
$V$ with its cellwise averaged lower approximation $c_h$:
\begin{equation}
\label{eq:ah-infty-def}
  a_h^\infty(u,v)
  := a(u,v)-((V-c_h)u,v)+\sigma(u,v),
  \qquad u,v\in H^1(\Omega).
\end{equation}
Equivalently,
\begin{equation}
\label{eq:ah-infty-relation}
  a_h^\infty(u,v)
  = a^\sigma(u,v)-((V-c_h)u,v)
  =(\nabla u,\nabla v)+(c_h u,v)+\sigma(u,v),
\end{equation}
so $a_h^\infty$ is exactly the shifted form for the averaged
operator $-\Delta+c_h+\sigma$.  By \Cref{ass:kato-ch}---a direct form bound on the
\emph{discrete} negative part $c_h^-$ that holds uniformly in $h$
(\Cref{rmk:kato-ch-valid})---we have
$(c_h^-u,u)\le\epsilon\|\nabla u\|^2+C_\epsilon\|u\|^2$, hence
$(c_h u,u)\ge-(c_h^-u,u)\ge-\epsilon\|\nabla u\|^2-C_\epsilon\|u\|^2$,
so
\begin{equation}
\label{eq:a-h-infty-coercive}
  a_h^\infty(u,u)
  = \|\nabla u\|^2 + ((c_h+\sigma)u,u)
  \ge (1-\epsilon)\|\nabla u\|^2 + (\sigma-C_\epsilon)\|u\|^2
  \ge \kappa_\sigma\,\|u\|_{H^1}^{\,2},
\end{equation}
with $\kappa_\sigma$ from \eqref{eq:kappa-sigma}.
By the defining property \eqref{eq:eps-abstract} of $\varepsilon_h$
and \eqref{eq:a-h-infty-coercive},
\[
  |((V-c_h)u,u)|
  \;\le\; \varepsilon_h\,\kappa_\sigma\,\|u\|_{H^1}^{\,2}
  \;\le\; \varepsilon_h\,a_h^\infty(u,u).
\]
Using \eqref{eq:ah-infty-relation}, this gives the form comparison
\[
  (1-\varepsilon_h)\,a_h^\infty(u,u)
  \le a^\sigma(u,u)
  \le (1+\varepsilon_h)\,a_h^\infty(u,u).
\]
Application of the min-max principle to the shifted eigenvalues
$\lambda_k+\sigma$ and $\mu_k+\sigma$ then yields
\eqref{eq:perturb-sigma} by the standard test-space argument.
\end{proof}

\subsection{Explicit Graded-Mesh Bounds for \texorpdfstring{$\varepsilon_h$}{epsilon\_h}}
\label{subsec:sobolev-constants}

Let $\mathcal T_h$ be a simplicial mesh of $\Omega$ indexed by a
positive mesh-family parameter $h$, and set
$h_{\max}:=\max_K h_K$.  The parameter $h$ controls the grading law
below and is not assumed to equal $h_{\max}$.  For a single Coulomb
singularity at the origin,
define
\[
  r_K:=\min_{x\in K}|x|,
  \qquad
  \omega_0:=\bigcup_{\{K:\,0\in\overline K\}}K.
\]
For a singularity at $a$, the same definitions are used with
$|x|$ replaced by $|x-a|$; for a multi-centre potential they are
applied locally at each centre.

\begin{assumption}[Graded mesh near a Coulomb singularity]
\label{ass:graded-mesh}
There exist constants $\vartheta>0$ and $\beta>0$, independent of the
mesh level $h$, and a patch exponent $\gamma_{\rm p}>0$,
such that
\begin{itemize}
  \item[(G1)] \emph{Linear grading away from the singularity:}
    $h_K\le\vartheta\,h\,r_K$ for every $K$ with $0\notin\overline K$.
  \item[(G2)] \emph{Small singular patch:}
    $\omega_0\subset B_{\beta h^{\gamma_{\rm p}}}$.
\end{itemize}
\end{assumption}

The constants $\vartheta$ and $\beta$ quantify, respectively, the
linear grading away from the singularity and the size of the singular
patch.  In the rate estimates below we take
$\gamma_{\rm p}=2$ in three dimensions.  In two dimensions
$\gamma_{\rm p}$ is a mesh-design parameter; choosing
$\gamma_{\rm p}>2$ and a compatible Sobolev exponent gives the expected
$O(h^2)$ perturbation rate.  The numerical certification itself uses
the directly computable patch norm, so it does not require replacing
that norm by a separate asymptotic upper bound.

For $V(x)=|x|^{-1}$, the derivation in \Cref{app:graded-eps} uses the
cell-average cancellation identity \eqref{eq:Vchuident} and splits
$((V-c_h)u,u)$ into off-patch and patch contributions.  On off-patch
elements, the cancellation is combined with the local $L^1$ Poincar\'e
inequality and the mean-value bound $|\nabla |x|^{-1}|=|x|^{-2}$.
This gives the cellwise factor
\[
  \frac12\sup_{\{K:\,0\notin\overline K\}}\frac{h_K^2}{r_K^2}
\]
in \eqref{eq:middle-way}.  On the singular patch, H\"older--Sobolev is
applied with $(p,q)=(3/2,3)$ in three dimensions and with any
$p\in(1,2)$, $q=p/(p-1)$ in two dimensions; Jensen's inequality
controls the element averages $c_h$.  The required bounded-domain
constant $C_6$ is recorded in \eqref{eq:C6-box-line} for boxes and in
\Cref{rmk:c6-ball} for balls.  The corresponding two-dimensional
$H^1\hookrightarrow L^8$ constants are recorded in
\Cref{rmk:rect-L8}.

Retaining the off-patch and patch contributions separately gives the
directly computable estimates \eqref{eq:direct-est-c-grd} in three
dimensions and \eqref{eq:direct-est-c-grd-2d} in two dimensions.
Dividing either right-hand side by $\kappa_\sigma$ gives an admissible
value of $\varepsilon_h$ in \eqref{eq:eps-abstract}.  Under
\Cref{ass:graded-mesh}, these estimates imply the explicit rates
\eqref{eq:graded-eps-3d} and \eqref{eq:graded-eps-2d}.  In particular,
choosing $\gamma_{\rm p}=2$ in three dimensions and
$\gamma_{\rm p}(2/p-1)\ge2$ in two dimensions gives
\begin{equation}
\label{eq:eps-graded-rate}
  \varepsilon_h
  \;\le\;
  \begin{cases}
    C_{\mathrm{grd}}\,h^2/\kappa_\sigma,
      & N=3,\quad C_{\mathrm{grd}}\text{ as in }\eqref{eq:Cgrd-3d},\\[2mm]
    C_{\mathrm{grd}}^{2D}\,h^2/\kappa_\sigma,
      & N=2,\quad C_{\mathrm{grd}}^{2D}\text{ as in }\eqref{eq:Cgrd-2d}.
  \end{cases}
\end{equation}
Thus $\varepsilon_h\to0$ as $h\to0$, with an expected second-order
rate when the singular patch is shrunk according to the above rule.

\subsection{Main Theorem}
\label{subsec:main-theorem}

\paragraph{Preparation for the theorem}
Let $\Omega=D(R)$ be the truncated domain and let $\lambda_k$ denote
the $k$-th Neumann eigenvalue of $-\Delta+V$ on $\Omega$
(i.e.\ $\lambda_k=\lambda_k^{R,N}$ in the notation of
\S\ref{subsec:truncation}).
Let $V \in L^{p_0}(\Omega)$ with $p_0 > N/2$ satisfy
\Cref{ass:V,ass:form-bounded}.  Choose a fixed shift $\sigma>0$ such
that
\begin{equation}
\label{eq:main-shift-assumptions}
  a_h^\sigma \text{ is positive definite on } V_h^{\ECR},
  \qquad
  \sigma\ge C_\epsilon+c_0
  \quad\text{for some }c_0>0 .
\end{equation}
The first condition is checked a posteriori by
$\mu_{1,h}^\sigma>0$, and the second gives the coercivity
\eqref{eq:a-h-infty-coercive} through \Cref{ass:kato-ch}.
Let $\varepsilon_h$ satisfy \eqref{eq:eps-abstract}; in applications it
is obtained from the computable estimates in
\Cref{subsec:sobolev-constants}.  Finally, let
$L_k^{\mu,\sigma}$ be the CECR lower bound in
\eqref{eq:convergent-LB-sigma}, and let
$\lambda_{k,h}^{\CG}$ be the conforming Galerkin Ritz value on the same
mesh.

\begin{theorem}[Fixed-shift two-sided bound on $\Omega$]
\label{thm:main}
Assume the preparation above and suppose that $\varepsilon_h\le 1$.
Then
\begin{equation}
\label{eq:main-bound}
  \underbrace{(1-\varepsilon_h)(L_k^{\mu,\sigma}+\sigma)
    - \sigma}_{=:\,L_k}
  \;\le\; \lambda_k
  \;\le\;
  \underbrace{\lambda_{k,h}^{\CG}}_{=:\,U_k}.
\end{equation}
\end{theorem}

\begin{proof}
Since $1-\varepsilon_h\ge 0$, the lower inequality of
\Cref{lem:perturb-sigma} together with $\mu_k\ge L_k^{\mu,\sigma}$
from \eqref{eq:convergent-LB-sigma} (a consequence of
\Cref{thm:cecr-lb-convergent}) yields
$(1-\varepsilon_h)(L_k^{\mu,\sigma}+\sigma) \le
\lambda_k+\sigma$, which is \eqref{eq:main-bound} after subtracting
$\sigma$.  The upper bound is the standard Galerkin min-max.
\end{proof}

\begin{remark}[Interpretation of \Cref{thm:main}]
\label{rmk:main-theorem-interpretation}
The theorem is stated for the Neumann eigenvalue $\lambda_k^{R,N}$ on
the truncated domain.  By \Cref{lem:neumann-lb},
$\lambda_k^{R,N}\le\lambda_k^{\R^N}$ whenever the confinement condition
holds; hence $L_k$ is also a certified lower bound for the corresponding
full-space Schr\"odinger eigenvalue.  The truncation gap
$\lambda_k^{\R^N}-\lambda_k^{R,N}\ge0$ is discussed in
\S\ref{subsec:truncation-error}.  Both $L_k$ and
$U_k=\lambda_{k,h}^{\CG}$ are computable from the mesh data, the CECR
eigenvalues, the conforming Galerkin solve, and $\varepsilon_h$; no
$L^\infty$ shift $\gamma_h=\|c_h^-\|_{L^\infty}$ enters the final
enclosure.
\end{remark}

\subsection{Convergence under Mesh Refinement}
\label{subsec:convergence}

Under \Cref{ass:V,ass:form-bounded} and the fixed-shift
hypotheses of \Cref{thm:main}, the perturbation parameter satisfies
$\varepsilon_h\to 0$ as the mesh-family parameter $h\to 0$.
On linearly graded meshes satisfying \Cref{ass:graded-mesh}, the explicit rates
are $\varepsilon_h=O(h^2/\kappa_\sigma)$ in both dimensions, with the
two-dimensional patch exponent chosen as in \Cref{rmk:2d-rate}; see
\eqref{eq:eps-graded-rate}.
The conforming Galerkin upper bound
$U_k=\lambda_{k,h}^{\CG}$ converges to $\lambda_k$ at the standard
rate for shape-regular families adapted to $V$. The convergence of
the CECR lower-bound component $L_k^{\mu,\sigma}$ is guaranteed by
\Cref{thm:cecr-lb-convergent}: the correction factor
$A_h=\Gamma_h/(1{-}\epsilon)\to 0$ and
$(C_h^{\mathrm{PW}})^2\mu_{k,h}^\sigma\to 0$, so
$L_k^{\mu,\sigma}\to\mu_k$ and consequently $L_k\to\lambda_k$.

\begin{remark}[Limitation of the $\gamma_h$-shift]
\label{rmk:sigma-equals-gamma}
The present framework reduces to \Cref{rmk:gamma-h-limitation} if one
sets $\sigma=\gamma_h$ and uses the purely kinetic constant $C_h$:
the resulting lower bound is
$L_k^{\gamma_h}:=
(\mu_{k,h}+\gamma_h)/\bigl(1+(\mu_{k,h}+\gamma_h)C_h^{\,2}\bigr)
-\gamma_h$.  Setting $A:=\mu_{k,h}+\gamma_h$, the gap to the Ritz
value $\mu_{k,h}$ admits the exact expression
\begin{equation*}
  \mu_{k,h}-L_k^{\gamma_h}
  \;=\; A-\frac{A}{1+A\,C_h^{\,2}}
  \;=\; \frac{A^{\,2}\,C_h^{\,2}}{1+A\,C_h^{\,2}},
\end{equation*}
so that for $\gamma_h\gg\mu_{k,h}$ the leading order is
$A^{\,2}C_h^{\,2}\approx\gamma_h^{\,2}C_h^{\,2}$---the factor
$\gamma_h$ appears twice (once from the numerator shift, once from
the $A$ inside the Rayleigh-quotient gap). For
$V\sim-1/|x-x_0|$ one has $\gamma_h=\Theta(1/h_{\min})$ and
$C_h=\Theta(h_{\max})$, so the gap satisfies
$\gamma_h^{\,2}C_h^{\,2}=\Theta(h_{\max}^{\,2}/h_{\min}^{\,2})$,
which does not vanish on uniform meshes and \emph{worsens} under
radial grading $h_K\propto|x|^\theta$, $\theta>0$ (since
$h_{\min}\to 0$ faster than $h_{\max}$). The classical
$\gamma_h$-shift therefore fails to produce a convergent lower
bound for Coulomb singularities in both 2D and 3D.
\end{remark}

\section{Two-Centered Coulomb Potential}
\label{sec:twocenter}

\subsection{Problem Formulation}
\label{subsec:twocenter-formulation}

Let $a_1, a_2 \in \R^N$ with $|a_1 - a_2| = R_0 > 0$ (internuclear
distance) and $Z_1, Z_2 > 0$ (nuclear charges).  The two-centered
Coulomb potential is
\begin{equation}
\label{eq:two-center-potential}
  V(x) = -\frac{Z_1}{|x - a_1|} - \frac{Z_2}{|x - a_2|}.
\end{equation}
In 3D, this models the Born--Oppenheimer Hamiltonian of a diatomic
molecule with fixed nuclei at $a_1$, $a_2$ and a single electron;
the most studied case is $\mathrm{H}_2^+$ ($N=3$, $Z_1=Z_2=1$).
The potential \eqref{eq:two-center-potential} satisfies
$V \le 0$, $V \in L^p_{\rm loc}(\R^N)$ for all $p < N$, and
$V(x) \to 0$ as $|x| \to \infty$. (Note: $V \notin L^p(\R^N)$ for any
$p\le N$ because the $|x|^{-1}$ tail fails $L^p$-integrability at
infinity; on the truncated domain $D(R)$ however
$V \in L^p(D(R))$ for every $p<N$.)
By Kato--Rellich, $H = -\Delta + V$ is self-adjoint on $H^2(\R^N)$ and
bounded below; $\sigma_{\rm ess}(H) = [0,\infty)$, with infinitely
many negative discrete eigenvalues accumulating at $0$.

\subsection{Integrability and Admissible Shift}
\label{subsec:twocenter-params}

\paragraph{Exact eigenvalues (reference values)}
For a single centre ($Z_2=0$), the operator $H=-\Delta-Z_1/r$ on $\R^N$
has exact ground state:
\begin{equation}
\label{eq:exact-gs}
  \lambda_1 =
  \begin{cases}
    -Z_1^2 & N=2 \quad (\text{eigenfunction } e^{-Z_1 r}), \\
    -Z_1^2/4 & N=3 \quad (\text{standard hydrogen atom}).
  \end{cases}
\end{equation}
For two centres ($Z_1=Z_2=1$, 3D), no closed
form is available.  The rescaling $x=2r$ maps our operator
$-\Delta-Z/|x{-}a_i|$ with internuclear distance $d=|a_1-a_2|$ to
$\tfrac{1}{2}\bigl(-\tfrac{1}{2}\nabla_r^{\,2}-Z/|r-R_i|\bigr)$
with physical internuclear distance $R_0=d/2$; hence
$\lambda_k(-\Delta+V;\,d)=E_k(-\tfrac{1}{2}\nabla^2+V;\,R_0{=}d/2)/2$.
We take $d=4\,\mathrm{bohr}$ (nuclei at $a_{1,2}=(\mp 2,0,0)$),
corresponding to $R_0=2\,\mathrm{bohr}$, the equilibrium internuclear
distance of $\mathrm{H}_2^+$.  The electronic ground-state energy at
equilibrium is $E_{\rm el}(R_0{=}2)=-1.1026\,\mathrm{Hartree}$
\cite{Bates1953,MadsenPeek1971}, giving
$\lambda_1= E_{\rm el}(2)/2=-0.5513$
in our $-\Delta$ convention.

\paragraph{Choice of \texorpdfstring{$p_0$}{p0}}
Since each $Z_i/|x-a_i| \in L^p$ for $p < N$, the potential
\eqref{eq:two-center-potential} satisfies $V \in L^{p_0}(D(R))$
for any $p_0 < N$.  We take
\begin{itemize}
  \item $N=2$: $p_0 = 4/3$ in the reported computations; for the
    theoretical patch-radius rate in \Cref{rmk:2d-rate}, any
    $p_0\in(1,2)$ may be used,
  \item $N=3$: $p_0 = 3/2$ (so $q_0=3$, the critical Sobolev exponent).
\end{itemize}

\paragraph{Admissible shift \texorpdfstring{$\sigma$}{sigma}}
The shift $\sigma$ is chosen separately for each numerical family
according to the available form-bound constant $C_\epsilon$; the
specific values are reported in the parameter paragraphs of
\S\ref{sec:numerics}. No global Sobolev-norm smallness condition on
$V^-$ is required. Positive definiteness is verified a posteriori on
each mesh by the equivalent check $\mu_{1,h}^\sigma>0$ (minimum CECR
Ritz value for potential $c_h+\sigma$ is positive). This holds on all
tested meshes.

\paragraph{Perturbation \texorpdfstring{$\varepsilon_h$}{epsilon\_h}}
The admissible perturbation bounds satisfying \eqref{eq:eps-abstract}
are obtained as in \S\ref{subsec:sobolev-constants}.  For
\emph{numerical computation}, we use the mesh-local estimates
\eqref{eq:direct-est-c-grd} and \eqref{eq:direct-est-c-grd-2d}, divided
by $\kappa_\sigma=\min(1-\epsilon,\sigma-C_\epsilon)$ in all
experiments.  For two-centre potentials, the patch and off-patch
contributions are summed over the two singularities.  For
\emph{convergence analysis} on linearly graded meshes satisfying
\Cref{ass:graded-mesh}, the single-centre estimates
\eqref{eq:direct-est-c-grd} and \eqref{eq:direct-est-c-grd-2d} apply locally at
each Coulomb centre and give the same rates:
\begin{equation}
\label{eq:eps-h-twocenter}
  \varepsilon_h \;\le\; \frac{C_{\mathrm{grd}}\,h^2}{\kappa_\sigma},
\end{equation}
provided the two-dimensional patch exponent and Sobolev exponent are
chosen as in \Cref{rmk:2d-rate}.
The graded meshes used in \S\ref{sec:numerics} satisfy
\Cref{ass:graded-mesh} near each singularity.

\subsection{Symmetry and Mesh Design}
\label{subsec:twocenter-symmetry}

\paragraph{Symmetric case}
When $Z_1 = Z_2 =: Z$ and $a_1 = -a_2 =: \frac{R_0}{2}\,e_1$
(nuclei on the $x_1$-axis, symmetric about the origin), the
Hamiltonian commutes with the reflection $x \mapsto -x$; eigenfunctions
are either \emph{gerade} ($\sigma_g$, even) or \emph{ungerade}
($\sigma_u$, odd), which can be exploited to halve the computational
domain.  The present implementation does not use this symmetry.

\paragraph{Mesh design near singularities}
The Coulomb singularities at $a_1$, $a_2$ call for local refinement.
We use linearly graded meshes satisfying \Cref{ass:graded-mesh} near
each singularity, which give the convergence rate
\eqref{eq:eps-h-twocenter} for the perturbation ratio.
Convergence of $L_k^{\mu,\sigma}$ itself (\Cref{thm:cecr-lb-convergent})
holds on any shape-regular family without any additional grading
hypothesis.

\subsection{Truncation Error for the Two-Centered Problem}
\label{subsec:twocenter-truncation}

For the purely attractive potential \eqref{eq:two-center-potential},
$V(x) \to 0$ as $|x| \to \infty$ and all discrete eigenvalues satisfy
$\lambda_k < 0$.  By Agmon's theory \cite{Agmon1982}, eigenfunctions
decay exponentially with rate $\sqrt{|\lambda_k|}$, so the truncation
error \emph{scales} as $\delta_k(R)\sim e^{-2\sqrt{|\lambda_k|}\,R}$
as $R\to\infty$. We emphasise that this scaling is heuristic: Agmon's
theorem provides exponential decay of $u_k$ but the constant relating
this to $\delta_k(R)$ depends on un-quantified prefactors; we use the
scaling here only to motivate the choice of $R$. For the
$\mathrm{H}_2^+$ ground state in 3D, $\lambda_1\approx-0.5513$ gives
decay rate $\sqrt{|\lambda_1|}\approx 0.742$ and the heuristic
$e^{-2\cdot 0.742\cdot 8}\approx 7\times 10^{-6}$ at $R=8$ suggests
the truncation error is well below the FEM approximation error on the
meshes used. An \emph{explicit, certified} computable upper bound on
$\delta_k(R)$ with all constants tracked is left for a companion paper.

\subsection{Summary: Two-Sided Bounds for Two-Centered Coulomb}
\label{subsec:twocenter-summary}

Combining the above with the main theorem (\Cref{thm:main}):

\begin{corollary}[Two-sided bounds for $\mathrm{H}_2^+$-type problems]
\label{cor:twocenter-bounds}
Let $V$ be as in \eqref{eq:two-center-potential},
$\Omega=D(R)$ with $R$ chosen so that $\delta_k(R)$ is negligible, and let $\sigma$ be the
fixed, mesh-independent shift of \S\ref{subsec:twocenter-params}.
With $L_k^{\mu,\sigma}$ from \Cref{thm:cecr-lb-convergent} and
$\varepsilon_h$ satisfying \eqref{eq:eps-abstract} (for instance,
the computable mesh-local bounds \eqref{eq:direct-est-c-grd} and
\eqref{eq:direct-est-c-grd-2d}, or the graded-rate bound
\eqref{eq:eps-graded-rate})
computed on a mesh $\mathcal{T}^h$ satisfying $\varepsilon_h\le 1$,
\begin{equation}
\label{eq:twocenter-final-bound}
  \underbrace{(1-\varepsilon_h)(L_k^{\mu,\sigma}+\sigma)
  - \sigma}_{=:\,L_k}
  \;\le\; \lambda_k^R
  \;\le\; \lambda_k
  \;\le\;
  \underbrace{\lambda_{k,h}^{\CG}}_{=:\,U_k},
\end{equation}
where $\lambda_k^R$ is the $k$-th eigenvalue on $D(R)$.
Whenever the confinement condition
$\sigma_{\mathrm{ext}}(D(R)) > \lambda_k$ of
\Cref{lem:neumann-lb} additionally holds---as verified for
specific $(k,R)$ pairs in \Cref{rmk:confinement-verify}---one
also has the Neumann-truncation sandwich
$\lambda_k^{R,N} \le \lambda_k \le \lambda_k^R$,
with $\delta_k(R) := \lambda_k^R - \lambda_k$ exponentially
small in $R$.  All quantities in
$L_k$ are computable from the mesh data ($C_h^{\mathrm{PW}}$,
$\Gamma_h$, $\varepsilon_h$), the shifted CECR Ritz values $\mu_{k,h}^\sigma$,
and the fixed shift $\sigma$.
\end{corollary}

\section{Numerical Experiments}
\label{sec:numerics}

We report numerical experiments on four problems with Coulomb-type
potentials: the 2D single-centre $V=-1/|x|$, the 2D two-centre
$V=-1/|x-a_1|-1/|x-a_2|$, the 3D hydrogen atom ($V=-1/|x|$), and
the 3D $\mathrm{H}_2^+$ molecular ion
($V=-1/|x-a_1|-1/|x-a_2|$).
For all four problems, we compute the explicit CECR lower bound
$L_k^{\mu,\sigma}$ of \Cref{thm:cecr-lb-convergent} with explicit
values of $C_h^{\mathrm{PW}}$, $\Gamma_h$, $A_h$, confirming that
the CECR gap $(\mu_{k,h}^\sigma{-}\sigma)-L_k^{\mu,\sigma}$ closes at
rate $O(h^2)$ on graded meshes (both 2D and 3D);
the certified $\lambda_k$-bound tightens at $O(h^2)$ in 3D and
$O(h)$ in 2D (\Cref{rmk:2d-rate}).
Conforming $P_1$ Galerkin upper bounds $\lambda_{k,h}^{\CG}$ are
computed on the 3D meshes for comparison.

\subsection{Implementation}
\label{subsec:implementation}

\paragraph{Pipeline}
For each mesh level, the same pipeline is
executed:
\begin{enumerate}[label=(\roman*)]
  \item The piecewise-constant element average $c_h|_K=|K|^{-1}\!\int_K V$
    is computed by a closed-form per-element polar quadrature for the
    2D Coulomb and a Duffy-transform tetrahedral quadrature for the
    3D Coulomb, so that $c_h|_K$ is accurate to machine precision
    for elements containing a singularity.
  \item The conforming $P_1$ Galerkin matrices for $-\Delta+V$ on the
    same mesh are assembled using the same high-accuracy quadrature for
    the Coulomb entries (with a positivising shift $\sigma_{\CG}$ used
    only as an eigensolver convenience).
  \item The shifted CECR Ritz values $\mu_{k,h}^\sigma$ of
    $-\Delta+(c_h+\sigma)$ are obtained by shift-invert Arnoldi
    (\texttt{eigs}); the reported column $\mu_{k,h}^\sigma{-}\sigma$
    satisfies $\mu_{k,h}^\sigma{-}\sigma\le\mu_{k,h}$ with gap
    $O(\sigma h^2)$ in the nominal mesh parameter.
    The a-posteriori admissibility check $\mu_{1,h}^\sigma>0$ is read
    directly from the smallest Ritz value of the shifted CECR solve;
    no additional eigenvalue computation is required.
  \item The conforming Galerkin eigenvalues $\lambda_{k,h}^{\CG}$
    are obtained in the same way from the $P_1$ system.
\end{enumerate}

\paragraph{Form-bound constants}
Two types of estimates for $C_\epsilon$ are used.  The first is a
purely analytic estimate obtained from explicit Hardy, H\"older, and
Sobolev-type inequalities; the concrete constants used here are
recorded in \Cref{app:ceps-analytic}.  The second is the sharper
certified computation of the optimal form constant described in
\Cref{app:ceps-eta}.  We keep both in the numerical section: the
analytic constants give a fully explicit baseline, while the optimized
constants show how strongly the final enclosure depends on a sharp
choice of $C_\epsilon$ and hence on the admissible shift
$\sigma>C_\epsilon$.

\paragraph{Mesh families}
Meshes are generated by distance-graded Delaunay triangulation (2D)
or tetrahedralisation (3D), with all nuclei $a_i$ enforced as exact
mesh vertices. The 2D computations use the linear grading stated in
\Cref{subsec:num-coulomb-2d,subsec:num-twocenter-2d}. The reported
3D computations use a nucleus patch of radius $O(h^2)$ together with
distance-based refinement away from the singularities; the certified
quantities in the tables are evaluated directly from the resulting
meshes, rather than inferred from an asymptotic grading formula.

\subsection{Single Coulomb Potential, 2D}
\label{subsec:num-coulomb-2d}

\paragraph{Setup}
$V(x) = -1/|x|$ on $Q=[-5,5]^2\subset\R^2$, Neumann BC for the CECR lower
bound (ECR space, natural boundary conditions; see \Cref{rmk:bc-convention}),
Dirichlet BC for the $\mathrm{P}_1$ upper bound.
The operator is $H = -\Delta - 1/|x|$, with exact eigenvalues
$\lambda_n = -1/(2n-1)^2$ on $\R^2$:
$\lambda_1 = -1$, $\lambda_2 = \lambda_3 = -1/9 \approx -0.1111$.
The mesh is a linearly-graded Delaunay triangulation
with the singularity at the origin enforced as an exact mesh vertex.
The grading law has the form $h_K\le\vartheta h r$, with $r=|x|$ and
$h$ the nominal mesh parameter.
The first-ring radius $r_0$ is chosen as the \emph{optimal-patch radius}:
it minimises $d_h(\rho)=c_0(\rho)+c_*(\rho)$, the sum of the patch and
off-patch contributions to $\varepsilon_h\kappa_\sigma$ (see
\Cref{rmk:2d-rate}).
This gives $\varepsilon_h=d_h^{\mathrm{opt}}/\kappa_\sigma=O(h^2)$,
in contrast to the fixed $r_0=\beta h^2$ (fixed $\beta$) approach
which gives $\varepsilon_h=O(h)$ in 2D.
\Cref{fig:mesh-case7} shows a sample mesh at level $h=0.4$.

\begin{figure}[htbp]
\centering
\includegraphics[width=0.85\textwidth]{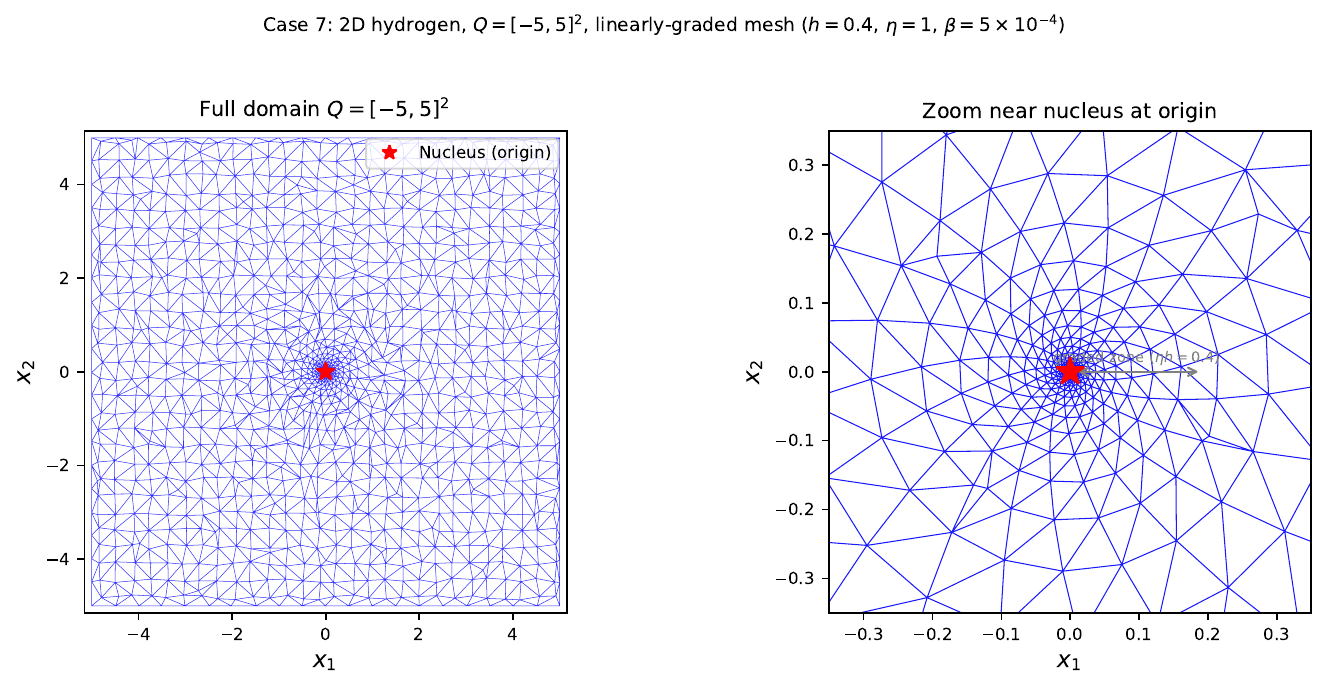}
\caption{Linearly-graded Delaunay triangulation for Case~7
(2D hydrogen, $Q=[-5,5]^2$, $h=0.4$, $\vartheta=1$, optimal-patch radius).
Left: full domain; right: zoom near the nucleus at the origin (red star).
The grading law $h_K\le\vartheta h r$ concentrates elements near the singularity;
the first-ring radius is the optimal-patch radius $\rho_{\rm opt}$ that minimises $d_h(\rho)$
(see \Cref{rmk:2d-rate}).}
\label{fig:mesh-case7}
\end{figure}

\paragraph{Parameters}
$\sigma=12.5$, $\epsilon=0.55$, $C_\epsilon=12$,
$\kappa_\sigma=\min(1-\epsilon,\,\sigma-C_\epsilon)=0.45$,
$S_8({[-5,5]}^2)\le1.466$.
The form-bound constants $\epsilon=0.55$ and $C_\epsilon=12$ are taken from the
explicit estimate $\int_Q c_h^-\,u^2\,dx \le 0.55\|\nabla u\|^2+12\|u\|^2$,
which requires $\sigma>C_\epsilon=12$; we take $\sigma=12.5$.
Positive-definiteness of $a_h^\sigma$ is confirmed by $\mu_{1,h}^\sigma>0$.
The Payne--Weinberger constant is $C_h^{\mathrm{PW}}=h_{\max}/\pi$.
The perturbation parameter is $\varepsilon_h=d_h^{\mathrm{opt}}/\kappa_\sigma$,
where
\[
  d_h^{\mathrm{opt}}
  =
  \min_{\rho>0}\{c_0(\rho)+c_*(\rho)\}.
\]
Here \(c_0(\rho)\) and \(c_*(\rho)\) are respectively the patch and off-patch
contributions in the direct estimate \eqref{eq:direct-est-c-grd-2d}.

\Cref{tab:coulomb-2d-cecr} reports the shifted Ritz value $\mu_{k,h}^\sigma{-}\sigma$,
the explicit lower bound $L_k^{\mu,\sigma}$ of \eqref{eq:cecr-lb-convergent},
and the certified corrected bound $L_k$ and $\mathrm{P}_1$ upper bound
$\lambda_{k,h}^{P_1}$ for $k=1,2,3$.

\begin{table}[htbp]
\caption{2D hydrogen ($V=-1/|x|$) on $Q=[-5,5]^2$,
Neumann BC (CECR lower bounds) / Dirichlet BC ($\mathrm{P}_1$ upper bounds),
linearly-graded Delaunay mesh with optimal-patch radius
(see \Cref{subsec:num-coulomb-2d}),
$\sigma=12.5$, $\epsilon=0.55$, $C_\epsilon=12$, $\kappa_\sigma=0.45$:
shifted Ritz value $\mu_{k,h}^\sigma{-}\sigma$,
default lower bound $L_k^{\mu,\sigma}$,
certified corrected bounds
$L_k^{(A)}=(1-\varepsilon_h^{(A)})(L_k^{\mu,\sigma}{+}\sigma)-\sigma$,
$L_k^*=(1-\varepsilon_h^*)(L_k^{\mu,\sigma}{+}\sigma)-\sigma$
(\Cref{lem:perturb-sigma}),
and the optimal-shift bound
$L_k^{\rm opt}=(1-\varepsilon_h^{\rm opt})(L_k^{\mu,\sigma_{\rm opt}}+\sigma_{\rm opt})-\sigma_{\rm opt}$
with $\sigma_{\rm opt}=C_\epsilon^{\rm cert}+\kappa_\sigma$,
and $\mathrm{P}_1$ Galerkin upper bound $\lambda_{k,h}^{P_1}$, for $k=1,2,3$.
$\varepsilon_h^{(A)}$: global $L^{p_0}$ norm estimate (Choice~A);
$\varepsilon_h^*=d_h^{\mathrm{opt}}/\kappa_\sigma$: optimal-patch bound (see \Cref{rmk:2d-rate});
$C_\epsilon^{\rm cert}$: certified form-bound constant from P1 diagnostic (\Cref{rmk:ceps-sharp});
$\varepsilon_h^{\rm opt}=d_h^{\mathrm{opt}}/\kappa_{\sigma_{\rm opt}}$ with $\kappa_{\sigma_{\rm opt}}=\min(1-\epsilon,\,\sigma_{\rm opt}-C_\epsilon^{\rm cert})$.
All levels certified ($\varepsilon_h^{(A)},\,\varepsilon_h^*<\kappa_\sigma=0.45$).
Exact: $\lambda_1=-1$, $\lambda_2=\lambda_3=-1/9\approx-0.1111$.}
\label{tab:coulomb-2d-cecr}
\centering
\footnotesize
\setlength{\tabcolsep}{2pt}
\resizebox{\textwidth}{!}{\begin{tabular}{ccrrccccc|rrrrrr}
\toprule
$h$ & $N_{\mathrm{tri}}$ & $h_{\max}$ &
$\Gamma_h$ & $A_h$ & $\varepsilon_h^{(A)}$ & $\varepsilon_h^{*}$ & $C_\epsilon^{\rm cert}$ &
$k$ & $\mu_{k,h}^\sigma{-}\sigma$ & $L_k^{\mu,\sigma}$ & $L_k^{(A)}$ & $L_k^{*}$ & $L_k^{\rm opt}$ & $\lambda_{k,h}^{P_1}$ \\
\midrule
\multirow{3}{*}{$0.4$}
  & \multirow{3}{*}{45\,340}
  & \multirow{3}{*}{$0.5362$}
  & \multirow{3}{*}{$0.00564$}
  & \multirow{3}{*}{$0.01254$}
  & \multirow{3}{*}{$0.2764$}
  & \multirow{3}{*}{$0.003493$}
  & \multirow{3}{*}{$0.8195$}
  & 1 & $-1.0055$ & $-3.9959$ & $-6.3466$ & $-4.0256$ & $-1.0086$ & $-0.9983$ \\
& & & & & & & & 2 & $-0.1642$ & $-3.5380$ & $-6.0153$ & $-3.5693$ & $-0.2162$ & $\phantom{-}0.0959$ \\
& & & & & & & & 3 & $-0.1642$ & $-3.5380$ & $-6.0153$ & $-3.5693$ & $-0.2162$ & $\phantom{-}0.0960$ \\
\midrule
\multirow{3}{*}{$0.2$}
  & \multirow{3}{*}{169\,538}
  & \multirow{3}{*}{$0.2745$}
  & \multirow{3}{*}{$0.001522$}
  & \multirow{3}{*}{$0.003383$}
  & \multirow{3}{*}{$0.1416$}
  & \multirow{3}{*}{$0.0008908$}
  & \multirow{3}{*}{$0.8067$}
  & 1 & $-1.0021$ & $-1.9658$ & $-3.4569$ & $-1.9751$ & $-1.0028$ & $-0.9989$ \\
& & & & & & & & 2 & $-0.1628$ & $-1.2630$ & $-2.8536$ & $-1.2730$ & $-0.1766$ & $\phantom{-}0.0952$ \\
& & & & & & & & 3 & $-0.1628$ & $-1.2630$ & $-2.8536$ & $-1.2730$ & $-0.1766$ & $\phantom{-}0.0952$ \\
\midrule
\multirow{3}{*}{$0.1$}
  & \multirow{3}{*}{636\,136}
  & \multirow{3}{*}{$0.1389$}
  & \multirow{3}{*}{$0.0003888$}
  & \multirow{3}{*}{$0.0008639$}
  & \multirow{3}{*}{$0.07191$}
  & \multirow{3}{*}{$0.0002244$}
  & \multirow{3}{*}{$0.8054$}
  & 1 & $-1.0012$ & $-1.2638$ & $-2.0718$ & $-1.2663$ & $-1.0014$ & $-0.9991$ \\
& & & & & & & & 2 & $-0.1624$ & $-0.4635$ & $-1.3290$ & $-0.4662$ & $-0.1660$ & $\phantom{-}0.0951$ \\
& & & & & & & & 3 & $-0.1624$ & $-0.4635$ & $-1.3290$ & $-0.4662$ & $-0.1660$ & $\phantom{-}0.0951$ \\
\bottomrule
\end{tabular}}
\end{table}

\paragraph{Discussion}
All three mesh levels are certified with both bounds ($\varepsilon_h^{(A)},\varepsilon_h^*<\kappa_\sigma=0.45$).
The optimal-patch bound $L_k^*$ is substantially tighter than the global Choice-A bound $L_k^{(A)}$:
for $k=1$ at $h=0.4$, $L_1^{(A)}=-6.347$ vs $L_1^*=-4.026$ (reduction of $2.32$),
and at $h=0.1$, $L_1^{(A)}=-2.072$ vs $L_1^*=-1.266$ (reduction of $0.806$).
The $L_k^*$ interval for $k=1$ narrows with refinement---from
$[-4.03,\,-0.998]$ at $h=0.4$ to $[-1.27,\,-0.999]$ at $h=0.1$---and brackets
the exact value $\lambda_1=-1$ at every level.
The key improvement is that $\varepsilon_h^*=d_h^{\rm opt}/\kappa_\sigma$
converges at $O(h^2)$ (confirmed numerically, see \Cref{subsec:convergence-1center}),
whereas $\varepsilon_h^{(A)}$ (classical patch-integral) gives only $O(h)$ in 2D,
causing $\varepsilon_h^{(A)}/\varepsilon_h^*\approx79,\,159,\,321$ at $h=0.4,\,0.2,\,0.1$.
At $h=0.1$: $\varepsilon_h^*=2.24\times10^{-4}$, giving a certified interval width of $0.267$
for $k=1$.
The $\mathrm{P}_1$ upper bound is unaffected and converges to within $0.001$ of
$\lambda_1=-1$ at the finest level.

The P1 diagnostic (Step~m7) yields certified form-bound constants
$C_\epsilon^{\rm cert}\approx0.820$, $0.807$, $0.805$ at $h=0.4$, $0.2$, $0.1$,
well below $C_\epsilon=12$ (the conservative value used for $L_k^{(A)}$ and $L_k^*$).
Using $\sigma_{\rm opt}=C_\epsilon^{\rm cert}+\kappa_\sigma\approx1.270$, $1.257$, $1.255$
tightens the certified lower bound dramatically:
$L_1^{\rm opt}=-1.009$, $-1.003$, $-1.001$ at $h=0.4$, $0.2$, $0.1$, giving
two-sided intervals $[-1.009,\,-0.998]$, $[-1.003,\,-0.999]$, $[-1.001,\,-0.999]$
for $\lambda_1=-1$ --- interval widths $0.011$, $0.004$, $0.002$ (compared to
$5.02$, $0.975$, $0.267$ for $L_k^*$).
The P1 diagnostic thus reduces the certification gap by a factor of $\mathord{\sim}500$
at $h=0.4$, confirming that $C_\epsilon^{\rm cert}\ll C_\epsilon$ for well-resolved Coulomb potentials.

\subsection{Two-centred Coulomb Potential, 2D}
\label{subsec:num-twocenter-2d}

\paragraph{Setup}
$V(x) = -1/|x-a_1|-1/|x-a_2|$ on $Q=[-7,7]\times[-5,5]\subset\R^2$,
Neumann BC for the CECR lower bound (ECR space, natural boundary conditions;
see \Cref{rmk:bc-convention}),
Dirichlet BC for the $\mathrm{P}_1$ upper bound,
with $a_1=(-2,0)$, $a_2=(2,0)$ (internuclear distance $R_0=4$).
The mesh family is the same linearly-graded Delaunay triangulation
as the single-centre case (\Cref{subsec:num-coulomb-2d}),
with both singularities enforced as exact mesh vertices;
grading law $h_K\le\vartheta h r$ where $h$ is the nominal mesh
parameter and $r$ is the distance to the nearest nucleus.
The first-ring radius around each nucleus is chosen as the
\emph{optimal-patch radius} $\rho_{\rm opt}$ minimising
$d_h(\rho)=c_0(\rho)+c_*(\rho)$ (see \Cref{rmk:2d-rate}).
\Cref{fig:mesh-case8} shows a sample mesh at level $h=0.4$, $\vartheta=1$.

\begin{figure}[htbp]
\centering
\includegraphics[width=0.92\textwidth]{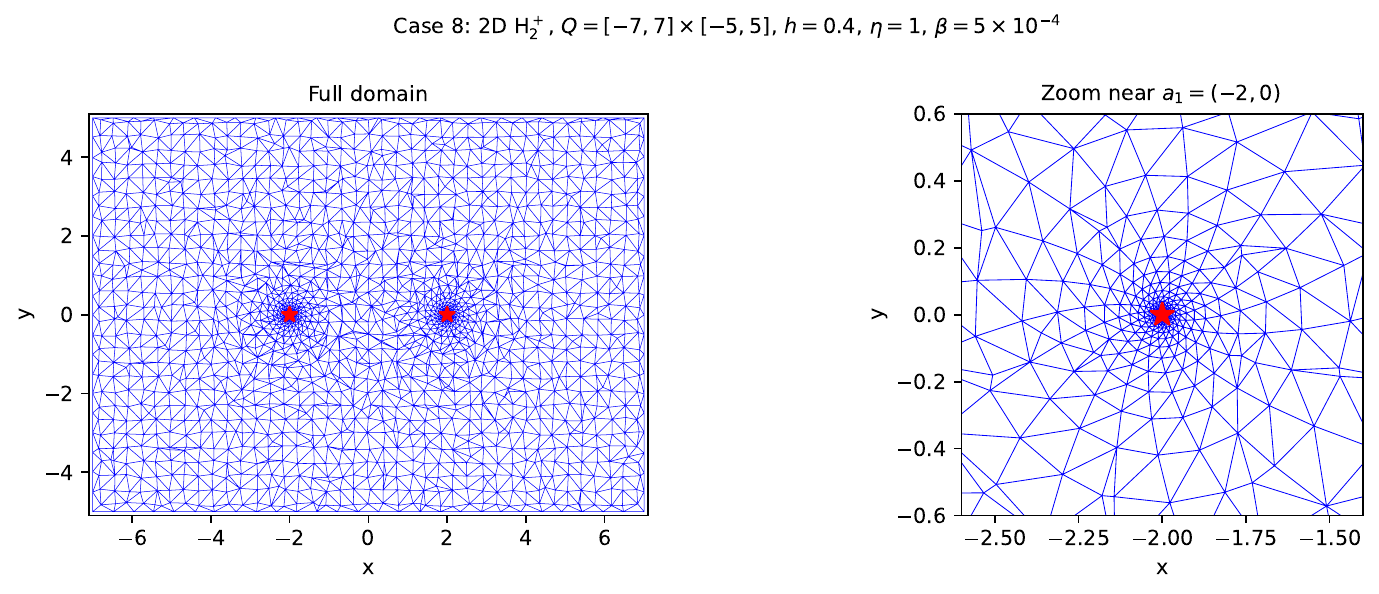}
\caption{Linearly-graded Delaunay triangulation for Case~8
(2D $\mathrm{H}_2^+$, $Q=[-7,7]\times[-5,5]$, $h=0.4$, $\vartheta=1$, optimal-patch radius).
Left: full domain; right: zoom near the left nucleus $a_1=(-2,0)$ (red star).
Both nuclei are enforced as exact mesh vertices;
the first-ring radius around each nucleus is the optimal-patch radius $\rho_{\rm opt}$.}
\label{fig:mesh-case8}
\end{figure}

\paragraph{Parameters}
$\sigma=12.5$, $\epsilon=0.55$, $C_\epsilon=12$,
$\kappa_\sigma=\min(1-\epsilon,\,\sigma-C_\epsilon)=0.45$,
$S_8([-7,7]\times[-5,5])\le1.462$
(same form-bound constants as in \Cref{subsec:num-coulomb-2d}).
The Payne--Weinberger constant is $C_h^{\mathrm{PW}}=h_{\max}/\pi$.
The perturbation parameter is $\varepsilon_h=d_h^{\mathrm{opt}}/\kappa_\sigma$,
where \(d_h^{\mathrm{opt}}\) is the minimum, over the common patch-radius
parameter, of the sum of the two-centre patch and off-patch contributions in
\eqref{eq:direct-est-c-grd-2d}.

\Cref{tab:twocenter-2d-cecr} reports the explicit CECR lower
bound and $\mathrm{P}_1$ upper bound for the first three eigenvalues.

\begin{table}[htbp]
\caption{2D $\mathrm{H}_2^+$
($V=-1/|x-a_1|-1/|x-a_2|$, $a_{1,2}=(\mp 2,0)$, $R_0=4$)
on $Q=[-7,7]\times[-5,5]$,
Neumann BC (CECR lower bounds) / Dirichlet BC ($\mathrm{P}_1$ upper bounds),
linearly-graded Delaunay mesh with optimal-patch radius
(same grading strategy as \Cref{subsec:num-coulomb-2d}),
$\sigma=12.5$, $\epsilon=0.55$, $C_\epsilon=12$, $\kappa_\sigma=0.45$:
shifted Ritz value $\mu_{k,h}^\sigma{-}\sigma$,
default lower bound $L_k^{\mu,\sigma}$,
certified corrected bounds
$L_k^{(A)}=(1-\varepsilon_h^{(A)})(L_k^{\mu,\sigma}{+}\sigma)-\sigma$,
$L_k^*=(1-\varepsilon_h^*)(L_k^{\mu,\sigma}{+}\sigma)-\sigma$
(\Cref{lem:perturb-sigma}),
and the optimal-shift bound
$L_k^{\rm opt}=(1-\varepsilon_h^{\rm opt})(L_k^{\mu,\sigma_{\rm opt}}+\sigma_{\rm opt})-\sigma_{\rm opt}$
with $\sigma_{\rm opt}=C_\epsilon^{\rm cert}+\kappa_\sigma$,
and $\mathrm{P}_1$ Galerkin upper bound $\lambda_{k,h}^{P_1}$, for $k=1,2,3$.
$\varepsilon_h^{(A)}$: global $L^{p_0}$ norm estimate (Choice~A);
$\varepsilon_h^*=d_h^{\mathrm{opt}}/\kappa_\sigma$: optimal-patch bound;
$C_\epsilon^{\rm cert}$: certified form-bound constant from P1 diagnostic (\Cref{rmk:ceps-sharp}).
All levels certified ($\varepsilon_h^{(A)},\,\varepsilon_h^*<\kappa_\sigma=0.45$).
Reference: $\lambda_1\approx-1.317$, $\lambda_2\approx-1.183$, $\lambda_3\approx-0.212$
(converged $\mathrm{P}_1$ values).}
\label{tab:twocenter-2d-cecr}
\centering
\footnotesize
\setlength{\tabcolsep}{2pt}
\resizebox{\textwidth}{!}{\begin{tabular}{ccrrccccc|rrrrrr}
\toprule
$h$ & $N_{\mathrm{tri}}$ & $h_{\max}$ &
$\Gamma_h$ & $A_h$ & $\varepsilon_h^{(A)}$ & $\varepsilon_h^{*}$ & $C_\epsilon^{\rm cert}$ &
$k$ & $\mu_{k,h}^\sigma{-}\sigma$ & $L_k^{\mu,\sigma}$ & $L_k^{(A)}$ & $L_k^{*}$ & $L_k^{\rm opt}$ & $\lambda_{k,h}^{P_1}$ \\
\midrule
\multirow{3}{*}{$0.4$}
  & \multirow{3}{*}{$166\,294$}
  & \multirow{3}{*}{$0.5529$}
  & \multirow{3}{*}{$0.01114$}
  & \multirow{3}{*}{$0.02475$}
  & \multirow{3}{*}{$0.3240$}
  & \multirow{3}{*}{$0.004607$}
  & \multirow{3}{*}{$1.0757$}
  & 1 & $-1.3198$ & $-4.3961$ & $-7.0216$ & $-4.4335$ & $-1.3261$ & $-1.3168$ \\
& & & & & & & & 2 & $-1.1862$ & $-4.3244$ & $-6.9731$ & $-4.3621$ & $-1.1983$ & $-1.1823$ \\
& & & & & & & & 3 & $-0.3853$ & $-3.9036$ & $-6.6886$ & $-3.9432$ & $-0.4579$ & $-0.2122$ \\
\midrule
\multirow{3}{*}{$0.2$}
  & \multirow{3}{*}{$623\,228$}
  & \multirow{3}{*}{$0.2883$}
  & \multirow{3}{*}{$0.002884$}
  & \multirow{3}{*}{$0.006409$}
  & \multirow{3}{*}{$0.1647$}
  & \multirow{3}{*}{$0.001179$}
  & \multirow{3}{*}{$1.0716$}
  & 1 & $-1.3186$ & $-2.3458$ & $-4.0186$ & $-2.3578$ & $-1.3202$ & $-1.3170$ \\
& & & & & & & & 2 & $-1.1844$ & $-2.2346$ & $-3.9257$ & $-2.2467$ & $-1.1876$ & $-1.1827$ \\
& & & & & & & & 3 & $-0.3847$ & $-1.5763$ & $-3.3759$ & $-1.5892$ & $-0.4046$ & $-0.2124$ \\
\bottomrule
\end{tabular}}
\end{table}

\paragraph{Discussion}
Both mesh levels are certified with both bounds
($\varepsilon_h^{(A)},\varepsilon_h^*<\kappa_\sigma=0.45$),
with the same optimal-patch grading strategy as in \Cref{subsec:num-coulomb-2d}.
The optimal-patch bound $L_k^*$ is significantly tighter than $L_k^{(A)}$:
for $k=1$ at $h=0.4$, $L_1^{(A)}=-7.022$ vs $L_1^*=-4.434$ (reduction of $2.59$),
and at $h=0.2$, $L_1^{(A)}=-4.019$ vs $L_1^*=-2.358$.
The $L_k^*$ interval for $k=1$ narrows from $[-4.43,\,-1.317]$ at $h=0.4$
to $[-2.36,\,-1.317]$ at $h=0.2$.
The $\varepsilon_h^*=d_h^{\rm opt}/\kappa_\sigma$ bound is $O(h^2)$, giving
$\varepsilon_h^*=0.004607$ at $h=0.4$ and $\varepsilon_h^*=0.001179$ at $h=0.2$,
versus $\varepsilon_h^{(A)}=0.3240$ and $0.1647$ respectively (ratio $\approx70$ and $140$).
The classical $\gamma_h$-shift bound diverges near both singularities;
the CECR bound remains finite at all levels by \Cref{def:Gamma-h}.
The converged reference values for the first three eigenvalues are
$\lambda_1\approx-1.317$, $\lambda_2\approx-1.183$, $\lambda_3\approx-0.212$
(from the $\mathrm{P}_1$ upper bounds at $h=0.2$).

The P1 diagnostic gives $C_\epsilon^{\rm cert}\approx1.076$ at both levels,
yielding $\sigma_{\rm opt}\approx1.526$ (vs.\ $\sigma=12.5$).
The resulting optimal-shift bounds $L_k^{\rm opt}$ are dramatically tighter:
for $k=1$, the interval narrows from $[-4.43,\,-1.317]$ ($L_1^*$)
to $[-1.326,\,-1.317]$ ($L_1^{\rm opt}$) at $h=0.4$,
and from $[-2.36,\,-1.317]$ to $[-1.320,\,-1.317]$ at $h=0.2$.
The certified width for $\lambda_1$ is $0.009$ ($h=0.4$) and $0.003$ ($h=0.2$),
bracketing the converged P1 reference value $\lambda_1\approx-1.317$.

\subsection{Single Coulomb Potential, 3D (hydrogen atom)}
\label{subsec:num-coulomb-3d}

\paragraph{Setup}
$V(x) = -1/|x|$ on the ball $B(0,6)\subset\R^3$.
The operator is $H = -\Delta - 1/|x|$; eigenvalues on the full space
follow $\lambda_n = -1/(4n^2)$:
$\lambda_1 = -0.25$, $\lambda_2 = -0.0625$,
$\lambda_3 = -1/36 \approx -0.02778$.

Neumann BC is imposed on $\partial B(0,6)$ for the CECR lower bound
(ECR space, natural conditions; see \Cref{rmk:bc-convention}).
Dirichlet BC is used for the $\mathrm{P}_1$ Galerkin upper bound in
\Cref{tab:coulomb-3d-cecr}; Neumann BC is used for the $\mathrm{P}_1$
Galerkin upper bound in the convergence study of \Cref{subsec:convergence-1center}.
The mesh is a structured icosahedron-shell ball
(angular parameter $\theta=0.1$, first-shell radius $h_{\min}=h^2/(10R)$,
subdivision level $n_{\rm sub}$ chosen to guarantee $h_{\max}<1.05h$),
generated without element filtering; the origin is a mesh vertex.
The direct-method bound $\varepsilon_h=d_h^{\rm opt}/\kappa_\sigma$
is computed using Duffy GL-5 quadrature on singular elements (those with
a vertex at the Coulomb centre).
Parameters: $\sigma=5.4$, $\epsilon=0.6$, $C_\epsilon=5.0$,
$\kappa_\sigma=\min(1-\epsilon,\,\sigma-C_\epsilon)=\min(0.4,0.4)=0.4$,
$C_6=0.642$ (rigorous $H^1(B(0,6))\hookrightarrow L^6(B(0,6))$ bound,
Neumann BC; see \Cref{rmk:c6-ball}).
Results are reported for four mesh levels ($h=1.4$, $h=1.2$, $h=1.0$, and $h=0.8$).

\begin{table}[htbp]
\caption{3D hydrogen ($V=-1/|x|$) on $B(0,6)$,
structured icosahedron-shell ball mesh ($\theta=0.1$, $n_{\rm sub}$ guaranteeing $h_{\max}<1.05h$),
$\sigma=5.4$, $\epsilon=0.6$, $C_\epsilon=5.0$, $\kappa_\sigma=0.4$, $C_6=0.642$:
shifted Ritz value $\mu_{k,h}^\sigma{-}\sigma$,
intermediate lower bound $L_k^{\mu,\sigma}$,
certified lower bound $L_k=(1-\varepsilon_h)(L_k^{\mu,\sigma}{+}\sigma)-\sigma$
using $\varepsilon_h=d_h^{\rm opt}/\kappa_\sigma$ (\Cref{lem:perturb-sigma}),
Dirichlet $\mathrm{P}_1$ upper bound $\lambda_{k,h}^{D,P_1}$,
and optimal-$\sigma$ bound
$L_k^{\rm opt}=(1-\varepsilon_h)(L_k^{\mu,\sigma_{\rm opt}}+\sigma_{\rm opt})-\sigma_{\rm opt}$
with $\sigma_{\rm opt}=C_\epsilon^{\rm cert}+\kappa_\sigma$ (\Cref{rmk:ceps-sharp}),
for $k=1,2,3$.
$\varepsilon_h=d_h^{\rm opt}/\kappa_\sigma$: direct-method bound (Duffy GL-5 quadrature);
$C_\epsilon^{\rm cert}$: m7 diagnostic eigenvalue.
All levels certified ($\varepsilon_h<\kappa_\sigma=0.4$).
Exact: $\lambda_1=-0.250$, $\lambda_2=-0.0625$, $\lambda_3\approx-0.02778$.}
\label{tab:coulomb-3d-cecr}
\centering
\footnotesize
\setlength{\tabcolsep}{2pt}
\resizebox{\textwidth}{!}{%
\begin{tabular}{ccrrccc|rrrrrr}
\toprule
$h$ & $N_{\mathrm{tet}}$ & $h_{\max}$ &
$\Gamma_h$ & $A_h$ & $\varepsilon_h$ & $C_\epsilon^{\rm cert}$ &
$k$ & $\mu_{k,h}^\sigma{-}\sigma$ & $L_k^{\mu,\sigma}$ & $L_k$ & $\lambda_{k,h}^{D,P_1}$ & $L_k^{\rm opt}$ \\
\midrule
\multirow{3}{*}{$1.4$}
  & \multirow{3}{*}{$331{,}200$}
  & \multirow{3}{*}{$1.442$}
  & \multirow{3}{*}{$4.881{\times}10^{-2}$}
  & \multirow{3}{*}{$1.220{\times}10^{-1}$}
  & \multirow{3}{*}{$1.432{\times}10^{-1}$}
  & \multirow{3}{*}{$0.411$}
  & 1 & $-0.313$ & $-3.212$ & $-3.525$ & $-0.209$ & $-0.463$ \\
& & & & & & & 2 & $-0.307$ & $-3.211$ & $-3.524$ & $\phantom{-}0.253$ & $-0.459$ \\
& & & & & & & 3 & $-0.288$ & $-3.207$ & $-3.521$ & $\phantom{-}0.256$ & $-0.448$ \\
\midrule
\multirow{3}{*}{$1.2$}
  & \multirow{3}{*}{$357{,}120$}
  & \multirow{3}{*}{$1.215$}
  & \multirow{3}{*}{$3.512{\times}10^{-2}$}
  & \multirow{3}{*}{$8.781{\times}10^{-2}$}
  & \multirow{3}{*}{$1.140{\times}10^{-1}$}
  & \multirow{3}{*}{$0.411$}
  & 1 & $-0.310$ & $-2.744$ & $-3.047$ & $-0.210$ & $-0.429$ \\
& & & & & & & 2 & $-0.306$ & $-2.743$ & $-3.046$ & $\phantom{-}0.250$ & $-0.426$ \\
& & & & & & & 3 & $-0.286$ & $-2.737$ & $-3.041$ & $\phantom{-}0.253$ & $-0.412$ \\
\midrule
\multirow{3}{*}{$1.0$}
  & \multirow{3}{*}{$400{,}320$}
  & \multirow{3}{*}{$1.030$}
  & \multirow{3}{*}{$2.318{\times}10^{-2}$}
  & \multirow{3}{*}{$5.794{\times}10^{-2}$}
  & \multirow{3}{*}{$7.555{\times}10^{-2}$}
  & \multirow{3}{*}{$0.412$}
  & 1 & $-0.308$ & $-2.289$ & $-2.524$ & $-0.210$ & $-0.392$ \\
& & & & & & & 2 & $-0.305$ & $-2.288$ & $-2.523$ & $\phantom{-}0.248$ & $-0.390$ \\
& & & & & & & 3 & $-0.285$ & $-2.280$ & $-2.516$ & $\phantom{-}0.250$ & $-0.375$ \\
\midrule
\multirow{3}{*}{$0.8$}
  & \multirow{3}{*}{$452{,}160$}
  & \multirow{3}{*}{$0.811$}
  & \multirow{3}{*}{$1.295{\times}10^{-2}$}
  & \multirow{3}{*}{$3.238{\times}10^{-2}$}
  & \multirow{3}{*}{$4.948{\times}10^{-2}$}
  & \multirow{3}{*}{$0.412$}
  & 1 & $-0.307$ & $-1.716$ & $-1.898$ & $-0.211$ & $-0.361$ \\
& & & & & & & 2 & $-0.305$ & $-1.715$ & $-1.898$ & $\phantom{-}0.245$ & $-0.359$ \\
& & & & & & & 3 & $-0.285$ & $-1.704$ & $-1.887$ & $\phantom{-}0.248$ & $-0.342$ \\
\bottomrule
\end{tabular}}
\end{table}

\paragraph{Discussion}
All four levels are certified ($\varepsilon_h<\kappa_\sigma=0.4$), with $\varepsilon_h$
decreasing from $1.43{\times}10^{-1}$ ($h=1.4$) to $1.14{\times}10^{-1}$ ($h=1.2$)
to $7.55{\times}10^{-2}$ ($h=1.0$) to $4.95{\times}10^{-2}$ ($h=0.8$).
The structured ball mesh avoids element filtering entirely;
the origin is an exact mesh vertex and the boundary is detected radially.
The $\mathrm{P}_1$ upper bound $\lambda_{1,h}^{P_1}$ is stable at approximately $-0.210$
across all levels, converging to $\lambda_1=-0.250$ from above.
Eigenvalues $\lambda_2$ and $\lambda_3$ are nearly degenerate at all levels,
reflecting the approximate spherical symmetry of the hydrogen atom on $B(0,6)$.
The m7 diagnostic (\Cref{rmk:ceps-sharp}) yields $C_\epsilon^{\rm cert}\approx 0.411$--$0.412$
at all levels (very stable), giving $\sigma_{\rm opt}\approx 0.811$--$0.812\ll\sigma=5.4$.
Re-running CECR at $\sigma_{\rm opt}$ yields
$L_1^{\rm opt}=-0.463$ ($h=1.4$, enclosure $[-0.463,\,-0.209]$, width $0.253$),
$L_1^{\rm opt}=-0.429$ ($h=1.2$, enclosure $[-0.429,\,-0.210]$, width $0.219$),
$L_1^{\rm opt}=-0.392$ ($h=1.0$, enclosure $[-0.392,\,-0.210]$, width $0.182$),
and $L_1^{\rm opt}=-0.361$ ($h=0.8$, enclosure $[-0.361,\,-0.211]$, width $0.150$),
all bracketing $\lambda_1=-0.250$.
All opt-$\sigma$ lower bounds are below the respective Dirichlet P1 upper bounds,
so the certified enclosures overlap and do not yet certify $\lambda_1<\lambda_2$.

\subsection[Two-centred Coulomb Potential, 3D (H2+ molecule)]{Two-centred Coulomb Potential, 3D
(\texorpdfstring{$\mathrm{H}_2^+$}{H2+} molecule)}
\label{subsec:num-twocenter-3d}

\paragraph{Setup}
$\mathrm{H}_2^+$ molecular ion:
$V(x) = -1/|x-a_1|-1/|x-a_2|$ in 3D,
$a_{1,2}=(\mp 2,0,0)$ (internuclear distance $d=4\,\text{bohr}$),
on the box $Q=[-8,8]\times[-6,6]^2$.
The operator $H=-\Delta + V$ uses the $-\Delta$ convention
(not $-\tfrac{1}{2}\Delta$).
Reference: $\lambda_1\approx -0.5513$ (Bates~1953; Madsen--Peek~1971).

Neumann BC is imposed on $\partial Q$ for the CECR lower bound
(ECR space, natural conditions; see \Cref{rmk:bc-convention});
Dirichlet BC is used for the $\mathrm{P}_1$ Galerkin upper bound.
The mesh uses a linearly-graded Delaunay tetrahedralisation
with both nuclei embedded as exact vertices.
The patch radius around each nucleus is chosen as the optimal-patch radius
(minimising $d_h(\rho)$; see \Cref{rmk:2d-rate}).
Parameters: $\sigma=5.4$, $\epsilon=0.6$, $C_\epsilon=5.0$,
$\kappa_\sigma=\min(1-\epsilon,\,\sigma-C_\epsilon)=\min(0.4,0.4)=0.4$,
$C_6=0.619$ (rigorous $H^1(\Omega)\hookrightarrow L^6(\Omega)$ bound via \eqref{eq:C6-box-line}).
The $h=1.0$ CECR eigensolver uses the ichol+PCG iterative path
($4.72{\times}10^6$ DOFs).

\begin{table}[htbp]
\caption{3D $\mathrm{H}_2^+$ ($V=-1/|x-a_1|-1/|x-a_2|$,
$a_{1,2}=(\mp 2,0,0)$, $d=4\,\mathrm{bohr}$) on $Q=[-8,8]\times[-6,6]^2$,
Neumann BC (CECR lower bounds) / Dirichlet BC ($\mathrm{P}_1$ upper bounds),
linearly-graded Delaunay mesh with optimal-patch radius,
$\sigma=5.4$, $\epsilon=0.6$, $C_\epsilon=5.0$, $\kappa_\sigma=0.4$, $C_6=0.619$:
shifted Ritz value $\mu_{k,h}^\sigma{-}\sigma$,
intermediate lower bound $L_k^{\mu,\sigma}$,
certified lower bound $L_k=(1-\varepsilon_h)(L_k^{\mu,\sigma}{+}\sigma)-\sigma$
using $\varepsilon_h=d_h^{\rm opt}/\kappa_\sigma$,
$\mathrm{P}_1$ Galerkin upper bound $\lambda_{k,h}^{P_1}$,
and optimal-$\sigma$ bound $L_k^{\rm opt}$,
for $k=1,2,3$.
$h=1.0$: ichol+PCG (4.72M DOFs), $\sigma_{\rm opt}=1.112$ ($C_\epsilon^{\rm cert}=0.712$).
Reference: $\lambda_1\approx -0.5513$ (Bates~1953).}
\label{tab:twocenter-3d-cecr}
\centering
\footnotesize
\setlength{\tabcolsep}{2pt}
\resizebox{\textwidth}{!}{%
\begin{tabular}{ccrrccc|rrrrrr}
\toprule
$h$ & $N_{\mathrm{tet}}$ & $h_{\max}$ &
$\Gamma_h$ & $A_h$ & $\varepsilon_h$ & $C_\epsilon^{\rm cert}$ &
$k$ & $\mu_{k,h}^\sigma{-}\sigma$ & $L_k^{\mu,\sigma}$ & $L_k$ & $\lambda_{k,h}^{P_1}$ & $L_k^{\rm opt}$ \\
\midrule
\multirow{3}{*}{$1.0$}
  & \multirow{3}{*}{$1{,}571{,}791$}
  & \multirow{3}{*}{$1.028$}
  & \multirow{3}{*}{$2.77{\times}10^{-2}$}
  & \multirow{3}{*}{$6.92{\times}10^{-2}$}
  & \multirow{3}{*}{$1.706{\times}10^{-2}$}
  & \multirow{3}{*}{$0.712$}
  & 1 & $-0.711$ & $-2.481$ & $-2.531$ & $-0.544$ & $-1.045$ \\
& & & & & & & 2 & $-0.376$ & $-2.346$ & $-2.398$ & $-0.316$ & $-0.569$ \\
& & & & & & & 3 & $-0.347$ & $-2.334$ & $-2.386$ & $-0.105$ & $-0.482$ \\
\bottomrule
\end{tabular}}
\end{table}

\paragraph{Discussion}
For the 3D H$_2^+$ molecule on the box $Q=[-8,8]\times[-6,6]^2$
with new parameters ($\sigma=5.4$, $\epsilon=0.6$, $C_\epsilon=5.0$, $\kappa_\sigma=0.4$),
one mesh level is certified.

The $h=1.0$ mesh ($N_{\rm tet}=1{,}571{,}791$, $h_{\max}=1.028$,
$\varepsilon_h=1.71\times10^{-2}$) has $4.72\times10^6$ DOFs.
The ichol+PCG eigensolver yields rough-$\sigma$ bounds
$\lambda_1\in[-2.531,\,-0.544]$ (width $1.987$);
the m7 diagnostic gives $C_\epsilon^{\rm cert}=0.712$,
$\sigma_{\rm opt}=1.112$.
At $\sigma_{\rm opt}$: $\lambda_1\in[-1.045,\,-0.544]$ (width $0.501$),
$\lambda_2\in[-0.569,\,-0.316]$, $\lambda_3\in[-0.482,\,-0.105]$.
The reference $\lambda_1\approx-0.5513$ lies within the certified interval.
The $\mathrm{P}_1$ upper bound $\lambda_{1,h}^{P_1}=-0.544$ at $h=1.0$
closely tracks the reference.

\subsection{Error convergence for one-centre Coulomb potential}
\label{subsec:convergence-1center}

Since the exact eigenvalues are known for the single-centre Coulomb
potential in both 2D ($\lambda_1 = -1$) and 3D ($\lambda_1 = -0.25$),
we can track the certified lower-bound gap $\lambda_1-L_1$ and
upper-bound gap $\lambda_{1,h}^{P_1}-\lambda_1$ across the mesh levels
in \Cref{tab:coulomb-2d-cecr,tab:coulomb-3d-cecr}.
\Cref{fig:convergence-1center} (2D) and \Cref{fig:convergence-1center-3d}
(3D) each show two panels on a log--log scale with $h_{\max}$
increasing left to right: panel~(a) plots the lower-bound gap
$\lambda_1-L_1$ (rough $\sigma$) and the Dirichlet P1 upper-bound gap
$\lambda_{1,h}^{P_1}-\lambda_1$;
panel~(b) shows the convergence of the mesh constants
$\Gamma_h$, $A_h$, and $\varepsilon_h$.

\begin{figure}[htbp]
  \centering
  \includegraphics[width=\linewidth]{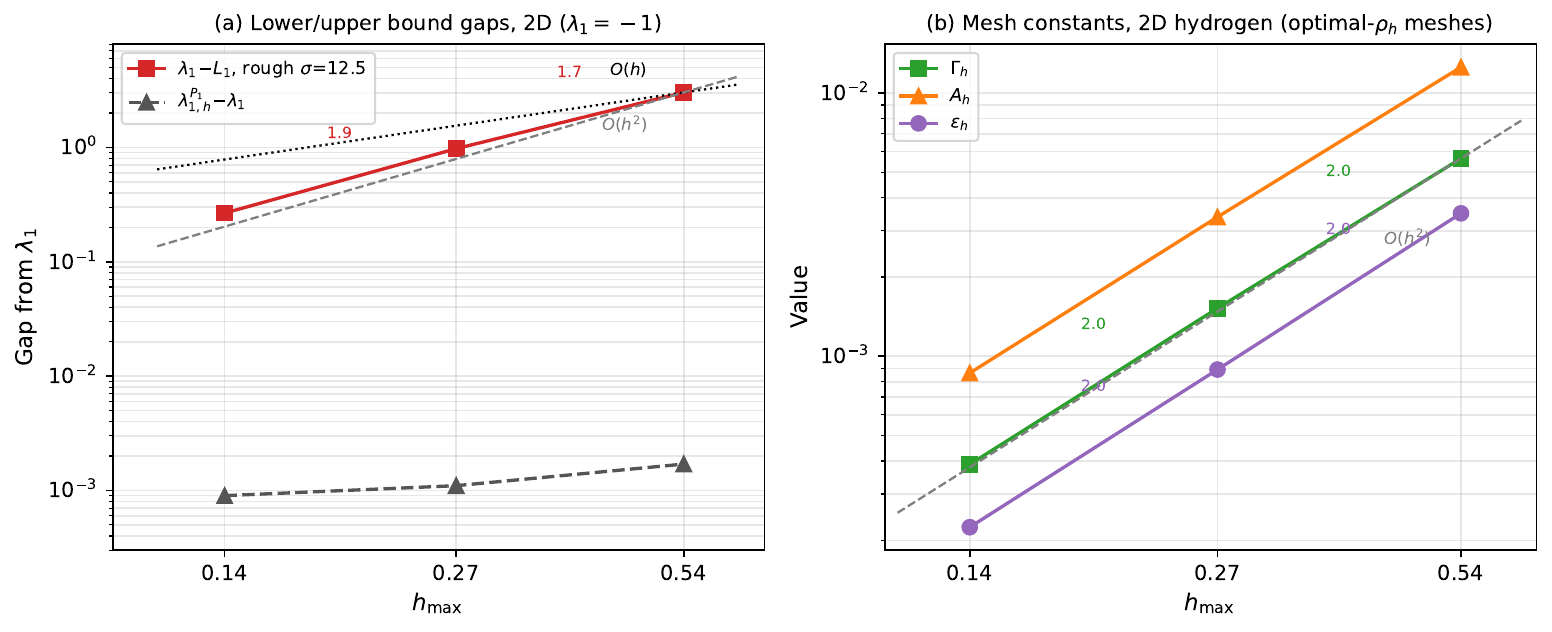}
  \caption{Convergence for 2D hydrogen ($V=-1/|x|$) on $Q=[-5,5]^2$,
    $\lambda_1=-1$, three mesh levels $h=0.4,0.2,0.1$
    (optimal-patch-radius meshes; \Cref{tab:coulomb-2d-cecr}).
    (a) Lower-bound gap $\lambda_1-L_1$ (rough $\sigma=12.5$, red squares)
    and upper-bound gap $\lambda_{1,h}^{P_1}-\lambda_1$ (grey triangles).
    Numbers on the lower-gap curve are local EOC.
    Dotted/$O(h)$ and dashed/$O(h^2)$ reference lines shown.
    (b) Mesh constants $\Gamma_h$ (green), $A_h$ (orange),
    $\varepsilon_h=d_h^{\rm opt}/\kappa_\sigma$ (purple circles);
    all converge at $O(h^2)$.}
  \label{fig:convergence-1center}
\end{figure}

\begin{figure}[htbp]
  \centering
  \includegraphics[width=\linewidth]{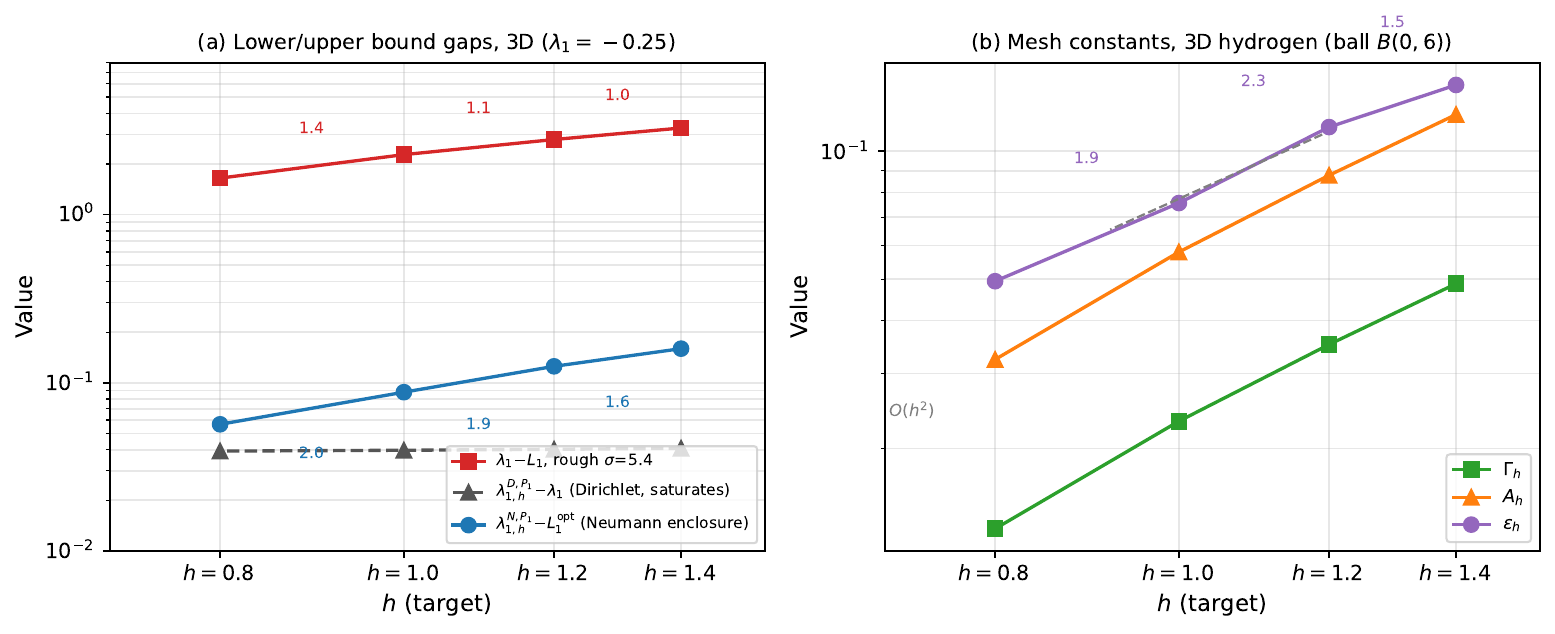}
  \caption{Convergence for 3D hydrogen ($V=-1/|x|$) on $B(0,6)$,
    reference eigenvalue $\lambda_1=-0.25$ (on $\R^3$),
    mesh levels $h=1.4$, $1.2$, $1.0$, and $0.8$
    (structured icosahedron-shell ball meshes; \Cref{tab:coulomb-3d-cecr}).
    (a) Lower-bound gap $\lambda_1-L_1$ (rough $\sigma=5.4$, red squares);
    Dirichlet $\mathrm{P}_1$ upper-bound gap $\lambda_{1,h}^{D,P_1}-\lambda_1$
    (grey triangles; saturates near $0.040$ due to domain truncation);
    and certified Neumann enclosure width
    $\lambda_{1,h}^{N,P_1}-L_1^{\rm opt}$ (blue circles),
    which converges at $O(h^2)$ (EOC $\approx1.6$--$2.0$).
    (b) Mesh constants $\Gamma_h$ (green), $A_h$ (orange),
    $\varepsilon_h=d_h^{\rm opt}/\kappa_\sigma$ (purple circles) for four mesh levels.}
  \label{fig:convergence-1center-3d}
\end{figure}

In the 2D panels, the rough shift $\sigma=12.5$ gives a lower-bound
gap $\lambda_1-L_1$ decreasing from $3.03$ at $h=0.4$ ($h_{\max}=0.537$)
to $0.27$ at $h=0.1$ ($h_{\max}=0.139$), with local EOC $\approx 1.7$
and $1.9$ between the three levels.
The upper-bound gap $\lambda_{1,h}^{P_1}-\lambda_1$ is already very small
($\approx 0.002$ at $h=0.4$, $\approx 0.001$ at $h=0.1$).
Panel~(b) confirms that $\varepsilon_h=d_h^{\rm opt}/\kappa_\sigma$,
$\Gamma_h$, and $A_h$ all converge at $O(h^2)$, consistent with the
optimal-patch-radius design of \Cref{rmk:2d-rate} and the estimate of
\Cref{rmk:Gamma-h-rate}.

In the 3D panels (levels $h=1.4$, $1.2$, $1.0$, and $0.8$ on the ball $B(0,6)$,
parameters $\sigma=5.4$, $\kappa_\sigma=0.4$),
the rough shift gives a lower-bound gap $\lambda_1-L_1$ decreasing from
$3.28$ at $h=1.4$ ($h_{\max}=1.442$) to $2.80$ ($h=1.2$), $2.27$ ($h=1.0$),
and $1.65$ ($h=0.8$, $h_{\max}=0.811$), with local EOC approximately $1.0$--$1.4$.
The gap remains large because the correction formula is dominated
by the conservative shift $\sigma=5.4$ at these coarse mesh sizes.
The Dirichlet $\mathrm{P}_1$ upper-bound gap saturates near $0.040$ due to the
domain-truncation error $\lambda_1^D(B_R)-\lambda_1\approx 0.039$ at $R=6$.
The certified Neumann enclosure width $\lambda_{1,h}^{N,P_1}-L_1^{\rm opt}$
decreases from $0.160$ ($h=1.4$) to $0.125$, $0.088$, and $0.057$ ($h=0.8$),
with local EOC approximately $1.6$--$2.0$.
This convergence demonstrates that the certified two-sided bounds
$[L_1^{\rm opt},\,\lambda_{1,h}^{N,P_1}]$ genuinely bracket the Neumann
eigenvalue $\lambda_1^N(B(0,6))$, with the enclosure width converging at
$O(h^2)$ consistent with the $\varepsilon_h$ estimate in panel~(b).
Panel~(b) shows $\varepsilon_h$ decreasing from $1.43{\times}10^{-1}$ ($h=1.4$)
to $4.95{\times}10^{-2}$ ($h=0.8$),
consistent with $O(h^2)$ behaviour of the direct-method bound.

\subsection{Discussion}
\label{subsec:numerics-discussion}

All four experiments demonstrate the convergent CECR lower bound of
\Cref{thm:cecr-lb-convergent}, with all computable quantities
($C_h^{\mathrm{PW}}$, $\Gamma_h$, $A_h$, $L_k^{\mu,\sigma}$)
evaluated explicitly.

In the 2D experiments, $\Gamma_h$ decreases by roughly two orders of
magnitude across the four mesh levels in both single- and two-centre
cases, confirming the $\Gamma_h=O(h^2)$ prediction, with respect to the
nominal mesh parameter, of
\Cref{rmk:Gamma-h-rate} for graded meshes.
The explicit gap $(\mu_{1,h}^\sigma{-}\sigma)-L_1^{\mu,\sigma}$ closes at
rate $O(h^2)$ with EOC $\approx 2.0$.
In contrast, the classical $\gamma_h$-shift
(\Cref{rmk:gamma-h-limitation}) produces divergent bounds for these
meshes because $\gamma_h=\|c_h^-\|_{L^\infty}\to\infty$.

In the 3D experiments, $\Gamma_h$ is substantially larger than in 2D
at comparable nominal mesh parameter $h$, because 3D element diameters
grow faster away from the singularity; it decreases with refinement in
both cases.
The explicit gap $(\mu_{k,h}^\sigma{-}\sigma)-L_k^{\mu,\sigma}$ is
correspondingly wider in 3D but narrows with refinement.
The shifted Ritz value $\mu_{1,h}^\sigma{-}\sigma$ is
\emph{not} itself a certified lower bound for $\lambda_1(-\Delta+V)$
on $\mathbb{R}^3$; the certified lower bound is $L_1$ obtained via
\Cref{thm:cecr-lb-convergent,thm:main}.

\begin{remark}[Certified spectral separation]
\label{rmk:numerical-separation}
The two-sided bounds also provide a direct way to verify spectral
separation.  If, for the same continuous problem, the certified enclosure
for $\lambda_k$ is
$[L_k,U_k]$, then
\[
  U_k < L_{k+1}
\]
proves a positive gap between $\lambda_k$ and $\lambda_{k+1}$.  This
criterion is especially useful when a later error estimate requires an
isolated eigenvalue or an isolated cluster.  Such a certified separation
can also be used as the spectral-gap input for sharper a posteriori
postprocessing, for example Weinstein-type bounds
\cite{WeinsteinStenger1972} or Kato--Lehmann--Goerisch bounds
\cite{Kato1949,Lehmann1963,Goerisch1986,BehnkeGoerisch1994}.
\begin{itemize}
\item \emph{2D one-centre Coulomb.}
  The exact ground state is separated from the double level
  $\lambda_2=\lambda_3=-1/9$.
  At the finest level $h=0.1$ in \Cref{tab:coulomb-2d-cecr},
  $U_1=-0.9991<L_2=-0.4662$, certifying a gap of at least
  $0.5329$ between $\lambda_1$ and $\lambda_2$.
  The method does not separate $\lambda_2$ and $\lambda_3$,
  as expected from the exact degeneracy.
\item \emph{2D two-centre Coulomb.}
  At the finest level $h=0.2$ in \Cref{tab:twocenter-2d-cecr},
  $U_1=-1.3170$ and $L_2=-2.2467$, so the certified enclosures
  $[L_1,U_1]=[-2.36,-1.317]$ and $[L_2,U_2]=[-2.25,-1.183]$ overlap.
  The current two mesh levels do not yet certify spectral separation;
  finer meshes are needed.
\item \emph{3D hydrogen.}
  At all four levels in \Cref{tab:coulomb-3d-cecr}, the certified
  lower bounds $L_1^{\rm opt},L_2^{\rm opt},L_3^{\rm opt}$ lie below
  the respective Dirichlet P1 upper bounds, so the enclosures overlap
  and do not yet prove $\lambda_1<\lambda_2$.
  The opt-$\sigma$ enclosure at $h=0.8$ is $[-0.361,\,-0.211]$
  (width $0.150$), bracketing the exact $\lambda_1=-0.250$.
  The Ritz values $\mu_{k,h}^{\sigma_{\rm opt}}-\sigma_{\rm opt}$ show
  three closely spaced values near $-0.30$, consistent with the
  near-degeneracy of the lowest eigenstates of hydrogen on $B(0,6)$.
\item \emph{3D $\mathrm{H}_2^+$.}
  At the certified level ($h=1.0$) in \Cref{tab:twocenter-3d-cecr},
  $L_2^{\rm opt}<U_1$ ($L_2^{\rm opt}=-0.569<U_1=-0.544$),
  so the certified enclosures overlap and spectral separation
  $\lambda_1<\lambda_2$ is not yet proven.
  The Ritz values show sizeable numerical gaps
  (about $0.24$ between $k=1$ and $k=2$),
  but finer mesh levels are required to close the certified gap.
\end{itemize}
\end{remark}

\section{Conclusion}
\label{sec:conclusion}

We have developed a pair-space framework for the CECR FEM method on
sign-changing and unbounded potentials, introduced the elementwise
reaction constant $\Gamma_h$ and a convergent CECR lower-bound
theorem, and combined these with a perturbation lemma into a fully
explicit two-sided eigenvalue bound for Coulomb-type potentials.
The main contributions are:

\begin{enumerate}[label=(\roman*)]
\item A \emph{pair-space CECR framework}
  (\Cref{sec:cecr}) that lifts the kinetic/reaction split into a
  two-component bilinear form with a self-contained max-min proof
  of the CECR lower bound for non-negative piecewise-constant $c_h$
  (\Cref{thm:cecr-abstract}), recovering Liu's positive-$c_h$
  lower bound \cite{Liu2015}.

\item A \emph{convergent CECR lower bound} via the elementwise
  reaction constant $\Gamma_h=\max_K c_h^-|_K\cdot h_K^{\,2}/\pi^2$
  (\Cref{thm:cecr-lb-convergent}). The key insight is that
  $\Gamma_h\to 0$ as the mesh is refined (while
  $\gamma_h\to\infty$), so the gap
  $(\mu_{k,h}^\sigma{-}\sigma)-L_k^{\mu,\sigma}$ closes at rate
  $O(h^2)$ in the mesh-family parameter on graded meshes. This bypasses the
  uncomputable form-norm constant
  $\widetilde C_h$ entirely.

\item A \emph{fixed-shift two-sided theorem} \Cref{thm:main}
  that combines the CECR lower bound with a perturbation lemma
  (\Cref{lem:perturb-sigma}) for the $V\leftrightarrow c_h$
  substitution, yielding a fully computable enclosure
  $L_k\le\lambda_k\le U_k$ without $L^\infty$-control of $c_h^-$.
  The perturbation ratio $\varepsilon_h$ is controlled by the
  mesh-local estimates \eqref{eq:direct-est-c-grd} and
  \eqref{eq:direct-est-c-grd-2d}; under \Cref{ass:graded-mesh} these reduce to the
  graded-rate bound \eqref{eq:eps-graded-rate}, namely
  $O(h^2/\kappa_\sigma)$ for both $N=3$ and, with the patch radius
  chosen as in \Cref{rmk:2d-rate}, $N=2$.

\end{enumerate}

\paragraph{Open questions and future directions}

\begin{itemize}
  \item \textbf{Finer 3D meshes and tighter enclosures.}
    The 3D explicit gaps $(\mu_{k,h}^\sigma{-}\sigma)-L_k^{\mu,\sigma}$ remain
    substantial at the reported mesh sizes because the elementwise
    reaction constant $\Gamma_h$ decreases more slowly on coarse
    tetrahedral meshes. Finer graded meshes or enriched trial
    spaces would tighten both the CECR lower bound and the
    conforming upper bound, narrowing the two-sided enclosure.

  \item \textbf{Interval-arithmetic verification.}
    The bounds reported in this paper are evaluated in
    double-precision floating-point arithmetic and are therefore
    subject to rounding and matrix-computation error (rounding
    errors $O(10^{-15})$ are small versus the reported gaps
    $O(10^{-2})$, but not accounted for in a verified sense).
    Elevating the explicit bounds of this paper to
    \emph{verified} (IEEE~754-safe) enclosures requires interval
    arithmetic (e.g., INTLAB \cite{Rump1999}) applied to all
    matrix assembly and eigensolve steps, and is left for future
    work.

  \item \textbf{Sharpened ground-state lower bound via the
    Lehmann--Goerisch method.}
    The CECR bounds developed here are \emph{projection-based},
    and their accuracy is ultimately limited by the coarse mesh
    size entering $C_h^{\mathrm{PW}}$ and $\Gamma_h$.  Once an
    explicit (even crude) lower bound
    $\rho\le\lambda_{m+1}$ is available---for instance,
    $\rho\le\lambda_2$ when $m=1$---classical
    eigenvalue-enclosure techniques can, in a second stage,
    combine $\rho$ with high-quality approximate eigenfunctions
    (e.g., from $hp$-FEM) to produce substantially sharper lower
    bounds on $\lambda_1,\ldots,\lambda_m$: Kato's bound
    \cite{Kato1949}, Temple's bound \cite{Temple1928}, the
    intermediate-problem method of Weinstein and Stenger
    \cite{WeinsteinStenger1972}, and the Lehmann--Maehly--Goerisch
    method \cite{Lehmann1963,Goerisch1986,BehnkeGoerisch1994}.
    As discussed in \cite[Ch.~5]{liu2024guaranteed}, the
    Lehmann--Goerisch bounds are essentially insensitive to the
    quality of $\rho$ and recover the optimal convergence rate
    of the approximate eigenfunction, thereby bypassing the
    sub-optimal rate inherent in projection-based lower bounds.
    The explicit CECR bounds of the present paper furnish
    precisely such a $\rho$ for Coulomb-type potentials, enabling
    this two-stage strategy; its implementation is pursued in
    forthcoming work.

  \item \textbf{Magnetic Hamiltonians and multi-electron systems.}
    Extension of the CECR framework to $(-i\nabla + A_h)^2 + V$
    (magnetic Hamiltonians) and to multi-electron operators
    with antisymmetry constraints are natural next directions.
\end{itemize}

\section*{Acknowledgments}
The author is supported by the Japan Society for the Promotion
of Science (JSPS) through KAKENHI Grant Numbers JP20KK0306,
JP22H00512, JP24K00538, and JP21H00998.

\appendix
\section{Proof of \texorpdfstring{\Cref{thm:cecr-abstract}}%
                  {Theorem~\ref{thm:cecr-abstract}}}
\label{app:cecr-abstract-proof}

As noted after the statement of \Cref{thm:cecr-abstract}, the
argument below uses only that $\widehat\Pi_h$ is an
$\widehat a$-orthogonal projection from $\widehat V$ onto
$\widehat V_h$---so that the Pythagorean identity
\eqref{eq:pair-pythagoras} is available---together with the
residual error estimate \eqref{eq:pair-error-estimate}. No
further information about $\widehat\Pi_h$ or $V_h^{\ECR}$ enters.
The proof is therefore identical in substance to the general
framework of \cite[Theorem~2.1]{Liu2015} and
\cite[Theorem~4.1]{liu2024guaranteed}, and is reproduced here
only to keep the paper self-contained.

\begin{proof}[Proof of \Cref{thm:cecr-abstract}]
It is convenient to work with the reciprocal Rayleigh quotient
\[
  \widetilde R(\widehat v)
  := \frac{\widehat b(\widehat v,\widehat v)}{\widehat a(\widehat v,\widehat v)},
  \qquad \widehat v\in\widehat V(h)\setminus\{0\},
\]
for which the max-min principle on $\widehat V$ and on $\widehat V_h$
reads
\begin{equation}
\label{eq:max-min-pair}
  \frac{1}{\mu_k}
  = \max_{\substack{H\subset\widehat V\\\dim H=k}}
    \min_{\widehat v\in H}\widetilde R(\widehat v),
  \qquad
  \frac{1}{\mu_{k,h}}
  = \max_{\substack{H\subset\widehat V_h\\\dim H=k}}
    \min_{\widehat v_h\in H}\widetilde R(\widehat v_h).
\end{equation}
Let $\{\widehat u_i\}_{i=1}^{\infty}$ and
$\{\widehat u_{i,h}\}_{i=1}^{n}$ be $\widehat a$-orthonormal
eigenfunctions of \eqref{eq:pair-evp} on $\widehat V$ and $\widehat V_h$
respectively, and set $E_k:=\operatorname{span}\{\widehat u_i\}_{i=1}^{k}$.

\medskip
\emph{Case 1.} Suppose there exists
$\widehat\phi\in E_k\setminus\{0\}$ with
$\widehat\Pi_h\widehat\phi=0$. Then
$\widehat\phi-\widehat\Pi_h\widehat\phi=\widehat\phi$ and
\eqref{eq:pair-error-estimate} gives
$\|\widehat\phi\|_{\widehat b}
\le C_h\|\widehat\phi\|_{\widehat a}$, whence
$\widetilde R(\widehat\phi)\le C_h^{\,2}$. Since
$\min_{E_k}\widetilde R=1/\mu_k$ by \eqref{eq:max-min-pair}
(taking $H=E_k$ attains the max), we conclude
$1/\mu_k\le\widetilde R(\widehat\phi)\le C_h^{\,2}$, i.e.
$\mu_k\ge 1/C_h^{\,2}\ge\mu_{k,h}/(1+C_h^{\,2}\mu_{k,h})$
(the last inequality is elementary), which is \eqref{eq:cecr-abstract}.

\medskip
\emph{Case 2.} Otherwise $\widehat\Pi_h$ is injective on $E_k$, so
$\widehat\Pi_h(E_k)\subset\widehat V_h$ is a $k$-dimensional subspace.
Using $\widehat\Pi_h(E_k)$ as a trial subspace in the max-min for
$1/\mu_{k,h}$ in \eqref{eq:max-min-pair},
\[
  \frac{1}{\mu_{k,h}}
  \;\ge\;
  \min_{\widehat v_h\in\widehat\Pi_h(E_k)\setminus\{0\}}
  \widetilde R(\widehat v_h).
\]
Let $\widehat\phi\in E_k$ be a minimiser, so that with
$\widehat v_h:=\widehat\Pi_h\widehat\phi$,
\begin{equation}
\label{eq:step-proj-bound}
  \|\widehat v_h\|_{\widehat b}
  \;\le\; \mu_{k,h}^{\,-1/2}\,\|\widehat v_h\|_{\widehat a}.
\end{equation}
Combining the triangle inequality in the $\widehat b$-norm,
estimate \eqref{eq:step-proj-bound}, the error estimate
\eqref{eq:pair-error-estimate}, and Cauchy--Schwarz in $\mathbb R^2$
with weights $(\mu_{k,h}^{\,-1/2},C_h)$,
\begin{align*}
  \|\widehat\phi\|_{\widehat b}
  &\le \|\widehat v_h\|_{\widehat b}
      + \|\widehat\phi-\widehat v_h\|_{\widehat b} \\
  &\le \mu_{k,h}^{\,-1/2}\,\|\widehat v_h\|_{\widehat a}
      + C_h\,\|\widehat\phi-\widehat v_h\|_{\widehat a} \\
  &\le \sqrt{\mu_{k,h}^{\,-1}+C_h^{\,2}}\;
       \sqrt{\|\widehat v_h\|_{\widehat a}^{\,2}
             +\|\widehat\phi-\widehat v_h\|_{\widehat a}^{\,2}} \\
  &= \sqrt{\mu_{k,h}^{\,-1}+C_h^{\,2}}\;\|\widehat\phi\|_{\widehat a},
\end{align*}
where the last equality uses the Pythagorean identity
\eqref{eq:pair-pythagoras}. Hence
$\widetilde R(\widehat\phi)\le\mu_{k,h}^{\,-1}+C_h^{\,2}$, and since
$\widehat\phi\in E_k$ the max-min principle gives
$1/\mu_k=\min_{E_k}\widetilde R\le\widetilde R(\widehat\phi)
\le\mu_{k,h}^{\,-1}+C_h^{\,2}$. Inverting,
\[
  \mu_k \;\ge\;
  \frac{1}{\mu_{k,h}^{\,-1}+C_h^{\,2}}
  \;=\;
  \frac{\mu_{k,h}}{1+C_h^{\,2}\,\mu_{k,h}}.
\]
\end{proof}


\section{Graded-mesh estimate for \texorpdfstring{$\varepsilon_h$}{epsilon\_h}}
\label{app:graded-eps}

We record here the graded-mesh perturbation estimates used in the main text.
Let $\mathcal T_h$ be a simplicial mesh of $Q$, let $h$ be the mesh parameter
used in the grading law, and let \(c_h\) denote the elementwise-average
projection of \(V\):
\[
  c_h|_K=V_K:=|K|^{-1}\int_K V(x)\,dx .
\]
Throughout this appendix, $\omega_0$ denotes the singular patch, that is, the
union of the elements touching the singular point(s).

\paragraph{Key identity.}
Since $\int_K (V-V_K)\,dx=0$ on every cell $K$,
\begin{align}
((V-c_h)u,u)
&=\sum_{K}\int_K (V-V_K)u^2\,dx \notag\\
&=\sum_K \int_K (V-V_K)\bigl(u^2-(u^2)_K\bigr)\,dx,
\label{eq:Vchuident}
\end{align}
where $(u^2)_K:=|K|^{-1}\int_K u^2\,dx$.

In this appendix we estimate an admissible mesh-dependent perturbation
constant $d_h$ defined by
\begin{equation}
|((V-c_h)u,u)| \le d_h \|u\|_{H^1(Q)}^2
\qquad \forall\,u\in H^1(Q).
\label{eq:dh-def}
\end{equation}

\begin{theorem}[Direct patch/off-patch estimate, $N=3$]
\label{thm:graded-eps}
Let $N=3$, let $Q\subset \mathbb R^3$ be the auxiliary domain, and let
$V(x)=|x|^{-1}$. Then the constant $d_h$ in \eqref{eq:dh-def} satisfies
\begin{equation}
d_h
\le
\|V-c_h\|_{L^{3/2}(\omega_0)}\,C_6^2
\;+\;
\max_{\{K:\,0\notin \overline K\}}
\frac{h_K}{2}\,\|V-V_K\|_{L^\infty(K)}.
\label{eq:direct-est-c-grd}
\end{equation}
Here $C_6$ is the embedding constant of
$H^1(Q)\hookrightarrow L^6(Q)$.
\end{theorem}

\begin{proof}
By \eqref{eq:Vchuident},
\[
|((V-c_h)u,u)|
\le
\sum_{K\subset \omega_0}\int_K |V-V_K|\,u^2\,dx
\;+\;
\sum_{K\not\subset \omega_0}
\left|\int_K (V-V_K)\bigl(u^2-(u^2)_K\bigr)\,dx\right|.
\]

For the singular patch, Hölder's inequality gives
\[
\sum_{K\subset \omega_0}\int_K |V-V_K|\,u^2\,dx
\le
\|V-c_h\|_{L^{3/2}(\omega_0)} \|u\|_{L^6(\omega_0)}^2
\le
\|V-c_h\|_{L^{3/2}(\omega_0)} C_6^2 \|u\|_{H^1(Q)}^2.
\]

For an off-patch cell $K$, since $K$ is convex,
\[
\|u^2-(u^2)_K\|_{L^1(K)}
\le
\frac{h_K}{2}\|\nabla(u^2)\|_{L^1(K)}
=
h_K\int_K |u||\nabla u|\,dx.
\]
Hence
\[
\|u^2-(u^2)_K\|_{L^1(K)}
\le
h_K \|u\|_{L^2(K)}\|\nabla u\|_{L^2(K)}
\le
\frac{h_K}{2}\bigl(\|u\|_{L^2(K)}^2+\|\nabla u\|_{L^2(K)}^2\bigr).
\]
Therefore
\[
\left|\int_K (V-V_K)\bigl(u^2-(u^2)_K\bigr)\,dx\right|
\le
\frac{h_K}{2}\|V-V_K\|_{L^\infty(K)}
\bigl(\|u\|_{L^2(K)}^2+\|\nabla u\|_{L^2(K)}^2\bigr).
\]
Summing over $K$ with $0\notin \overline K$ yields \eqref{eq:direct-est-c-grd}.
\end{proof}

\begin{remark}[Expected graded rate in three dimensions]
If the mesh satisfies \Cref{ass:graded-mesh}(G1) and
\(\omega_0\subset B_\rho\), then the proof also gives
\begin{equation}
\label{eq:middle-way}
  d_h
  \le
  \frac12\sup_{\{K:\,0\notin\overline K\}}\frac{h_K^2}{r_K^2}
  +2\left(\frac{8\pi}{3}\right)^{2/3}\rho\,C_6^2 .
\end{equation}
Thus, for the three-dimensional choice \(\rho=\beta h^2\),
\begin{equation}
\label{eq:graded-eps-3d}
  d_h\le C_{\mathrm{grd}}h^2.
\end{equation}
\begin{equation}
\label{eq:Cgrd-3d}
  C_{\mathrm{grd}}
  :=
  \frac{\vartheta^2}{2}
  +2\beta C_6^2\left(\frac{8\pi}{3}\right)^{2/3}.
\end{equation}
\end{remark}

\begin{theorem}[Direct patch/off-patch estimate, $N=2$]
\label{thm:graded-eps-2d}
Let $N=2$, let $Q\subset \mathbb R^2$ be the auxiliary rectangle, and let
$V(x)=|x|^{-1}$. Then the constant $d_h$ in \eqref{eq:dh-def} satisfies
\begin{equation}
d_h
\le
\|V-c_h\|_{L^{4/3}(\omega_0)}\,S_8^2
\;+\;
\max_{\{K:\,0\notin \overline K\}}
\frac{h_K}{2}\,\|V-V_K\|_{L^\infty(K)}.
\label{eq:direct-est-c-grd-2d}
\end{equation}
Here $S_8$ is the embedding constant of
$H^1(Q)\hookrightarrow L^8(Q)$.
\end{theorem}

\begin{proof}
The proof is the same as in three dimensions. The off-patch estimate is
identical. On the singular patch one uses Hölder with exponents
$4/3$ and $4$:
\[
\sum_{K\subset \omega_0}\int_K |V-V_K|\,u^2\,dx
\le
\|V-c_h\|_{L^{4/3}(\omega_0)} \|u^2\|_{L^4(\omega_0)}
\le
\|V-c_h\|_{L^{4/3}(\omega_0)} S_8^2 \|u\|_{H^1(Q)}^2.
\]
This gives \eqref{eq:direct-est-c-grd-2d}.
\end{proof}

\begin{remark}[Expected graded rate in two dimensions]
For a two-dimensional rate statement, take any \(p\in(1,2)\), set
\(q=p/(p-1)\), and let \(S_{2q}\) be the constant of
\(H^1(Q)\hookrightarrow L^{2q}(Q)\).  If
\(\omega_0\subset B_{\beta h^{\gamma_{\rm p}}}\), then Jensen's inequality and
polar integration give
\begin{equation}
\label{eq:graded-eps-2d}
  d_h
  \le
  \frac{\vartheta^2}{2}h^2
  +2S_{2q}^2
  \left(\frac{2\pi}{2-p}\right)^{1/p}
  \beta^{2/p-1}h^{\gamma_{\rm p}(2/p-1)}.
\end{equation}
Consequently, if \(\gamma_{\rm p}(2/p-1)\ge2\), then
\begin{equation}
\label{eq:Cgrd-2d}
  d_h\le C_{\mathrm{grd}}^{2D}h^2,
  \qquad
  C_{\mathrm{grd}}^{2D}
  :=
  \frac{\vartheta^2}{2}
  +2S_{2q}^2
  \left(\frac{2\pi}{2-p}\right)^{1/p}
  \beta^{2/p-1}.
\end{equation}
\end{remark}

\medskip

\begin{remark}[Two-dimensional constants on rectangles]
\label{rmk:rect-L8}
Let $Q=[-a,a]\times[-b,b]$ and let $S_r(Q)$ denote a constant
satisfying
\[
  \|u\|_{L^r(Q)} \le S_r(Q)\,\|u\|_{H^1(Q)}.
\]
We recover here the explicit rectangle constants used in the two-dimensional
examples.

For $w\in W^{1,1}(Q)$, averaging the fundamental theorem of calculus in the
two coordinate directions yields
\begin{equation}
\label{eq:rectangle-line-ineq}
  \|w\|_{L^2(Q)}^2
  \le
  \left(\frac{\|w\|_{L^1(Q)}}{2a}+\|\partial_{x_1}w\|_{L^1(Q)}\right)
  \left(\frac{\|w\|_{L^1(Q)}}{2b}+\|\partial_{x_2}w\|_{L^1(Q)}\right).
\end{equation}

Applying \eqref{eq:rectangle-line-ineq} with $W=|u|^2$, followed by
Cauchy--Schwarz, yields
\[
  \|u\|_{L^4(Q)}^4
  \le
  \left(\frac{A^2}{2a}+2AB_x\right)
  \left(\frac{A^2}{2b}+2AB_y\right),
\]
where $A=\|u\|_{L^2(Q)}$, $B_x=\|u_{x_1}\|_{L^2(Q)}$, and
$B_y=\|u_{x_2}\|_{L^2(Q)}$. Hence
\begin{equation}
\label{eq:S4-rect-line}
  S_4(Q)^4
  \le
  \max_{A^2+B_x^2+B_y^2\le1}
  \left(\frac{A^2}{2a}+2AB_x\right)
  \left(\frac{A^2}{2b}+2AB_y\right).
\end{equation}

For $S_6$, apply \eqref{eq:rectangle-line-ineq} with $W=|u|^3$. Since
\[
  \int_Q |\partial_i(|u|^3)|
  \le 3\int_Q |u|^2|\partial_i u|
  \le 3\|u\|_{L^4(Q)}^2\|\partial_i u\|_{L^2(Q)}
\]
and
\[
  \|u\|_{L^3(Q)}^3\le \|u\|_{L^2(Q)}\|u\|_{L^4(Q)}^2,
\]
we obtain, after normalizing $\|u\|_{H^1(Q)}=1$,
\begin{equation}
\label{eq:S6-rect-line}
  S_6(Q)^6
  \le
  S_4(Q)^4
  \max_{A^2+B_x^2+B_y^2\le1}
  \left(\frac{A}{2a}+3B_x\right)
  \left(\frac{A}{2b}+3B_y\right).
\end{equation}

For $S_8$, use $W=|u|^4$. The estimates
\[
  \int_Q |\partial_i(|u|^4)|
  \le 4\int_Q |u|^3|\partial_i u|
  \le 4\|u\|_{L^6(Q)}^3\|\partial_i u\|_{L^2(Q)}
\]
give
\begin{equation}
\label{eq:S8-rect-line}
  S_8(Q)^8
  \le
  \max_{B_x^2+B_y^2\le1}
  \left(\frac{S_4(Q)^4}{2a}+4S_6(Q)^3B_x\right)
  \left(\frac{S_4(Q)^4}{2b}+4S_6(Q)^3B_y\right).
\end{equation}

Solving the finite-dimensional maximization problems
\eqref{eq:S4-rect-line}--\eqref{eq:S8-rect-line} gives
\[
\begin{array}{c|ccc}
\hline
Q & S_4(Q) & S_6(Q) & S_8(Q)\\ \hline
[-5,5]^2 & 0.872 & 1.173 & 1.466\\
\hline
[-7,7]\times[-5,5] & 0.867 & 1.169 & 1.462 \\
\hline
\end{array}
\]
\end{remark}

\begin{remark}[A directly computable $L^6$ bound on rectangular boxes]
\label{rmk:c6-boxes}
For a box $Q_{a,b,c}:=[-a,a]\times[-b,b]\times[-c,c]$, a sharper three-dimensional
embedding constant is obtained by combining a pointwise
fundamental-theorem-of-calculus estimate with interpolation.

For smooth $u$ on $Q_{a,b,c}$, fix $(x_1,x_2,x_3)\in Q_{a,b,c}$. Applying the
fundamental theorem of calculus successively in the three coordinate
directions and averaging over the corresponding interval gives
\begin{align*}
  |u(x_1,x_2,x_3)|
  \le{}&
  (8abc)^{-1/2}\|u\|_{L^2(Q_{a,b,c})}
  + \sqrt{\frac{a}{2bc}}\|u_{x_1}\|_{L^2(Q_{a,b,c})}\\
  &+ \sqrt{\frac{b}{2ac}}\|u_{x_2}\|_{L^2(Q_{a,b,c})}
  + \sqrt{\frac{c}{2ab}}\|u_{x_3}\|_{L^2(Q_{a,b,c})}.
\end{align*}
Hence, by Cauchy--Schwarz,
\[
\|u\|_{L^\infty(Q_{a,b,c})}^2
\le
\frac{1+4(a^2+b^2+c^2)}{8abc}\,
\|u\|_{H^1(Q_{a,b,c})}^2.
\]
Now interpolate between $L^\infty$ and $L^2$:
\[
\|u\|_{L^6(Q_{a,b,c})}
\le
\|u\|_{L^\infty(Q_{a,b,c})}^{2/3}
\|u\|_{L^2(Q_{a,b,c})}^{1/3},
\]
so that
\[
  \|u\|_{L^6(Q_{a,b,c})}
  \le
  \left(
    \frac{1+4(a^2+b^2+c^2)}{8abc}
  \right)^{1/3}
  \|u\|_{H^1(Q_{a,b,c})}.
\]
This yields
\begin{equation}
C_6(Q_{a,b,c})
\le
\left(\frac{1+4(a^2+b^2+c^2)}{8abc}\right)^{1/3}.
\label{eq:C6-box-line}
\end{equation}
The derivation uses $\|u\|_{L^\infty}$ only for smooth functions. Since
$C^\infty(\overline Q_{a,b,c})$ is dense in $H^1(Q_{a,b,c})$ and the constant
in \eqref{eq:C6-box-line} depends only on the box, the same inequality extends
to every $u\in H^1(Q_{a,b,c})$ by approximation.

For the two-center three-dimensional computation,
\[
Q_2=[-8,8]\times[-6,6]^2,
\]
and therefore
\[
C_6(Q_2)\le \left(\frac{545}{2304}\right)^{1/3}\approx 0.6184.
\]
\end{remark}
\begin{remark}[Embedding constant for a ball]
\label{rmk:c6-ball}
For balls, however, one can obtain a fully explicit estimate by a
radial line-segment extension.  Let
\[
  B_R:=\{x\in\mathbb R^3: |x|<R\}.
\]
We use the sharp Sobolev constant on \(\mathbb R^3\),
\begin{equation}
\label{eq:K3-talenti-correct}
  \|w\|_{L^6(\mathbb R^3)}
  \le K_3\|\nabla w\|_{L^2(\mathbb R^3)},
  \qquad
  K_3
  =
  \frac{1}{\sqrt{3\pi}}
  \left(\frac{4}{\sqrt{\pi}}\right)^{1/3}
  \approx 0.4273 ,
\end{equation}
see Talenti~\cite{Talenti1976}.
For smooth $u$ define an extension $Eu$ on $\mathbb R^3$ by
\[
  Eu(x)=u(x),\quad |x|<R,
\]
\[
  Eu(r\theta)=\left(\frac{2R-r}{R}\right)^2u((2R-r)\theta),
  \quad R<r<2R,\quad \theta\in\mathbb S^2,
\]
and $Eu(x)=0$ for $|x|\ge 2R$.  This extension is continuous across
$\partial B_R$ and vanishes at $|x|=2R$.  Writing $s=2R-r$ in the annulus,
we now estimate its gradient explicitly.

In polar coordinates,
\[
  |\nabla f|^2 = |\partial_r f|^2+r^{-2}|\nabla_{\mathbb S^2}f|^2 .
\]
On the annulus $R<r<2R$, set
\[
  w(r,\theta):=Eu(r\theta)=\left(\frac{s}{R}\right)^2u(s\theta),
  \qquad s=2R-r .
\]
The radial derivative is
\[
  \partial_r w
  =
  -\left(\frac{s}{R}\right)^2 u_r(s\theta)
  -\frac{2s}{R^2}u(s\theta),
\]
where $u_r$ denotes the radial derivative of $u$.  Hence, for any
$\delta>0$,
\[
  |\partial_r w|^2
  \le
  (1+\delta^{-1})\left(\frac{s}{R}\right)^4 |u_r(s\theta)|^2
  +(1+\delta)\frac{4s^2}{R^4}|u(s\theta)|^2 .
\]
After multiplying by the annular Jacobian $r^2\,dr\,d\theta$ and changing
variables from $r$ to $s$, we use
\[
  r=2R-s,\qquad 0<s<R,\qquad
  \left(\frac{s}{R}\right)^4r^2\le s^2,\qquad
  \frac{4s^2r^2}{R^4}\le \frac{16}{R^2}s^2 .
\]
Therefore
\begin{equation}
\label{eq:radial-extension-detail}
  \int_{R<|x|<2R} |\partial_r Eu|^2\,dx
  \le
  (1+\delta^{-1})\int_{B_R}|u_r|^2\,dx
  +
  \frac{16(1+\delta)}{R^2}\int_{B_R}|u|^2\,dx .
\end{equation}
For the tangential part,
\[
  \nabla_{\mathbb S^2}w
  =
  \left(\frac{s}{R}\right)^2\nabla_{\mathbb S^2}u(s\theta).
\]
Thus
\[
  r^{-2}|\nabla_{\mathbb S^2}w|^2 r^2\,dr\,d\theta
  =
  \left(\frac{s}{R}\right)^4
  |\nabla_{\mathbb S^2}u(s\theta)|^2\,dr\,d\theta .
\]
Changing variables again and using $s^2/R^4\le s^{-2}$ for $0<s<R$ gives
\begin{equation}
\label{eq:tangential-extension-detail}
  \int_{R<|x|<2R} r^{-2}|\nabla_{\mathbb S^2}Eu|^2\,dx
  \le
  \int_{B_R}s^{-2}|\nabla_{\mathbb S^2}u|^2\,dx .
\end{equation}
Adding \eqref{eq:radial-extension-detail} and
\eqref{eq:tangential-extension-detail}, and then adding the interior
contribution over $B_R$, gives
\begin{equation}
\label{eq:ball-extension-gradient}
  \|\nabla Eu\|_{L^2(\mathbb R^3)}^2
  \le
  \left(2+\delta^{-1}\right)\|\nabla u\|_{L^2(B_R)}^2
  +
  \frac{16(1+\delta)}{R^2}\|u\|_{L^2(B_R)}^2 .
\end{equation}

Applying the sharp Sobolev inequality \eqref{eq:K3-talenti-correct} to $Eu$
gives the direct ball embedding
\begin{equation}
\label{eq:C6-ball-polar-delta}
  C_6(B_R)
  \le
  K_3
  \max\left\{
    2+\delta^{-1},\frac{16(1+\delta)}{R^2}
  \right\}^{1/2},
  \qquad \delta>0 .
\end{equation}
Equating the two terms gives the optimal choice
\[
  2+\delta_R^{-1}=\frac{16(1+\delta_R)}{R^2},
\]
and hence the following rigorous values:
\[
\begin{array}{c|c|c}
\hline
R & \delta_R & C_6(B_R)\text{ from }\eqref{eq:C6-ball-polar-delta}\\ \hline
5 & 2.703 & 0.658\\
6 & 4.055 & 0.642\\
8 & 7.531 & 0.624\\
\hline
\end{array}
\]

\end{remark}


\section{Explicit form-bound constants}
\label{app:ceps-analytic}

We record here explicit form-bound constants for the Coulomb potentials used in
the numerical experiments. The form-bound constant $C_\epsilon$ is defined as
the smallest constant such that
\[
\int_Q \frac{u(x)^2}{|x|}\,dx
\le
\epsilon \|\nabla u\|_{L^2(Q)}^2
+ C_\epsilon \|u\|_{L^2(Q)}^2
\qquad (u\in H^1(Q)).
\]
In the two-center cases, the potential is
\[
V(x)=\frac{1}{|x-a_1|}+\frac{1}{|x-a_2|}.
\]

\subsection{Direct spectral estimate in two dimensions}
\label{rmk:ceps-square-direct}

For the square
\[
Q=[-5,5]^2,
\]
the quantity of direct interest is
\[
C_\epsilon^\sharp(Q)
:=
\sup_{u\in H^1(Q)\setminus\{0\}}
\frac{\displaystyle \int_Q |x|^{-1}u(x)^2\,dx
-\epsilon\int_Q |\nabla u(x)|^2\,dx}
{\displaystyle \int_Q u(x)^2\,dx}.
\]
Then
\[
\int_Q \frac{u(x)^2}{|x|}\,dx
\le
\epsilon\|\nabla u\|_{L^2(Q)}^2
+
C_\epsilon^\sharp(Q)\|u\|_{L^2(Q)}^2.
\]
For $\epsilon=0.5$, the direct finite-dimensional spectral approximation gives
\[
C_{0.5}^\sharp([-5,5]^2)\approx 0.215.
\]
This value is numerically sharp, though not a verified enclosure.

\paragraph{Rigorous one-center square bound.}
For the same square, splitting $Q$ into an inner disk $B_\rho$ and the outer
part $Q\setminus B_\rho$ yields
\[
\int_Q \frac{u(x)^2}{|x|}\,dx
\le
0.5\,\|\nabla u\|_{L^2(Q)}^2
+
\left(
\frac{1}{\rho}
+
\frac{(3\pi)^{3/4}}{10}\rho^{1/2}
+
(3\pi)^{3/2}\rho
\right)\|u\|_{L^2(Q)}^2.
\]
Optimizing in $\rho$ gives the rigorous explicit bound
\[
C_{0.5}([-5,5]^2)\le 10.99.
\]

\paragraph{Rigorous two-center rectangle bound.}
Let
\[
Q=[-7,7]\times[-5,5],
\qquad
a_1=(-2,0),
\qquad
a_2=(2,0).
\]
Splitting $Q$ into the two half-rectangles $\{x_1\le 0\}$ and $\{x_1\ge 0\}$,
the far-center contribution is bounded by $1/2$, and the same disk-splitting
argument gives
\[
\int_Q \left(\frac{1}{|x-a_1|}+\frac{1}{|x-a_2|}\right)u(x)^2\,dx
\le
0.5\,\|\nabla u\|_{L^2(Q)}^2
+
11.54\,\|u\|_{L^2(Q)}^2.
\]
Hence
\[
C_{0.5}^{(2)}([-7,7]\times[-5,5])\le 11.54.
\]

\subsection{Three-dimensional constants used in the computations}

\paragraph{One-center ball.}
The one-center 3D computation is carried out on the ball
\[
Q=B(0,6).
\]
By \Cref{rmk:c6-ball}, the embedding constant satisfies
\[
  C_6(B(0,6))\le 0.642.
\]
For any \(0<\rho<6\), splitting \(B(0,6)\) into \(B_\rho\) and its complement gives
\[
\int_{B(0,6)}\frac{u(x)^2}{|x|}\,dx
\le
\left(\frac{8\pi}{3}\right)^{2/3}\rho\,C_6(B(0,6))^2
\|u\|_{H^1(B(0,6))}^2
+
\frac1\rho\|u\|_{L^2(B(0,6))}^2 .
\]
Set
\[
  A_B:=\left(\frac{8\pi}{3}\right)^{2/3} C_6(B(0,6))^2 .
\]
Then
\[
\int_{B(0,6)}\frac{u(x)^2}{|x|}\,dx
\le
A_B\rho\,\|\nabla u\|_{L^2(B(0,6))}^2
+
\left(\frac1\rho+A_B\rho\right)\|u\|_{L^2(B(0,6))}^2 .
\]
For the numerical value \(\epsilon=0.6\), choose \(\rho=\epsilon/A_B\).
Since \(A_B\le 1.701\), this gives \(\rho\approx0.353<6\), and therefore
\begin{equation}
\label{eq:ceps-ball-3d}
  C_{0.6}(B(0,6))
  \le
  \frac{A_B}{0.6}+0.6
  \le 3.44 .
\end{equation}
The computations use the slightly more conservative value \(C_\epsilon=5.0\).

\paragraph{Two-center box.}
The two-center 3D computation uses
\[
Q=[-8,8]\times[-6,6]^2,\qquad
a_1=(-2,0,0),\qquad a_2=(2,0,0).
\]
From \Cref{rmk:c6-boxes},
\[
  C_6(Q)\le \left(\frac{545}{2304}\right)^{1/3}\approx0.619.
\]
Splitting \(Q\) into the two half-boxes \(\{x_1\le0\}\) and \(\{x_1\ge0\}\),
the contribution from the farther Coulomb center is bounded by
\(\frac12\|u\|_{L^2(Q)}^2\).  Applying the preceding inner/outer split to
the nearer center on each half-box gives, with
\[
  A_Q:=\left(\frac{8\pi}{3}\right)^{2/3}C_6(Q)^2,
\]
\[
\int_Q\left(\frac1{|x-a_1|}+\frac1{|x-a_2|}\right)u(x)^2\,dx
\le
A_Q\rho\,\|\nabla u\|_{L^2(Q)}^2
+
\left(\frac1\rho+A_Q\rho+\frac12\right)\|u\|_{L^2(Q)}^2 .
\]
For \(\epsilon=0.6\), take \(\rho=\epsilon/A_Q\).  Since
\(A_Q\le1.581\), this gives
\begin{equation}
\label{eq:ceps-box-3d-two-center}
  C_{0.6}^{(2)}(Q)
  \le
  \frac{A_Q}{0.6}+0.6+\frac12
  \le 3.74 .
\end{equation}
Again, the numerical experiments use the conservative value
\(C_\epsilon=5.0\).

\section{Rigorous lower bound for the sharp form constant}
\label{app:ceps-eta}

This appendix gives the computable formula promised in
\Cref{rmk:ceps-sharp} for the minimization problem
\eqref{eq:opt-c-eps}.  Fix the coefficient $\epsilon\in(0,1)$ in that
quotient and write $c_h^-$ for the negative part of the cell-average
potential.
The optimal form-bound constant for this fixed $\epsilon$ is
\[
  C_\epsilon^{\rm opt}=-\eta,\qquad
  \eta
  =
  \inf_{0\ne u\in H^1(\Omega)}
  \frac{\epsilon\|\nabla u\|^2-(c_h^-u,u)}{\|u\|^2}.
\]
Since the trial function $u=1$ gives $\eta\le -(c_h^-,1)/(1,1)<0$ whenever
$c_h^-\not\equiv0$, a rigorous lower bound for $\eta$ immediately yields a
rigorous upper bound for $C_\epsilon^{\rm opt}$.

\subsection*{Scaled Schr\"odinger problem}

Divide the Rayleigh quotient by $\epsilon$ and introduce the sign-changing
piecewise constant potential
\[
  q_h:=-\frac{c_h^-}{\epsilon}.
\]
Let $\nu_1$ be the first eigenvalue of
\[
  -\Delta + q_h
  \quad\text{on }H^1(\Omega)
\]
with the natural Neumann boundary condition.  Then
\begin{equation}
\label{eq:eta-nu-scale}
  \eta=\epsilon\,\nu_1.
\end{equation}
Thus it is enough to compute a certified lower bound for $\nu_1$.

\subsection*{Preliminary coercive shift}

Choose a preliminary, possibly {\em non-sharp}, form bound for $c_h^-$:
\begin{equation}
\label{eq:appC-prelim-form-bound}
  (c_h^-u,u)
  \le
  \alpha\|\nabla u\|^2 + C_\alpha\|u\|^2
  \qquad \forall u\in H^1(\Omega),
\end{equation}
where $0<\alpha<\epsilon$.  For instance, $\alpha$ and $C_\alpha$ may
be taken from the Hardy--cutoff/local-domination estimate discussed in
\Cref{rmk:kato-ch-valid}.  Then the negative part of $q_h$ satisfies
\[
  (q_h^-u,u)
  =
  \epsilon^{-1}(c_h^-u,u)
  \le
  \frac{\alpha}{\epsilon}\|\nabla u\|^2
  + \frac{C_\alpha}{\epsilon}\|u\|^2.
\]
Set
\[
  \varepsilon_*:=\frac{\alpha}{\epsilon},
  \qquad
  C_*:=\frac{C_\alpha}{\epsilon},
\]
and choose a shift $\sigma_*>C_*$.  This makes the shifted form for
$-\Delta+q_h+\sigma_*$ coercive.

\subsection*{Discrete shifted problem}

Let $\nu_{1,h}^{\sigma_*}$ be the first CECR Ritz value of the shifted problem
\begin{equation}
\label{eq:evp-eta}
  (\nabla_h u_h,\nabla_h v_h)
  + ((q_h+\sigma_*)\Pi_{0,h}u_h,\Pi_{0,h}v_h)
  =
  \nu_{1,h}^{\sigma_*}(u_h,v_h),
  \qquad u_h,v_h\in V_h^{\ECR}.
\end{equation}
Equivalently, this is the generalized matrix eigenvalue problem
\begin{equation}
\label{eq:appC-matrix-shifted}
  \left(K_{\ECR}
  -\frac{1}{\epsilon}C^-_{0,h}
  +\sigma_* M_{0,h}\right)U
  =
  \nu_{1,h}^{\sigma_*}M_{\ECR}U,
\end{equation}
where $K_{\ECR}$ and $M_{\ECR}$ are the CECR stiffness and mass matrices for
$(\nabla_h\cdot,\nabla_h\cdot)$ and $(\cdot,\cdot)$, while
$M_{0,h}$ and $C^-_{0,h}$ are the projected mass and reaction matrices for
$(\Pi_{0,h}\cdot,\Pi_{0,h}\cdot)$ and
$(c_h^-\Pi_{0,h}\cdot,\Pi_{0,h}\cdot)$.

Define the elementwise correction factor
\begin{equation}
\label{eq:appC-gamma-A}
  \Gamma_h^*
  :=
  \max_{K\in\mathcal T^h}
  \frac{h_K^2\|q_h^-\|_{L^\infty(K)}}{\pi^2}
  =
  \frac{1}{\epsilon}
  \max_{K\in\mathcal T^h}
  \frac{h_K^2\|c_h^-\|_{L^\infty(K)}}{\pi^2},
  \qquad
  A_h^*
  :=
  \frac{\Gamma_h^*}{1-\varepsilon_*}.
\end{equation}
Let $\widetilde C_h\ge C_h$ be any certified upper bound for the CECR
interpolation constant, for example
\[
  \widetilde C_h=\max_{K\in\mathcal T^h}h_K/\pi.
\]

\subsection*{Certified lower bound for \texorpdfstring{$\eta$}{eta}}

Applying \Cref{thm:cecr-lb-convergent} to the scaled potential $q_h$ gives
\begin{equation}
\label{eq:appC-nu-lower}
  \nu_1
  \ge
  \underline\nu_{1,h}
  :=
  \frac{\nu_{1,h}^{\sigma_*}}
       {(1+A_h^*)\bigl(1+\widetilde C_h^{\,2}\nu_{1,h}^{\sigma_*}\bigr)}
  -\sigma_* .
\end{equation}
Combining \eqref{eq:eta-nu-scale} and \eqref{eq:appC-nu-lower}, the desired
computable lower bound for the minimizer in \eqref{eq:opt-c-eps} is
\begin{equation}
\label{eq:appC-eta-lower}
  \eta
  \ge
  \underline\eta_h
  :=
  \epsilon
  \left[
  \frac{\nu_{1,h}^{\sigma_*}}
       {(1+A_h^*)\bigl(1+\widetilde C_h^{\,2}\nu_{1,h}^{\sigma_*}\bigr)}
  -\sigma_*
  \right].
\end{equation}
Consequently,
\begin{equation}
\label{eq:appC-Ceps-upper}
  C_\epsilon^{\rm opt}
  =
  -\eta
  \le
  -\underline\eta_h,
\end{equation}
provided $\underline\eta_h<0$; otherwise the sharper statement is simply that
$C_\epsilon^{\rm opt}=0$.  In computation one therefore uses
\[
  C_{\epsilon,h}^{\rm cert}:=\max\{0,-\underline\eta_h\}
\]
as a rigorous, mesh-computable replacement for the crude constant obtained from
Hardy or Sobolev inequalities.

\bibliographystyle{siamplain}
\bibliography{library}

\end{document}